\documentclass[a4paper]{amsart}

\usepackage{amssymb,amsmath,mathrsfs,amscd,bbold,enumitem}
\usepackage{graphicx,xspace,hyperref}
\usepackage[usenames, dvipsnames]{xcolor} 
\usepackage{todonotes}  

\usepackage[showonlyrefs]{mathtools}

 \usepackage{tikz-cd}  
 \usetikzlibrary{matrix,arrows}
 \usetikzlibrary{decorations.pathreplacing,calc,positioning}

\newtheorem{theorem}{Theorem}[section]
\newtheorem{lemma}[theorem]{Lemma}
\newtheorem{proposition}[theorem]{Proposition}
\newtheorem{corollary}[theorem]{Corollary}
\newtheorem{Mthm}{Theorem}

\theoremstyle{definition}
\newtheorem{definition}[theorem]{Definition}

\newtheorem{AN}[theorem]{Assumption-Notation}

\newtheorem{defass}[theorem]{Definition-assumption}
\newtheorem{examplex}[theorem]{Example}

\newenvironment{example}
  {\pushQED{\qed}\examplex}
  {\popQED\endexamplex}
\theoremstyle{remark}
\newtheorem{remark}[theorem]{Remark}
\newcommand{\setmapG}[2]{\underline{w}_{#1,#2}}
\newcommand{\setmap}[2]{w_{#1,#2}}
\newcommand{\setmapS}[2]{\widehat{w}_{#1,#2}}
\newcommand{\setmapSG}[2]{\underline{\widehat{w}}_{#1,#2}}
\def\mol{\mathfrak m}
\def\nm{n_{\mol}}
\def\mL{\mol_{L}}

\def\Ac{A^c}

\def\stab{\operatorname{stab}}
\def\rk{\operatorname{rk}}
\def\rkE{\underline{\operatorname{rk}}}

\def\TT{\mathcal T}

\def\LL{\mathcal L}
\def\SS{\mathcal S}
\def\CC{\mathcal C}
\def\BB{\mathcal B}

\def\AA{\mathscr A}
\def\im{\operatorname{im}}
\def\id{\textup{id}}
\def\scr{\mathscr}

\def\PP{\mathcal P}
\def\kg{\kappa_\GS}

\def\GS{\mathfrak S}

\def\mm{m_\GS}
\def\PS{\mathcal P_\GS}

\def\ES{E_{\GS}}
\def\CS{\unten{\CC}}
\newcommand{\CGS}{{\CC_{\GS}}}
\newcommand{\unten}[1]{\underline{#1}}

\def\nn{\mathbb{N}}
\def\kk{\mathbb{K}}

\def\oo{\mathcal O}

\def\rc{\rk_\mathcal{C}}
\def\rl{\rk_\LL}
\def\lc{\LL(\mathcal{S})}
\def\lcl{\LL(\mathcal{S}_\LL)}
\def\scr{(S,\CC,\rc)}
\def\CCl{\mathcal{C}_\LL}
\def\rcl{\rk_{\CCl}}
\def\ssq{\subseteq}

\newcommand{\oC}[1]{\lceil #1 \rceil^\CC}

\DeclareMathOperator{\cl}{cl}

\usepackage[showonlyrefs]{mathtools}

\usepackage{multicol}

\DeclareFontFamily{U}{matha}{\hyphenchar\font45}
\DeclareFontShape{U}{matha}{m}{n}{
      <5> <6> <7> <8> <9> <10> gen * mathb
      <10.95> mathb10 <12> <14.4> <17.28> <20.74> <24.88> mathb12
      }{}
\DeclareSymbolFont{matha}{U}{matha}{m}{n}

\DeclareMathSymbol{\precneq}{3}{matha}{"AC}

\def\ST{translative\xspace}
\def\STy{translativity\xspace}
\def\WT{weakly translative\xspace}
\def\WTy{weak translativity\xspace}

\begin{document}

\title{Group actions on semimatroids}


\author{Emanuele Delucchi}
\address[E.~Delucchi]{Departement de Math\'ematiques, universit\'e de Fribourg, chemin du mus\'ee 23, 1700 Fribourg (Switzerland).}
\email{emanuele.delucchi@unifr.ch}
\author{Sonja Riedel}
\address[S.~Riedel]{Departement de Math\'ematiques, universit\'e de Fribourg, chemin du mus\'ee 23, 1700 Fribourg (Switzerland) and 
Institute for algebra, geometry, topology and their applications, Fachbereich Mathematik und Informatik, Universit\"at Bremen, Bibliothekstra\ss e 1,28359 Bremen (Germany).}
\email{sriedel@math.uni-bremen.de}

\maketitle

\begin{abstract} We initiate the
 study of group actions on (possibly infinite) semimatroids and geometric semilattices. To every such action is naturally associated an orbit-counting function, a two-variable "Tutte" polynomial and a poset which, in the realizable case, coincides with the poset of connected components of intersections of the associated toric arrangement.

  In this structural framework we recover and strongly generalize many
  enumerative results about arithmetic matroids, arithmetic Tutte
  polynomials and toric arrangements by finding new combinatorial
  interpretations beyond the realizable case. In particular, we thus
  find  
  the first class of natural examples of nonrealizable arithmetic matroids. Moreover, under additional conditions 
these actions give rise to a matroid over $\mathbb
  Z$. As a stepping stone toward our results we also prove an extension of the cryptomorphism between semimatroids and geometric semilattices to the infinite case.
\end{abstract}

%%%%%%%%%%%%%%%Table of contents%%%%%%%%%%%%%
\tableofcontents

\vspace{-15pt}

\section*{Introduction}
This paper is about group actions on combinatorial structures. 
There is an extensive literature on enumerative aspects of group
actions, from P\'olya's classical work \cite{Polya} 
to, e.g., recent results on polynomial invariants of actions on
graphs \cite{PJC}. The chapter on group actions in Stanley's book \cite{StAC} offers a survey of some of the results in this vein, together with a sizable literature list. 
Moreover, group actions on (finite) partially ordered sets have been
studied from the point of view of representation theory \cite{StaG}, of homotopy theory \cite{MSZ}, and of the poset's topology \cite{BK,ThWe}. 

Here we consider group actions on (possibly infinite) semimatroids and geometric semilattices from a structural perspective. 
We develop an abstract setting that fits different contexts arising in the
 literature, allowing us to unify and generalize many recent results.\medskip

\noindent{\bf Motivation.} Our original motivation came from the desire to better understand the different new combinatorial structures that have been introduced in the wake of recent work of De Concini--Procesi--Vergne \cite{dCP2,dCPV} on toric arrangements and partition functions, and have soon gained independent research interest. 
Our motivating goals are  
\begin{itemize}
\item[--] to organize these
different structures into a unifying theoretical framework, in
particular developing new combinatorial interpretations also in the nonrealizable case; 
\item[--] to understand the geometric side
of this theory, in particular in terms of an abstract class of posets (an 'arithmetic' analogue of geometric lattices).
\end{itemize}

To be more precise, let us consider a list $a_1,\ldots,a_n\in \mathbb
Z^d$ of integer vectors. Such a list gives rise to 
  an {\em arithmetic matroid} (d'Adderio--Moci \cite{dAM} and
  Br\"and\'en--Moci \cite{BM}) with an associated {\em arithmetic Tutte
    polynomial} \cite{Moc1}, and a {\em matroid over the ring $\mathbb
    Z$} (Fink--Moci \cite{FM}).
  Moreover, by interpreting the $a_i$ as characters of the
  torus $\operatorname{Hom}(\mathbb Z^d,\mathbb C^*)\simeq (\mathbb C^*)^d$ we obtain a {\em toric arrangement} in
  $(S^1)^d\subseteq (\mathbb C^*)^d$ defined by the kernels of the characters, 
  with an associated {\em poset of connected components} of intersections of these hypersurfaces.
  In this case, the arithmetic Tutte polynomial computes the
  characteristic polynomial of the arrangement's poset and the Poincar\'e
  polynomial of the arrangement's complement, as well as the Ehrhart
  polynomial of the zonotope spanned by the $a_i$ and the dimension
  of the associated Dahmen-Micchelli space \cite{Moc1}.  Other contexts of application of arithmetic matroids include the theory of spanning trees of simplicial complexes \cite{MaKl} and interpretations in graph theory \cite{DMG}.
After a first version of this paper was submitted, we learned about current work of Aguiar and Chan \cite{MAG} focussing on toric arrangements defined by graphs. Although they stay in the ``realizable'' realm, their interesting work refines some statistics related to arithmetic matroids and fits well into our setup.

  On an abstract level, arithmetic matroids offer an abstract theory
  supporting some notable properties of the arithmetic Tutte
  polynomial, while matroids over rings are a very general and
  strongly algebraic theory with different applications for suitable
  choices of the ``base ring'' (e.g., to
  tropical geometry for matroids over discrete valuation
  rings). However, outside the case of lists of integer
  vectors in abelian groups, the arithmetic Tutte polynomial and
  arithmetic matroids have few combinatorial interpretations.
For instance, the poset of connected components of intersections of a
toric arrangement -- which provides combinatorial interpretations for
many an evaluation of arithmetic Tutte polynomials -- has no
counterpart in the case of nonrealizable arithmetic matroids. 
Moreover, from a structural
  point of view it is striking (and unusual for matroidal objects)
  that there is no known cryptomorphism for arithmetic matroids, while
  for matroids over a ring a single one was recently presented
  \cite{FiPisa}. In addition, some conceptual relationships between arithmetic
  matroids (which come in different variants, see \cite{BM,dAM}) and
  matroids over rings are not yet cleared.

    In research unrelated to arithmetic matroids -- e.g.\ by Ehrenborg, Readdy and Slone \cite{ERS} and Lawrence \cite{Law} on enumeration on the torus, and by Kamiya, Takemura and Terao \cite{KTT,KTT2} on characteristic quasipolynomials of affine arrangements -- posets and `multiplicities' related to (but not satisfying the strict requirements of those arising with) arithmetic matroids were brought to light, calling for a systematic study of the abstract properties of ``periodic'' combinatorial structures. %Moreover,

Further motivation comes from recent progress in the study of complements of arrangements on
products of elliptic curves \cite{Bibby} which, combinatorially and
topologically, can be seen as quotients of ``doubly periodic''
subspace arrangements.\medskip
\noindent{\bf Results.} We initiate the study of actions of groups
by automorphisms on semimatroids (for short ``$G$-semimatroids''). 
Helpful intuition comes, once
again, from the case of integer vectors, where the associated toric
arrangement is covered naturally by a periodic affine hyperplane
arrangement: here semimatroids, introduced by Ardila \cite{Ard} (independently by Kawahara
\cite{Kawa}), enter the picture as abstract combinatorial descriptions
of finite arrangements of affine hyperplanes.  
In particular, we obtain the following results (see also 
Table  \ref{table} for a quick overview).
\begin{itemize}
\item[--] 
An equivalence (a.k.a.\ {\em cryptomorphism}) between
  $G$-semimatroids, which are defined in terms of certain set systems, and group actions on geometric semilattices (in the  sense of Wachs and Walker \cite{WW}), based on a theorem extending
  Ardila's equivalence between semimatroids and geometric semilattices
  to the infinite case (Theorem \ref{thm:fsl}).
\item[--] Under appropriate conditions every $G$-semimatroid gives
  rise to an ``underlying'' finite (poly)matroid (Theorem \ref{thm:polymat}).
Additional conditions can be imposed so that orbit
  enumeration determines an arithmetic matroid, often nonrealizable. In fact, we see that the
  defining properties of arithmetic matroids arise in a natural
  `hierarchy' with stronger conditions on the action (Theorem \ref{thm:almost}
  and Theorem \ref{thm:arithm}).
\item[--] In particular, we obtain the first natural class of examples of nonrealizable
  arithmetic matroids.
\item[--] To every $G$-semimatroid is naturally associated a poset
  $\mathcal P$
  obtained as a quotient of the geometric semilattice of the
  semimatroid acted upon. In particular, this gives a natural
  abstract generalization of the poset of connected components of
  intersections of a toric arrangement.
\item[--] To every $G$-semimatroid is associated a two-variable polynomial
  which evaluates as the characteristic polynomial of $\mathcal P$
  (Theorem \ref{thm:CP})
  and, under mild conditions on the action, 
satisfies a natural deletion-contraction recursion (Theorem \ref{thm:MainCD}) and a generalization
  of Crapo's basis-activity decomposition (Theorem \ref{thm:craponew}). In particular, for every
  arithmetic matroid arising from group actions we have a new
  combinatorial interpretation of the coefficients of the arithmetic
  Tutte polynomial in terms of enumeration on $\mathcal P$ subsuming
  Br\"and\'en and Moci's interpretation \cite[Theorem 6.3]{BM} in the realizable case.
\item[--] A $G$-semimatroid satisfying appropriate algebraic properties gives rise to a matroid over $\mathbb Z$, and we discuss conditions under which the single modules have combinatorial interpretations  (Theorem \ref{thm:RAM}).
\end{itemize}

\noindent{\bf Structure of the paper.} First, in Section \ref{sec:AMMR}
we recall the definitions of semimatroids, arithmetic matroids and
matroids over a ring. Then we devote Section \ref{sec:MaEx} to explaining our
guiding example, namely the ``realizable'' case of a $\mathbb Z^d$--action by translations on an affine hyperplane arrangement.
Then, Section \ref{sec:defs} gives a panoramic run-through of the main
definitions and results, in order to establish the ``Leitfaden'' of our work.
Before delving into the technicalities of the proofs, in Section
\ref{sec:exs} we will discuss some specific examples (mostly arising
from actions on arrangements of pseudolines) in order to
illustrate and distinguish the different concepts we introduce. 
Then we will move towards proving the announced results. First,
in Section \ref{sec:FSGS} we prove the cryptomorphism between finitary
semimatroids and finitary geometric semilattices. Section
\ref{sec:orb} is devoted to the construction of the underlying
(poly)matroid and semimatroid of an action.  
Then, in Section \ref{sec:AAA} we will focus on {\em
  translative} actions (Definition 
\ref{def:prinzipal}), for which the orbit-counting function gives rise to a {\em
  pseudo-arithmetic semimatroid} over the action's underlying
semimatroid. Subsequently, in Section \ref{sec:almost}, we will
further (but mildly) restrict to {\em almost-arithmetic} actions, and recover
``most of'' the properties required in the definition of arithmetic
matroids. In Section \ref{sec:arithm} we will then discuss the much
more restrictive condition on the action which ensures that our
orbit-count function fully satisfies the definition of an arithmetic
matroid and, for actions of abelian groups, we will 
derive a characterization of realizable matroids over $\mathbb Z$.
The closing Section \ref{sec:tutte} is devoted to the study
of certain ``Tutte'' polynomials associated to $G$-semimatroids. \medskip

\noindent {\bf Acknowledgements.} We thank Alex Fink for multiple discussions at different stages of our work, Kolja Knauer, Joseph Kung, Matthias Lenz and an anonymous referee for useful feedback on the first versions of this paper, as well as Katharina Jochemko for stimulating discussions on integer-point enumeration.  Both authors have been partially
supported by  Swiss National Science Foundation Professorship grant PP00P2\_150552/1. Sonja Riedel
has also benefited from  support of the ``Studienstiftung des Deutschen Volkes''.

\addtocontents{toc}{\protect\setcounter{tocdepth}{2}}

\section{The main characters}\label{sec:AMMR}

We start by introducing some basic definitions and terminology, sometimes modified with respect to the standard literature in order to better fit our setting. The reader may, in a first lecture, skip the technical details; however, a quick look at the main examples we offer in this section might be illuminating and help the intuition later on.

\subsection{Finitary semimatroids}\label{ssec:FSM}

We start by recalling the definition of a semimatroid,
which we state without finiteness assumptions on the ground set. This
relaxation substantially impacts the theory developed by Ardila \cite{Ard}, much of which rests on
the fact that any finite semimatroid can be viewed as a certain substructure of a (finite) 
`ambient' matroid. Here we list the definition and some immediate observations, while Section \ref{sec:FSGS} will be devoted to
proving the cryptomorphism with geometric semilattices.  
We note that equivalent structures were
also introduced by Kawahara \cite{Kawa} under the name
quasi-matroids, with a view on studying the associated
Orlik-Solomon algebra.

The motivation for introducing these structures was, in both \cite{Ard} and \cite{Kawa}, the combinatorial study of
affine hyperplane arrangements. In particular, keeping an eye on Example \ref{ex:ur} below will help make the following definition plausible. For a pictorial representation of an instance of this definition that does not arise from hyperplane arrangements we point to Example \ref{ex:running1}, which we will also keep as a running example throughout the paper.

\begin{definition}[Compare {\cite[Definition 2.1]{Ard}}]\label{def:FS} A \textit{finitary semimatroid} is a
  triple $\SS=(S,\CC,\rc)$ consisting of a (possibly infinite) set
  $S$, a non-empty finite dimensional simplicial complex \thinspace$\CC$ on $S$ and a bounded function $\rc:\CC\rightarrow\nn$ satisfying the following conditions.
\begin{itemize}
  \item[(R1)] If $X\in\CC,$ then $0\leq \rc(X)\leq |X|.$
  \item[(R2)] If $X,Y\in\CC$ and $X\ssq Y,$ then $\rc(X)\leq\rc(Y).$
  \item[(R3)] If $X,Y\in\CC$ and $X\cup Y\in\CC,$ then $\rc(X)+\rc(Y)\geq\rc(X\cup Y)+\rc(X\cap Y).$
  \item[(CR1)] If $X,Y\in\CC$ and $\rc(X)=\rc(X\cap Y),$ then $X\cup Y\in\CC.$
  \item[(CR2)] If $X,Y\in\CC$ and $\rc(X)<\rc(Y),$ then $X\cup y\in\CC$ for some $y\in Y-X.$
\end{itemize}
\label{def:lrt}
If only (R1), (R2), (R3) are known to hold, we call $\SS$ a {\em locally ranked triple}. 

A {\em finite} semimatroid is a finitary semimatroid with a finite ground set. Finiteness of locally ranked triples is defined accordingly.

Here, and in the following, we often write $\rk$ instead of $\rc$ and omit braces when representing singleton sets, thus writing $\rk(x)$ for $\rk(\{x\})$ and $X\cup x$ for $X\cup \{x\}$, when no confusion can occur.\end{definition}
 
We call $S$ the \textit{ground set}, $\CC$ the \textit{collection of central sets} and $\rk$ the \textit{rank function} of the finitary semimatroid
$\SS=(S,\CC,\rk),$ respectively. The {\em rank} of the semimatroid is the maximum
value of $\rk$ on $\CC$ and we will denote it by $\rk(\SS)$. A set $X\in \CC$ is called {\em independent} if
$\vert X \vert = \rk(X)$. A {\em basis} of $\SS$ is an
inclusion-maximal independent set. 

\begin{remark}
  We adopt the convention that every $x\in S$ is a vertex of $\mathcal
  C$, i.e., $\{x\}\in \mathcal C$ for all $x\in S$. Although this is
  not required in \cite{Ard}, it will not affect our considerations
  while simplifying the formalism. See also Remark \ref{rem:vertex}.
\end{remark}

\begin{definition}\label{df:simple}
A finitary semimatroid $\SS=(S,\CC,\rk)$ is \textit{simple} if 
$\rk(x)=1$ for all $x\in S$ and  $\rk(x,y)=2$ for all $\{x,y\}\in\CC$ with $x\neq y.$

A {\em loop} of a locally ranked triple $\SS=(S,\CC,\rk)$ is any $s\in
S$ with $\rk(s)=0$. Two elements $s,t\in S$ that are not loops are called parallel if $\{s,t\}\in \underline{\CC}$ and $\rk(\{s,t\})= 1$. The triple $\SS$ is called {\em simple} if it has no loops and no parallel elements.
 An {\em isthmus} of $\SS$ is any $s\in S$ such that, for every  $X\in \CC$, $X\cup s \in \CC$ and $\rk(X\cup s)=\rk(X)+1$.
\end{definition}

\begin{remark}\label{rem:dualmat}
  A {\em matroid} is, by definition, a finite semimatroid where every subset is
  central. Equivalently (and more classically), a matroid is given by
  a finite ground set $S$ and a rank function $\rk:2^S \to \mathbb N$
  satisfying (R1), (R2), (R3). The {\em dual} to a matroid $(S,\rk)$ is $(S,\rk^*)$, where $\rk^*(X):= \rk(S\setminus X) - \vert X \vert -\rk(S)$ for all $X\subseteq S$.
\end{remark}

\begin{remark}\label{def:polymat}
A {\em polymatroid} is given by a finite ground set $S$ and a rank function $\rk:2^S \to \mathbb N$ satisfying (R2), (R3) and $\rk(\emptyset)=0$. Polymatroids will appear furtively but naturally in our considerations; we refer e.g.\ to \cite[\S 18.2]{Welsh} for background on these structures.
\end{remark}

\begin{example}[The representable case, see Proposition 2.2 in \cite{Ard}]\label{ex:ur}
Given a positive integer $d$ and a field $\kk$, an \textit{affine
  hyperplane} is an affine subspace of dimension $d-1$ in the vector
space $\kk^d.$ An \textit{arrangement of hyperplanes} in $\kk^d$ is a
collection $\AA$ of affine hyperplanes in $\kk^d$. The arrangement is
called \textit{locally finite} if every point in $\kk^d$ has a
neighbourhood that intersects only finitely many hyperplanes of $\AA$.
A subset $X\subseteq\AA$ is \textit{central} if 
$\cap X \neq \emptyset$.
Let $\CC_\AA$ denote the set of
central subsets of $\AA$ and define the rank function
$\rk_\AA:\CC_\AA\rightarrow\nn$ as $\rk_\AA(X)=d-\dim\cap X.$

Then, the triple $(\AA,\CC_\AA,\rk_\AA)$ is a finitary semimatroid. It is simple if all elements of $\AA$ are distinct, and it is a matroid if all elements of $\AA$ are linear subspaces (i.e.\ contain the origin of $\mathbb K^d$).
\end{example}

\begin{example}[Pseudoline arrangements] 
\label{ex:running1}
There are cases of nonrepresentable semimatroids in which we can still take advantage of a pictorial illustration --- one such instance is given by {\em arrangements of pseudolines} in the sense of Gr\"unbaum \cite{Gru}, i.e., sets of homeomorphic images of $\mathbb R$ in $\mathbb R^2$ (``pseudolines'') such that 
\begin{itemize}\item[(1)]
 every point of $\mathbb R^2$ has a neighborhood intersecting only finitely many pseudolines, 
 \item[(2)] any two pseudolines in the set intersect at most in one point (and if they intersect, they do so transversally). 
\end{itemize}

Figure \ref{fig:PseudoFirst} shows such an arrangement of pseudolines.
The definitions of Example \ref{ex:ur} can be carried over to this context. The triple associated $(S,\CC,\rk)$ associated to the pseudoline arrangement is given by 
$$S=\{a_i \mid i\in \mathbb Z\}\cup 
\{b_i \mid i\in \mathbb Z\}\cup
\{c_i \mid i\in \mathbb Z\}\cup
\{d_i \mid i\in \mathbb Z\}\cup
\{e_i \mid i\in \mathbb Z\},$$
\begin{align*}
\CC = &\{\emptyset\} \cup \{a_i\}_{i}\cup
\{b_i\}_{i}\cup
\{c_i\}_{i}\cup
\{d_i\}_{i}\cup 
\{e_i\}_{i}\cup 
\{a_i,b_j\}_{i,j}\cup
\{a_i,c_j\}_{i,j}\\
&\cup\{a_i,d_j\}_{i,j}\cup
\{a_i,e_j\}_{i,j}\cup
\{b_i,c_j\}_{i,j}\cup
\{b_i,d_j\}_{i,j}\cup
\{b_i,e_j\}_{i,j}\cup
\{c_i,d_j\}_{i,j}\\
&\cup\{d_i,e_j\}_{i,j}\cup
 \{a_{2i+k},b_{2i-k},c_k\}_{i,k}\cup
 \{a_{2i+k},b_{2i-k},d_k\}_{i,k}\cup
 \{a_k,b_{k-2i-1},e_i\}_{i,k} \\ 
 &\cup\{a_{2i+k},c_{k},d_{i}\}_{i,k}\cup\{b_{2i-k},c_{k},d_{i}\}_{i,k}
 \cup\{a_{2i+k},b_{2i-k},c_{k},d_{i}\}_{i,k},
\end{align*}
\begin{align*}
\rk(X) =& \operatorname{codim} (\cap X)\textrm{ for all }X\in \CC
\end{align*}
and one easily checks that this defines a finitary semimatroid.

For readability's sake, here and in all following examples we omit to specify that all indices run over $\mathbb Z$ and that the union is taken over sets of sets, thus using the shorthand notation $\{a_i,b_j\}_{i,j}$ for $\{\{a_i,b_j\} \mid i,j\in \mathbb Z\}$.

 Notice that this triple cannot be obtained from an arrangement of straight lines: such an arrangement is called {\em non-stretchable}. 
 \end{example}

\begin{figure}[h]
\scalebox{.7}{%
\begin{tikzpicture}
    \node at (0,0) (O) {};
    \node at (4,0) (E) {};
    \node at (4,4) (NE) {};
    \node at (0,4) (N) {};
    \node at (-4,0) (W) {};
    \node at (-4,-4) (SW) {};
    \node at (0,-4) (S) {};
    \node at (4,-4) (SE) {};
    \node at (-4,4) (NW) {};
    \node at (2.5,0) (Er) {};
    \node at (1.5,0) (El) {};
    \node at (-2.5,0) (Wl) {};
    \node at (-1.5,0) (Wr) {};    
    \node at (-2.5,4) (NWl) {};
    \node at (-1.5,4) (NWr) {};    
    \node at (-2.5,-4) (SWl) {};
    \node at (-1.5,-4) (SWr) {};    
    \node at (2.5,4) (NEr) {};
    \node at (1.5,4) (NEl) {};
    \node at (2.5,-4) (SEr) {};
    \node at (1.5,-4) (SEl) {};
    %%%%%%%%%%%%
    \draw [->] (-4.7,-4) -- (-4.7,0);
    \draw [->] (-4,-4.7) -- (0,-4.7);
    \node at (-5,-2) (e2) {$\epsilon_2$};
    \node at (-2,-5) (e1) {$\epsilon_1$};
    %%%%%%%%%%%%
    \draw [-,red] (W.center) -- (E.center);
    \draw [-,red] (SW.center) -- (SE.center);
    \draw [-,red] (NW.center) -- (NE.center);    
    \draw [-,green] (SW.center) -- (Wl.center);
    \draw [-,green] (Wl.center) -- (N.center);
    \draw [-,green] (SWl.center) -- (O.center) -- (NEl.center);
    \draw [-,green] (S.center) -- (El.center) -- (NE.center);
    \draw [-,green] (SW.center) -- (Wl.center) -- (N.center);
    \draw [-,green] (SEl.center) -- (E.center);    
    \draw [-,green] (W.center) -- (NWl.center);    
    \draw [-,blue] (N.center) -- (S.center);
    \draw [-,blue] (NE.center) -- (SE.center);
    \draw [-,blue] (NW.center) -- (SW.center);                
    \draw [-,orange] (NW.center) -- (Wr.center) -- (S.center);
    \draw [-,orange] (NWr.center) -- (O.center) -- (SEr.center);
    \draw [-,orange] (N.center) -- (Er.center) -- (SE.center);
    \draw [-,orange] (NEr.center) -- (E.center);
    \draw [-,orange] (W.center) -- (SWr.center);
    \draw [dashed,purple] (-4,-1.5) -- (4,-1.5);
    \draw [dashed,purple] (-4,2.5) -- (4,2.5);    
    \node[text=red] at (-3.3,-4.2) (c0) {$c_0$}; 
    \node[text=red] at (-3.3,-.2) (c1) {$c_1$}; 
    \node[text=red] at (-3.3,3.8) (c2) {$c_2$}; 
    \node[text=blue] at (-4.2,-1.2) (d0) {$d_0$};     
    \node[text=blue] at (-.2,-1.2) (d1) {$d_1$}; 
    \node[text=blue] at (3.8,-1.2) (d2) {$d_2$}; 
    \node[text=green] at (-1.3,1.3) (b0) {$b_0$}; 
    \node[text=green] at (-3.3,.7) (bm) {$b_{-1}$}; 
    \node[text=green] at (.8,1.3) (b1) {$b_1$}; 
    \node[text=green] at (2.7,1.3) (b2) {$b_2$}; 
    \node[text=green] at (2.7,-2.7) (b3) {$b_3$}; 
    \node[text=orange] at (-2.6,-2.6) (c0) {$a_1$}; 
    \node[text=orange] at (-.8,-2.6) (c0) {$a_2$}; 
    \node[text=orange] at (1.4,-2.6) (c0) {$a_3$}; 
    \node[text=orange] at (2.6,-1) (c0) {$a_4$}; 
    \node[text=orange] at (3.5,.5) (c0) {$a_5$};
    \node[text=purple] at (-2,2.7) (e1) {$e_1$};
    \node[text=purple] at (-2,-1.3) (e2) {$e_0$};     
\end{tikzpicture}
}
\caption{A non-stretchable pseudoline arrangement (it should be thought of as repeating and tiling the plane). 
}\label{fig:PseudoFirst}
\end{figure}
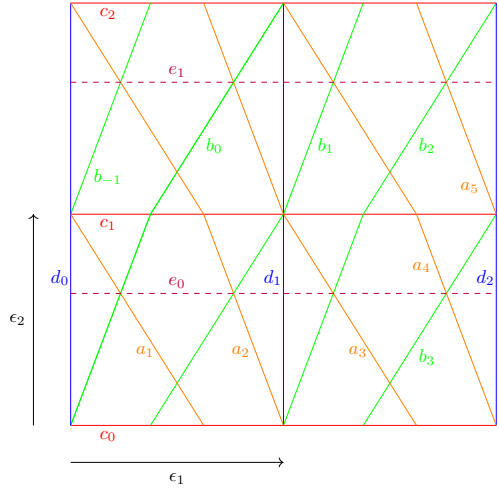

We now state some basic facts and definitions about semimatroids for later
reference. Except where otherwise specified, the proofs are parallel to those given in
\cite[Section 2]{Ard}.

%    \begin{remark}\label{CiErPr}
%A finitary semimatroid satisfies a 'local' version of %(R1) and
% (R2) and a stronger version of (CR1) and (CR2), as well.
%\begin{itemize}
%  \item[(R2')] If $X\cup x\in\CC$ then $\rk(X\cup x)-\rk(X)$ equals 0 or 1.
%  \item[(CR1')] If $X,Y\in\CC$ and $\rk(X)=\rk(X\cap Y),$ then $X\cup Y\in\CC$ and $\rk(X\cup Y)=\rk(Y).$
%  \item[(CR2')] If $X,Y\in\CC$ and $\rk(X)<\rk(Y),$ then $X\cup y\in\CC$ and $\rk(X\cup y)=\rk(X)+1$ for some $y\in Y-X.$
%\end{itemize}
%       \end{remark}

\begin{definition}\label{def:latticeofflats} Let $\SS=(S,\CC,\rk)$ be a finitary semimatroid and
  $X\in\CC$. The \textit{closure of $X$ in $\CC$} is 
\begin{align*}
  \cl(X) :=\{x\in S\mid X\cup x\in\CC,\rk(X\cup x)=\rk(X)\}. 
\end{align*}

 A \textit{flat} of a finitary semimatroid $\SS$ is a set $X\in\CC$
 such that $\cl(X)=X.$ The set of flats of $\SS$ ordered by containment forms the \textit{poset of flats of} $\SS$, which we denote by $\lc.$ \label{defi:lc}
\end{definition}

\begin{remark}\label{rem:monotone}
For all $X\in \CC$ we have $\cl(X) = \max \{Y\supseteq X
  \mid X\in \CC,\, \rk(X) = \rk(Y)\}$, i.e., the closure of $X$ is the
  maximal central set containing $X$ and having same rank as $X$. In
  particular, we have a monotone function $\cl: \CC \to \CC$.
\end{remark}

\begin{remark}\label{rem:GL}
A poset is the poset of flats of a matroid if and only if it is a
  geometric lattice (see Definition \ref{df:GL}). 
  %or \cite[Section 3.3]{Welsh}). 
  In Section \ref{sec:FSGS} we will prove a similar correspondence between finitary semimatroids and geometric
  semilattices (Theorem \ref{thm:fsl}).
\end{remark}

We now introduce the notions of deletion and contraction for locally ranked triples. Example \ref{ex:code} below will illustrate the case of pseudoline arrangements.

\begin{definition}\label{def:DeRedef}
  Let $\SS=(S,\CC,\rk)$ be a locally ranked triple. For every
  $T\subseteq S$ let $\CC_{\setminus T} := \CC\cap 2^{S\setminus T}$
  and define the {\em deletion} of $T$ from $\SS$ as
$$
\SS \setminus T :=(S\setminus T, \CC_{\setminus T}, \rk),
$$
where we slightly abuse notation and write $\rk$ for
$\rk\vert_{\CC_{\setminus T}}$.
Moreover, we will denote by $\SS[T]:=\SS\setminus (S\setminus T)$ the {\em restriction} to $T$.

Furthermore, for every central set $X\in \CC$ let 
$$\CC_{/ X}:=\{Y\in \CC_{\setminus
  X} \mid Y\cup X \in \CC\},\quad\quad S_{/X}:=\{s\in S \mid
\{s\}\in\CC_{/X}\}$$ and define the {\em contraction} of $X$ in $\SS$ as 
$$
\SS / X :=(S_{/X}, \CC_{/X}, \rk_{/X}),
$$
where, for every $Y\in \CC_{/X}$, $\rk_{/X}(Y):=\rc(Y\cup X)-\rc(X)$.
\end{definition}

\begin{remark}\label{rem:vertex}
This definition applies in particular to the case where $\SS$ is a semimatroid and, in this case, differs slightly from that given in \cite{Ard}: since we assume every element of the ground set of a semimatroid to be contained in a central set, we need to further constrain the ground set of the contraction.
\end{remark}

\begin{example}\label{ex:code}
Let $\SS=(S,\CC,\rk)$ be  the semimatroid of Example \ref{ex:running1} (see Figure \ref{fig:PseudoFirst}).
If $T:=\{e_i\}_{i\in \mathbb Z}$, then $$\CC_{\setminus T}=\CC\setminus 
(\{e_i\}_i\cup \{a_i,e_j\}_{i,j}\cup \{b_i,e_j\}_{i,j}\cup \{d_i,e_j\}_{i,j}\cup \{a_k,b_{k-2i-1},e_i\}_{i,k}),$$ and  $\SS \setminus T$ is the semimatroid associated to the arrangement on the left-hand side in Figure \ref{fig:DeCo}.

The contraction of $\SS$ to $e_0\in S$ has ground set
$S_{/\{e_0\}} = S\setminus (\{c_i\}_{i\in \mathbb Z}\cup \{e_i\}_{i\in \mathbb Z})$ and family of central sets
$\CC_{/\{e\}} = \{\emptyset\} \cup 
\{a_i\}_{i}\cup
\{b_i\}_{i}\cup
\{d_i\}_{i}\cup 
\{a_i,b_{i-1}\}_{i}$ with rank function $\rk_{/\{e_0\}}$ given by
\begin{align*}
&\rk_{/\{e_0\}}(\emptyset) = \rk(\{e_0\}) - \rk(\{e_0\})=0;\\ 
&\rk_{/\{e_0\}}(\{a_i\}) = \rk(\{a_i,e_0\}) - \rk(\{e_0\}) =1, \\
&\textrm{similarly }\rk_{/\{e_0\}}(\{b_i\})=\rk_{/\{e_0\}}(\{d_i\}) = 1; \\
&\rk_{/\{e_0\}}(\{a_i,b_{i-1}\}) = \rk(\{a_i,b_{i-1},e_0\}) -\rk(\{e_0\}) = 1;
\end{align*}
where $i$ ranges over the integers. This triple is represented by the arrangement of points depicted on  the right-hand side in Figure \ref{fig:DeCo}.
\end{example}

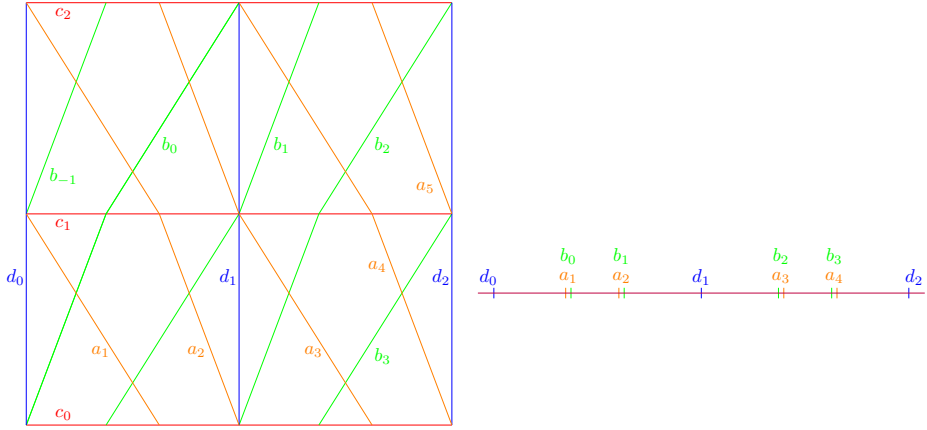
\begin{figure}
\scalebox{.7}{%
\begin{tikzpicture}
    \node at (0,0) (O) {};
    \node at (4,0) (E) {};
    \node at (4,4) (NE) {};
    \node at (0,4) (N) {};
    \node at (-4,0) (W) {};
    \node at (-4,-4) (SW) {};
    \node at (0,-4) (S) {};
    \node at (4,-4) (SE) {};
    \node at (-4,4) (NW) {};
    \node at (2.5,0) (Er) {};
    \node at (1.5,0) (El) {};
    \node at (-2.5,0) (Wl) {};
    \node at (-1.5,0) (Wr) {};    
    \node at (-2.5,4) (NWl) {};
    \node at (-1.5,4) (NWr) {};    
    \node at (-2.5,-4) (SWl) {};
    \node at (-1.5,-4) (SWr) {};    
    \node at (2.5,4) (NEr) {};
    \node at (1.5,4) (NEl) {};
    \node at (2.5,-4) (SEr) {};
    \node at (1.5,-4) (SEl) {};
    \draw [-,red] (W.center) -- (E.center);
    \draw [-,red] (SW.center) -- (SE.center);
    \draw [-,red] (NW.center) -- (NE.center);    
    \draw [-,green] (SW.center) -- (Wl.center);
    \draw [-,green] (Wl.center) -- (N.center);
    \draw [-,green] (SWl.center) -- (O.center) -- (NEl.center);
    \draw [-,green] (S.center) -- (El.center) -- (NE.center);
    \draw [-,green] (SW.center) -- (Wl.center) -- (N.center);
    \draw [-,green] (SEl.center) -- (E.center);    
    \draw [-,green] (W.center) -- (NWl.center);    
    \draw [-,blue] (N.center) -- (S.center);
    \draw [-,blue] (NE.center) -- (SE.center);
    \draw [-,blue] (NW.center) -- (SW.center);                
    \draw [-,orange] (NW.center) -- (Wr.center) -- (S.center);
    \draw [-,orange] (NWr.center) -- (O.center) -- (SEr.center);
    \draw [-,orange] (N.center) -- (Er.center) -- (SE.center);
    \draw [-,orange] (NEr.center) -- (E.center);
    \draw [-,orange] (W.center) -- (SWr.center);   
    \node[text=red] at (-3.3,-3.8) (c0) {$c_0$}; 
    \node[text=red] at (-3.3,-.2) (c1) {$c_1$}; 
    \node[text=red] at (-3.3,3.8) (c2) {$c_2$}; 
    \node[text=blue] at (-4.2,-1.2) (d0) {$d_0$};     
    \node[text=blue] at (-.2,-1.2) (d1) {$d_1$}; 
    \node[text=blue] at (3.8,-1.2) (d2) {$d_2$}; 
    \node[text=green] at (-1.3,1.3) (b0) {$b_0$}; 
    \node[text=green] at (-3.3,.7) (bm) {$b_{-1}$}; 
    \node[text=green] at (.8,1.3) (b1) {$b_1$}; 
    \node[text=green] at (2.7,1.3) (b2) {$b_2$}; 
    \node[text=green] at (2.7,-2.7) (b3) {$b_3$}; 
    \node[text=orange] at (-2.6,-2.6) (c0) {$a_1$}; 
    \node[text=orange] at (-.8,-2.6) (c0) {$a_2$}; 
    \node[text=orange] at (1.4,-2.6) (c0) {$a_3$}; 
    \node[text=orange] at (2.6,-1) (c0) {$a_4$}; 
    \node[text=orange] at (3.5,.5) (c0) {$a_5$};
\end{tikzpicture}
}
\scalebox{.7}{%
\begin{tikzpicture}
    \node at (0,5.5) (N) {};
    \node at (0,-2.5) (S) {};
    \draw [-,blue] (-3.9,-.1) -- (-3.9,.1);
    \draw [-,green] (-2.45,-.1) -- (-2.45,.1);
    \draw [-,orange] (-2.55,-.1) -- (-2.55,.1);    
    \draw [-,green] (-1.45,-.1) -- (-1.45,.1);
    \draw [-,orange] (-1.55,-.1) -- (-1.55,.1);    
    \draw [-,blue] (0,-.1) -- (0,.1);
    \draw [-,green] (1.45,-.1) -- (1.45,.1);
    \draw [-,orange] (1.55,-.1) -- (1.55,.1);        
    \draw [-,green] (2.45,-.1) -- (2.45,.1);
    \draw [-,orange] (2.55,-.1) -- (2.55,.1);     
    \draw [-,blue] (3.9,-.1) -- (3.9,.1);
    \draw [-,purple] (-4.2,0) -- (4.2,0);
    \node[text=blue] at (-4,.3) (d0) {$d_0$};     
    \node[text=blue] at (0,.3) (d1) {$d_1$}; 
    \node[text=blue] at (4,.3) (d2) {$d_2$}; 
    \node[text=green] at (-2.5,.7) (b0) {$b_0$};  
    \node[text=green] at (-1.5,.7) (b1) {$b_1$}; 
    \node[text=green] at (1.5,.7) (b2) {$b_2$}; 
    \node[text=green] at (2.5,.7) (b3) {$b_3$}; 
    \node[text=orange] at (-2.5,.3) (c0) {$a_1$}; 
    \node[text=orange] at (-1.5,.3) (c0) {$a_2$}; 
    \node[text=orange] at (1.5,.3) (c0) {$a_3$}; 
    \node[text=orange] at (2.5,.3) (c0) {$a_4$}; 
\end{tikzpicture}
}
\caption{Arrangements of pseudolines corresponding to the deletion $\SS\setminus\{e_i\}_{i}$ (l.h.s.), and  the contraction $\SS/\{e_0\}$ (r.h.s.), 
 where $\SS$ is the semimatroid of Example \ref{ex:running1}. Again, we show only local pieces of these infinite arrangements, and the pictures must be thought of as being repeated in order to fill the plane (resp. the line).}
\label{fig:DeCo}
\end{figure}

\begin{proposition}
 Let $\SS=(S,\rk,\CC)$ be a finitary semimatroid. For every $T\subset S$, $\SS\setminus T$ is a finitary semimatroid and, for every $X\in \CC$, $\SS/X$ is a finitary semimatroid.
\end{proposition}
\begin{proof}
The proof of {\cite[Proposition 7.5 and 7.7]{Ard}} adapts straightforwardly.
\end{proof}

\begin{definition}
To every finite locally ranked triple $\SS=(S,\CC,\rk)$  we associate the following polynomial.
$$
T_{\SS} (x,y) := \sum_{X\in \CC} (x-1)^{\rk(\SS)-\rk(X)} (y-1)^{\vert X
  \vert - \rk(X)}
$$
\end{definition}

\begin{remark}
If $\SS$ is a finite semimatroid, this is exactly the {\em Tutte polynomial} of $\SS$ introduced and studied by Ardila \cite{Ard}. In particular, if $\SS$ is a matroid, this is the associated Tutte polynomial. 
\end{remark}

A celebrated result about Tutte polynomials of matroids is the following ``activities decomposition theorem'' due to Crapo (for terminology we refer to \cite{Oxl}).
\begin{proposition}[{\cite[Theorem 1]{Crapo}}]\label{prop:classicCrapo} Let  $\SS$ be a matroid with set of
  bases $\BB$ and fix a total ordering $<$ on $S$. Then,
  \begin{equation*}
    \label{eq:6}
    T_{\SS}(x,y) = \sum_{B\in \BB} x^{\vert I(B)\vert} y ^{\vert E(B) \vert},
  \end{equation*}
  where, for every $B\in\BB$, 
  \begin{enumerate}
  \item[$I(B)$] is the set of {\em internally active elements} of $B$, i.e., the set of all $b\in B$ which are $<$-minimal in some codependent subset of $S\setminus (B \setminus b)$.  
  \item[$E(B)$] is the set of {\em externally active elements} of $B$, i.e., the set of all $e\in S\setminus B$ that are $<$-minimal in some dependent subset of $B\cup e$.
  \end{enumerate}

\end{proposition}

\begin{remark}
Arithmetic Tutte polynomials satisfy an analogue to Crapo's theorem for realizable arithmetic matroids (see Remark \ref{rem:lifo}). One of our results is the generalization of this theorem to all centered translative $G$-semimatroids (Theorem \ref{thm:craponew}).
\end{remark}

\subsection{Arithmetic (semi)matroids and their Tutte polynomials}\label{ssec:AM}
We extend the definition of arithmetic matroids given in \cite{BM} and \cite{dAM} to include the case where the underlying structure is a finite semimatroid.

\begin{definition}[Compare Section 2 of \cite{BM}]\label{def:rho} Let $\SS=(S,\CC,\rk)$ be a locally ranked triple.
  A {\em molecule} of $\SS$ is any triple $(R,F,T)$  of disjoint sets with $R\cup
  F\cup T\in \CC$ and such that, for every $A$ with $R\subseteq A
  \subseteq R\cup F\cup T$, 
$$\rk(A) = \rk(R) + \vert A\cap F\vert.$$
\end{definition}

%Here and in the following, given any two sets $X\subseteq Y$ we will denote by \linebreak$[X,Y]=\{A \subseteq Y \mid X\subseteq A\}$ the interval between $X$ and $Y$ in the boolean poset of subsets of $Y$.

\begin{remark}
Once a total ordering of the ground set $S$ is fixed, the notion of basis activities for matroids briefly recapped in Proposition \ref{prop:classicCrapo} above allows us to associate to every basis $B$ a molecule $(B\setminus I(B), I(B), E(B))$. 
\end{remark}

\begin{definition}[Extending Moci and Br\"and\'en \cite{BM}] \label{def:AM}
  Let $\SS=(S,\CC,\rk)$ be a finite locally ranked triple
  and $m:\CC \to \mathbb R$ any
  function. 
  If $(R,F,T)$ is a molecule, define
$$\rho(R, R\cup F\cup T):=(-1)^{\vert T\vert}\sum_{
R\subseteq A
  \subseteq R\cup F\cup T
%A \in [R, R\cup F\cup T]
}
(-1)^{\vert R\cup F \cup T \vert - \vert A \vert} m(A).$$
We call the pair $(\SS,m)$ {\em arithmetic} if
  the following axioms are satisfied:
  \begin{itemize}
  \item[(P)] For every molecule $(R,F,T)$,
$$
\rho(R,R\cup F \cup T) \geq 0.
$$
\item[(A1)] For all $A\subseteq S$ and $e\in S$ with $A\cup e \in \CC$:
  \begin{itemize}
  \item[(A.1.1)] If $\rk(A\cup \{e\})=\rk(A)$ then $m(A\cup \{e\})$
    divides $m(A)$. 
  \item[(A.1.2)] If $\rk(A\cup \{e\})>\rk(A)$ then $m(A)$
    divides $m(A\cup \{e\})$. 
  \end{itemize}
\item[(A2)] For every molecule $(R,F,T)$
  $$m(R)m(R\cup F\cup T)=m(R\cup F) m(R\cup T).$$
  \end{itemize}
Following \cite{BM} we use the expression {\em pseudo-arithmetic} to denote the case where $m$ only satisfies (P).
An {\em arithmetic matroid} is an arithmetic pair $(\SS,m)$
where $\SS$ is a matroid. 
\end{definition}
\begin{remark}
Following \cite{dAM}, the {\em dual} to an arithmetic matroid $(\SS,m)$ is the pair $(\SS^*,m^*)$, where $\SS^*$ is the dual matroid to $\SS$ and $m^*(A):=m(S\setminus A)$.
\end{remark}

\begin{example}\label{ex:exarmat}
  To every set of integer vectors, say $a_1,\ldots,a_n \in \mathbb
  Z^d$ is associated a matroid on the ground set $[n]:=\{1,\ldots,n\}$ with
  rank function $$\rk(I):=\dim_{\mathbb
    Q}(\operatorname{span}(a_i)_{i\in I}),$$ and a multiplicity
  function $m(I)$ defined for every $I\subseteq [n]$ as the greatest
  common divisor of the minors of the matrix with columns $(a_i)_{i\in
  I}$. These determine an arithmetic matroid \cite{dAM}. We say that
the vectors $a_i$ {\em realize} this arithmetic matroid which we call
then {\em realizable}.
\end{example}

To every arithmetic pair $(\SS,m)$ we associate an arithmetic Tutte
polynomial as a straightforward extension of Moci's definition from
\cite{Moc1}.

\begin{definition}\label{eq:5}\label{caramai} Given an arithmetic pair $(\SS,m)$, set
$$  T_{(\SS,m)}(x,y):=\sum_{X\in \CC} m (X) (x-1)^{\rk(\SS)-\rk(X)}(y-1)^{\vert X \vert - \rk(X)}.$$
\end{definition}

\begin{remark}\label{rem:lifo}
  When $(\SS,m)$ is an arithmetic matroid, the polynomial
  $T_{(\SS,m)}(x,y)$ enjoys a rich structure theory, investigated for
  instance in \cite{BM,dAM}. When this arithmetic matroid is
  realizable, say by a set of vectors $a_1,\ldots,a_n\in \mathbb Z^d$,
  the arithmetic Tutte polynomial specializes e.g.\ to the
  characteristic polynomial of the associated toric arrangement (see
  Section \ref{sec:MaEx}) and to the Ehrhart polynomial of the
  zonotope obtained as the Minkowski sum of the $a_i$. Moreover,
  always in the realizable case, Crapo's decomposition theorem
  (Proposition \ref{prop:classicCrapo}) has an analogue \cite[Theorem
  6.3]{BM} which gives a combinatorial interpretation of the
  coefficients of the polynomial in terms of counting integer points
  of zonotopes and intersections in the associated toric arrangement.
\end{remark}

\subsection{Matroids over rings}

We give the general definition and some properties of matroids over
rings. Further explanations and proofs of statements can be found in \cite{FM}.

\begin{definition}[Fink and Moci \cite{FM}]\label{def:mator}
  Let $E$ be a finite set, 
  $R$ a commutative ring and $M:2^E \to
  R\operatorname{-mod}$ any function associating an $R$-module to each
  subset of $E$. This defines a {\em matroid
    over $R$} if
  \begin{itemize}
  \item[(R)]  for any $A\subset E$, $e_1,e_2\in E$ ,
    there is a pushout square
$$
\begin{CD}
  M(A) @>>> M(A\cup\{e_1\}) \\
  @VVV @VVV\\
 M(A\cup\{e_2\}) @>>> M(A\cup \{e_1,e_2\})
\end{CD}
$$
such that all morphisms are surjections with cyclic kernel.
  \end{itemize}
\end{definition}

\begin{remark}[{\cite[Section 6.1]{FM}}]\label{rem:uam}
  Every matroid over the ring $R=\mathbb Z$ induces an arithmetic
  matroid on the ground set $E$ with  rank function  satisfying $\rk(E)-\rk(A)=\operatorname{rank}_{\mathbb Z} M(A)$
  %, the rank of $M(A)$ as a $\mathbb Z$-module, 
  and $m(A)$ equal to the cardinality of the torsion part of $M(A)$, for all $A\subseteq E$. We call $(E,\rk)$ the {\em underlying matroid} to $M_{\GS}$ and $(E,\rk,m)$ the {\em underlying semimatroid} to $M_{\GS}$.
\end{remark}

\begin{remark}[See Definition 2.2 in \cite{FM}]
  A matroid $M$ over a ring $R$ is called {\em realizable} if there is a
  finitely generated $R$-module $N$ and a list $(x_e)_{e\in E}$ of
  elements of $N$ such that for all $A\subseteq E$ we have that $M(A)$
  is isomorphic to the quotient $N/(\sum_{e\in A} R
  x_e)$. Realizability is preserved under duality.
\end{remark}

        \addtocontents{toc}{\protect\setcounter{tocdepth}{1}}
\section{Geometric intuition: Periodic arrangements}\label{sec:MaEx}

As an introductory example  we describe the
arithmetic matroid and the matroid over $\mathbb Z$ associated to periodic hyperplane arrangements, highlighting the structures we
will encounter in the general theory later. 

Let $\kk$ stand for either $\mathbb R$ or $\mathbb C$ and recall that an {\em affine hyperplane arrangement} is a locally finite set $\AA$ of
hyperplanes in $\kk^d$. It is called {\em periodic} if it is
(globally) invariant under the action of a group acting on $\kk^d$ by
translations. 

\def\eee{\varepsilon}
For simplicity, we will consider the standard action of
$\mathbb Z^d$ on $\mathbb K^d$, with $k\in \mathbb Z^d$ acting as $t_k(x) = x + \sum_i k_i \eee_i$, where $\eee_1,\ldots, \eee_d$ is the standard basis of $\kk^d$, and we will suppose the arrangement $\AA$ being given by a finite list of integer vectors $a_1,\ldots,a_n \in \mathbb Z^d$ (which we think of as the columns of a $d\times n$ matrix $A$) together with a corresponding list $\alpha_1,\ldots,\alpha_n\in \kk$ of real numbers as follows.

For $X\subseteq [n]$ let $A[X]$ be the $d\times \vert X\vert$ matrix
obtained by restricting $A$ to the relevant columns. Moreover, given and $k\in \mathbb Z^X$ we define the subspace
  \begin{align*}
    H(X,k) &:=\{x\in \kk^d \mid
    \forall i\in X : 
    a_i^{T}x=
    \alpha_i + k_i \}.
  \end{align*}
Then, 
$$
\AA=\{H(\{i\},j) \mid i\in [n], j\in \mathbb Z\}.
$$

\begin{example}\label{ex:rev2}
The periodic arrangement given by 
$$
A:=\left(\begin{array}{ccc}
1 & 1 & 1 \\ 1 & -1 & 0
\end{array}\right)
\quad\quad
\alpha_1=\alpha_2=\alpha_3=0
$$
Is the set
\begin{align*}
\AA = & \quad\,\{{\color {red} H(\{1\},j)} = \{x\in \mathbb R^2 \mid x_1+x_2=j\} \mid j\in \mathbb Z\}\\
& \cup \{{\color {blue} H(\{2\},j)} = \{x\in \mathbb R^2 \mid x_1-x_2=j\} \mid j\in \mathbb Z \} \\
& \cup \{{\color {green} H(\{3\},j)} = \{x\in \mathbb R^2 \mid x_1=j\} \mid j\in \mathbb Z \}
\end{align*}

\begin{figure}[h]
\begin{tikzpicture}[every node/.style={font=\small}]
\node at (0,0) (pic) {\includegraphics[scale=.65]{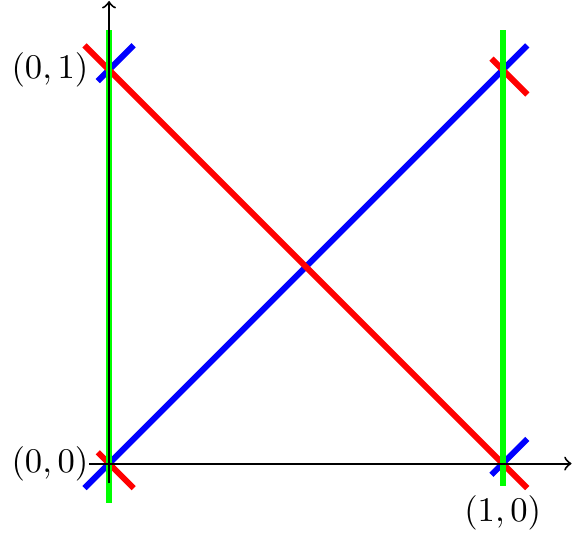}};
\node at (0.2,0) (z) {};
\node at (-1.05,-1.2) (o) {};
\node at (-1.15,-1.2) (o1) {};
\node [anchor = west] at (2.3,1.1) (121) {$(\frac{1}{2},\frac{1}{2})=H(\{1,2\}, (1,0))$};
\draw [<-,dashed] (z) -- (121.west);
\filldraw (0.13,0.03) circle (.2em);
%\node [anchor = west] at (2.3,0.5) (120) {$(0,0)=H(\{1,2\}, (0,0))=H(\{1,2,3\}, (0,0,0))$};
\node [anchor = west] at (-.7,-2.2) (120) {$(0,0)=H(\{1,2\}, (0,0))=H(\{1,2,3\}, (0,0,0))$};
\filldraw (-1.17,-1.3) circle (.2em);
\draw [<-,dashed] (-1.13,-1.4) -- (-.6,-2.2);
\node [anchor = west] at (-3.5,-2) (30) {{\color {green} $H(\{3\}, 0)$}};
\draw [<-,green] (-1.16,-1.6) -- (-1.16,-2);
\draw [green] (-1.16,-2) -- (30.east);
\node [anchor = west] at (-3.5,-.5) (10) {{\color {red} $H(\{1\}, 0)$}};
\draw [<-,red] (o1) -- (10.east);
\node [anchor = west] at (-3.5,0.1) (11) {{\color {red} $H(\{1\}, 1)$}};
\draw [<-,red] (-.35,.4) -- (11.east);
\node [anchor = west] at (-3.5,0.7) (12) {{\color {red} $H(\{1\}, 2)$}};
\draw [<-,red] (1.3,1.4) -- (12.east);
\node [anchor = west] at (3,0) (20) {{\color {blue} $H(\{2\},0)$}};
\draw [<-,blue] (-.3,-.5) -- (20.west);
\node [anchor = west] at (3,-.6) (21) {{\color {blue} $H(\{2\}, 1)$}};
\draw [<-,blue] (1.6,-1.1) -- (21.west);
%\draw (6.5,1.5) -- (6.5,-2);
%\node [anchor = west] at (7,.5) {$H(\{1,2,3\},(1,0,0))=\emptyset$};
\end{tikzpicture}
\caption{A drawing of a ``piece'' (in fact, a neighborhood of a fundamental region) of the arrangement $\AA$ of Example \ref{ex:rev2}, with explicit labeling of some of the $H(X,k)$s.\newline Notice that $H(\{1,2,3\},(1,0,0))=\emptyset$, and that $H(\emptyset,0)=\mathbb R^2$.}
\label{fig:lex1}
\end{figure}
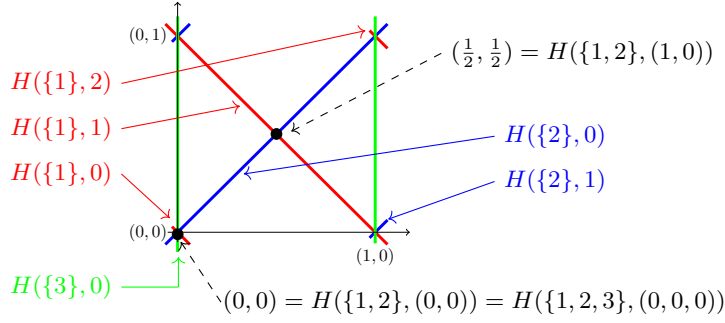

\end{example}

The poset of intersections of $\AA$ is the set
$$
\LL(\AA):=\{\cap \mathscr K \mid \mathscr K \subseteq \AA\}
\setminus\{\emptyset\}
$$
ordered by reverse inclusion (i.e., $x\leq y$ if $x\supseteq y$). This is a {\em geometric semilattice} in
the sense of Wachs and Walker \cite{WW}, see also Definition
\ref{df:GS}. 

A closer look at the definition will reveal that $\LL(\AA)$ is the poset of all nonempty $H(X,k)$, % with $X\neq \emptyset$, 
ordered by reverse inclusion.

\begin{remark}\label{rem:toric:first}
The {\em toric arrangement} associated to
  $\AA$ is the set
$$
\underline \AA:= \{ H/\mathbb Z^d \mid H\in \AA/\mathbb Z^d\}
$$
of quotients of orbits of the action on $\AA$. (Notice that $\mathbb Z^d$ acts on the set $\AA$ by permuting the hyperplanes and, for every $H_0\in \AA$, it acts on the space $H=\mathbb Z^d H_0$ by translations; in particular $H/\mathbb Z^d = H_0/\mathbb Z^d$ is a torus.)

The {\em poset of layers} of $\underline \AA$ is the set $\mathcal C(\AA)$ of connected components of the intersections of elements of $\underline
\AA$, ordered by reverse inclusion. This poset is an important feature
of toric arrangements: when $\kk = \mathbb C$, we have an
arrangement in the complex torus $\mathbb C^d / \mathbb Z^d$ (customarily given as a family of level sets of characters, see e.g.\ \cite[\S 2.1]{dAD2}) 
%Remark \ref{rem:toric2}) 
and
$\CC(\underline\AA)$  encodes much of the homological data about
the arrangement's complement (see e.g.\ \cite{dCP1,CaDe}). When $\kk =
\mathbb R$, this is the poset considered in \cite{ERS,Law}  pertaining to enumeration of the induced cell structure on the compact torus $\mathbb R^d / \mathbb Z^d \simeq (S^1)^d$ .

%With remark \ref{}
\end{remark}

\begin{figure}[h]
\begin{center}
\begin{tikzpicture}
%\node at (-2.5,0) (HA) {\includegraphics[scale=.55]{C_R_S_RGB.pdf}};
\node at (-3.5,0) (TA) {\includegraphics[scale=.4]{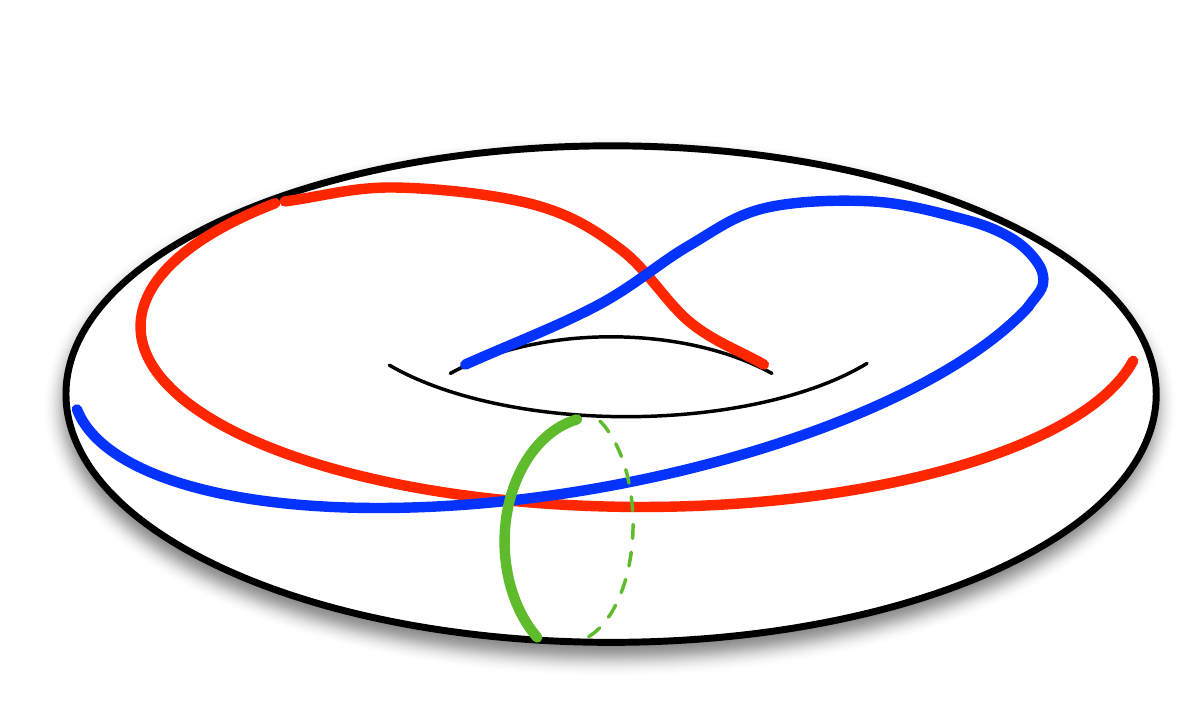}};
\node at (3.5,0) (PL) {\includegraphics[scale=.9]{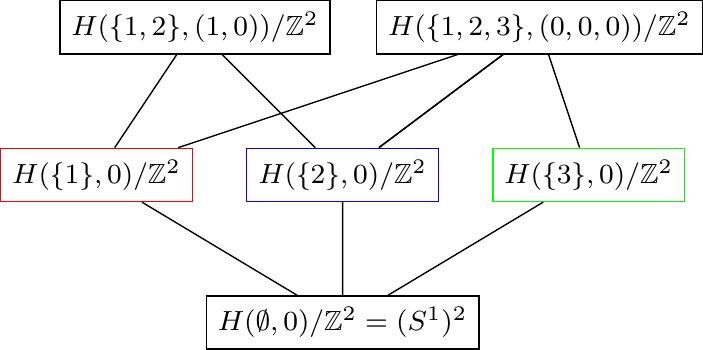}};
\end{tikzpicture}
\end{center}
\caption{A drawing of the toric arrangement $\underline{\AA}$ associated to the periodic line arrangement $\AA$ of Example \ref{ex:rev2} (left-hand side), and the poset of layers $\mathcal C (\underline{\AA})$ (right-hand side).}
\end{figure}

\begin{remark}\label{rem:nuovazione}
We see that $\mathcal C(\underline\AA)$ is the quotient (in the sense of Definition \ref{def:LS}) of the poset $\LL(\AA)$ under the induced action of $\mathbb Z^d$ (where  the element $\eee_l\in \mathbb Z^d$ maps $H(\{i\},j)$  to $H(\{i\}, j + \langle \eee_l\mid a_i\rangle )$).
\end{remark}

For $X\subseteq [n]$ and $k\in \mathbb Z^X$ define
  \begin{equation}\label{WXfirst}
    W(X):=\{k \in \mathbb Z^X \mid H(X,k)\neq \emptyset\}.
  \end{equation}

 We call $\AA$ {\em centered} if $\alpha_i=0$ for all $i=1,\ldots,n$  and assume this for simplicity throughout this section. Notice that the toric arrangements considered in \cite{Moc1} can be obtained from actions on centered arrangements.

\begin{remark}\label{rem:Wcenter} If $\AA$ is centered, then 
   $ W(X)= (A[X]^T \mathbb R^d) \cap \mathbb Z^X $ for all $X\subseteq [n]$, thus $W(X)$ is a pure subgroup (hence a direct summand) of  $\mathbb Z^X$. 
\end{remark}

%\begin{remark}\label{rem:urbild}
%  Notice that $H(X,k)$ is the preimage of $\alpha+k$ with respect to the
%  linear function $\mathbb R^d \to \mathbb R^X,\, x\mapsto A[X]^Tx$, thus $H(X,k)$ is connected
%  whenever nonempty . 
%\end{remark}

\begin{remark}\label{rem:urbild}
  Notice that $H(X,k)$ is the preimage of $\alpha+k$ with respect to the
  linear function $\mathbb R^d \to \mathbb R^X,\, x\mapsto A[X]^Tx$, thus $H(X,k)$ is connected
  whenever nonempty . 
\end{remark}

\begin{lemma}\label{lem:fiX}
  If $\AA$ is centered, the map $$\varphi_X : k\mapsto H(X,k) = \bigcap_{i\in X} H(\{i\},k_i)$$
  is a bijection between $W(X)$ and the connected
  components of %$H(X):=
  $\bigcup_{k\in \mathbb Z^X} H(X,k)$.
\end{lemma}

\begin{proof}
The map $\varphi_X$ is well-defined and surjective by definition of $W(X)$. It is injective by Remark \ref{rem:urbild}, as $A[X]^T$-preimages of distinct elements are disjoint.
\end{proof}

\begin{example}[Continued from Example \ref{ex:rev2}] \label{ex:rev22}
\begin{align*}
W(\{1,2,3\}) &= \{k \in \mathbb Z^3 \mid A^T x = k \textrm{ for some }x\in \mathbb R^2\}\\
&=\{ k \in \mathbb Z^3 \mid k_1+k_2 = 2k_3\}
\end{align*}
\begin{align*}
W(\{1,2\}) %&= \{k \in \mathbb Z^2 \mid A[\{1,2\}]^T x = k \textrm{ for some }x\in \mathbb R^2\}\\
&=\{ k\in \mathbb Z^2 \mid x_1+x_2 = k_1,\, x_1-x_2 = k_2\textrm{ for some }x\in \mathbb R^2 \} = \mathbb Z^2
\end{align*}

\end{example}

\begin{remark}\label{rem:azione2}
   We say that $\mathbb
  Z^d$ acts on $\mathbb Z^{\{i\}}$ by $\eee_l(j)= j + \langle \eee_l\mid a_i
  \rangle$ and, by coordinatewise extension, we obtain an action of
  $\mathbb Z^d$ on $\mathbb Z^X$ for all $X\subseteq [n]$.
  This induces an action of $\mathbb Z^d$ on $W(X)$ which is the
  action on $W(X)$ of its subgroup $A[X]^T\mathbb Z^d$ by addition and
  coincides with the ''natural'' action described in Remark \ref{rem:nuovazione}. 
\end{remark}

\begin{definition}
  For $X\subseteq [n]$ let
  $I(X):=A[X]^T\mathbb Z^d$ and consider
$$
Z(X):= \mathbb Z^X / I(X).
$$
\end{definition}

\newcommand{\colvec}[3]{
    \begin{pmatrix}\ifx\relax#1\relax\else#1\\\fi#2\\#3\end{pmatrix}
}
\newcommand{\bivec}[2]{
    \begin{pmatrix}\ifx\relax#1\relax\else#1\\\fi#2\end{pmatrix}
}

\begin{example}[Continued from Example \ref{ex:rev22}]\label{ex:rev222} In the case $X=\{1,2,3\}$, 
we have $I(\{1,2,3\})= A^T\mathbb Z^2= \colvec{1}{1}{1}\mathbb Z + \colvec{1}{-1}{0}\mathbb Z = W(X)$.
Since $\colvec{1}{1}{1},\colvec{1}{-1}{0},\colvec{1}{0}{0}$ is a unimodular basis of $\mathbb Z^3$, 
$$
Z(\{1,2,3\}) = (\mathbb Z \eee_1 \oplus W(X) )/I(X) = \mathbb Z \eee_1 \simeq \mathbb Z.
$$

In the case $X=\{1,2\}$ we have $W(\{1,2\})=\mathbb Z^2$ and $I(\{1,2\})=\left(\begin{array}{rr} 1 & 1 \\ 1 & -1 \end{array}\right) \mathbb Z^2$.
Hence, here 
$$
Z(\{1,2\}) = W(\{1,2\})/I(\{1,2\}) = \left\{\bivec{0}{0}+I(\{1,2\}), \bivec{1}{0}+I(\{1,2\})\right\} \simeq \mathbb Z / 2\mathbb Z.
$$
\end{example}

In general, we have the following description.

\begin{lemma}\label{lem:etaprimo}
  There is a direct sum decomposition of abelian groups $$Z(X) \simeq
  %\underbrace{\ker_{\mathbb Z}(A[X]^T) }_{ \cong
   \mathbb Z^\eta 
  %}
   \oplus W(X)/I(X),$$ where
  $\eta = \vert X \vert - \operatorname{rk}A[X]^T$, the {\em nullity}
  of $X$, is the rank of $Z(X)$ as a $\mathbb Z$-module.
\end{lemma}

\begin{proof} The decomposition $\mathbb Z^X \simeq \mathbb Z^{\eta} \oplus W(X)$ exists by Remark \ref{rem:Wcenter}, and $Z(X)$ decomposes as stated because $I(X)\subseteq W(X)$. For the claim on the rank, notice that both $W(X)$ and $I(X)$ are, by construction, free abelian groups of rank $\rk A[X]^T$, thus the quotient on the right hand side is pure torsion.
\end{proof}

\begin{remark}\label{rem:ram}
  Arithmetic matroids were introduced by d'Adderio and Moci in \cite{dAM}
 in order to study, in the centered case, the combinatorial properties of the rank and
  multiplicity functions on the subsets of $[n]$, where every $X$ has 
  $\rk(X):=\rk A[X]$ and  $m(X):= [\mathbb Z^d \cap A[X] \mathbb R^X : A[X]\mathbb
  Z^X]$.
  Since, by Remark \ref{rem:Wcenter} and Remark \ref{rem:azione2}, $$\vert W(X) / I(X) \vert = [W(X) : I(X)] = [ \mathbb Z^X
  \cap A[X]^T\mathbb R^d : A[X]^T\mathbb Z^d],$$  classical work of
  McMullen \cite{MM} shows that $m(X) = \vert W(X) / I(X)\vert$, and we
  recover in a geometric way the multiplicity function from \cite{dAM}.
\end{remark}

\begin{remark}\label{rem:pure}
  The function $\varphi_X$ of Lemma \ref{lem:fiX} induces a (natural)
  bijection  between the elements of $W(X)/ I(X)$ and the layers of
  $\left[\bigcup_{k\in \mathbb Z^X} H(X,k) \right]  / \mathbb Z^d$ in the toric arrangement $\underline\AA $.  This bijection exhibits the enumerative results proved in \cite{dAM}.
\end{remark}

\begin{example}[Continued from Example \ref{ex:rev222}]
Let us consider $X=\{1,2\}$. We have seen that the family $\left[\bigcup_{k\in \mathbb Z^X} H(X,k) \right] / \mathbb Z^2$ equals
$$
  \{H(X,\,(0,0) + A[X]^T\mathbb Z^2),\,\, H(X,\,(1,0) + A[X]^T\mathbb Z^2)\}.
$$ The map $\varphi_{X}$ is defined by 
$$
\varphi_{X}\left(\bivec{i}{j}\right)=H(X,(i,j))
$$ We have also previously seen that 
$$W(X)/I(X)=\left\{\bivec{0}{0}+I(X), \bivec{1}{0}+I(X)\right\}.$$
Hence we can easily compute
$$
\varphi_{X}\left(\bivec{0}{0}+I(X)\right) = H(X,(0,0)+ I(X))$$$$
\varphi_{X}\left(\bivec{1}{0}+I(X)\right) = H(X,(1,0)+ I(X))
$$
which, since by definition $I(X)=A[X]^T\mathbb Z^2$, is a bijection as stated.
% WITH THE {1,2} EXPLICIT
%$$
%  \{H(\{1,2\},(0,0) + A[\{1,2\}]^T\mathbb Z^2)  , H(\{1,2\},(1,0) + A[\{1,2\}]^T\mathbb Z^2)\}.
%$$ The map $\varphi_{\{1,2\}}$ is defined by 
%$$
%\varphi_{\{1,2\}}(\bivec{i}{j})=H(\{1,2\},(i,j))
%$$ We have also previously seen that 
%$$W(\{1,2\})/I(\{1,2\})=\left\{\bivec{0}{0}+I(\{1,2\}), \bivec{1}{0}+I(\{1,2\})\right\}.$$
%Hence we can easily compute
%$$
%\varphi_{\{1,2\}}\left(\bivec{0}{0}+I(\{1,2\})\right) = H(\{1,2\},(0,0)+ I(\{1,2\}))$$$$
%\varphi_{\{1,2\}}\left(\bivec{1}{0}+I(\{1,2\})\right) = H(\{1,2\},(1,0)+ I(\{1,2\}))
%$$
%which, since by definition $I(\{1,2\})=A[\{1,2\}]^T\mathbb Z^2$, is a bijection as stated.

\end{example}

\begin{remark}
  As proved in \cite{Moc1}, the arithmetic Tutte polynomial associated to this arithmetic
  matroid evaluates to many interesting invariants --- for instance to the
  characteristic polynomial of the poset $\mathcal C (\underline\AA)$. Thus, it counts the number of chambers of the associated toric
  arrangement in $(S^1)^d$. Moreover, the quotient of the induced action on the complexification
  of $\AA$ is an arrangement of subtori in
  $(\mathbb C^*)^d$, and the arithmetic Tutte polynomial specializes to
  the Poincar\'e polynomial of its complement.
\end{remark}

For $Y\subseteq X \subseteq [n]$ we consider $\mathbb Z^{X\setminus Y}\subseteq \mathbb Z^X$ as an intersection of coordinate subspaces and let $\pi_{X,Y}$ denote the coordinate projection of $\mathbb Z^X$ onto $\mathbb Z^{X\setminus Y}$. 
Since $I(X\setminus Y)=I(X) \cap \mathbb Z^{X\setminus Y} $, the map $\pi_{X,Y}$ restricts to a
surjection $I(X)\to I(X\setminus Y)$ and  
induces a map $\pi_{X,Y}: Z(X)\to Z(X\setminus Y)$ which, if $\vert Y \vert =1$, has cyclic kernel.

 \begin{lemma} For $X\subseteq [n]$, $i,j \in X$, the diagram
   $$\begin{CD}
     Z(X) @>\pi_{X,\,  i}>> Z(X\setminus i) \\
     @VV\pi_{X,\, j}V   @VV\pi_{X\setminus i, \,  j}V\\
     Z(X\setminus j) @>\pi_{X\setminus j ,\,  i}>> Z(X\setminus \{i,j\})
   \end{CD}
   $$
   is a pushout square of epimorphisms with cyclic kernels.
\end{lemma}
\begin{proof}
This can be verified either directly, or by  applying Lemma \ref{lem:backdoor}, where the rows of the required diagram arise from short exact sequences of the type $0\to I(X) \to \mathbb Z^X \to Z(X) \to 0$, and the morphisms between the sequences are induced by the projections $\pi_{\ast,\ast}$.
\end{proof}

\begin{theorem}\label{thm:dualMR}
  The assignment $I\mapsto Z([n]\setminus I)$ defines a matroid over
  $\mathbb Z$ on the ground set $[n]$. The underlying arithmetic matroid is dual to that  
  %It is the dual of the matroid $M_X$ over $\mathbb Z$
  associated to the list $X:=\{a_1,\ldots,a_n\} \subset \mathbb Z^d$
  in Example \ref{ex:exarmat}, see \cite{dAM}.
\end{theorem}

\begin{proof} The previous lemma shows that this in fact defines a
  matroid over $\mathbb Z$. For the duality claim let us write $\rk$ for matrix rank, $\rk_X$ for the rank function of the arithmetic matroid associated to the list $X$, and $\rk_Z$ for the rank function of the underlying arithmetic matroid, respectively. 
  Now, by Remark \ref{rem:uam} and Lemma \ref{lem:etaprimo} we have
  $$\rk_Z([n])-\rk_Z(I) = \operatorname{rank}_{\mathbb Z} Z(I\setminus [n])=\vert I^c \vert - \rk(A^T[I^c]),$$ where we write $I^c:=[n]\setminus I$. 
  Moreover, $\rk_X(I)=\rk A[X] = \rk A^T[I]$ (see  Example \ref{ex:exarmat}). Therefore we conclude
  $$
  \rk_X(I^c) = \vert I^c \vert + \rk_Z(I) - \rk_Z([n]),
  $$
  which is the very definition of $\rk_X$ being the rank function of the dual of the matroid defined by $\rk_Z$ (see \cite[Proposition 2.1.9]{Oxl}). 
  
  Similarly, let us write $m_X$ for the multiplicity function of the arithmetic matroid associated to the list $X$, and $m_Z$ for the rank function of the underlying arithmetic matroid, respectively.
Lemma \ref{lem:etaprimo} implies $m_Z(J)=\vert W(J^c)/I(J^c)\vert$ for all $J\subseteq [n]$. By remark \ref{rem:ram}, we conclude 
$m_Z(I) = m_X(I^c)$, corresponding to the relationship between multiplicity functions of dual arithmetic matroids in \cite{dAM}.
\end{proof}

\section{Overview: setup and main results}
\label{sec:defs}
Throughout, we fix a finitary semimatroid $\SS=(S,\CC,\rk)$ on the ground set
$S$ with set of central sets $\CC$, rank function $\rk: \CC \to
\mathbb N$ and semilattice
of flats $\LL$.

Let $G$ be a  group acting on $S$. Given $x\in S$  write $g(x)$ (or simply $gx$) for its image under $g\in G$, and $Gx$ for its orbit.
For every $X\subseteq S$ define
  $$\underline{X}:=\{Gx \mid x\in X\} \subseteq S/G$$ for the set of orbits met by $X$, and write
$$gX:=\{g(x) \mid x\in
X\},$$
thus obtaining an action of $G$ on the power set $2^S$.

\begin{remark} As a support for the intuition, the reader can think of the realizable case described in Example \ref{ex:ur}, namely that of a periodic arrangement of hyperplanes. As a tangible instance, consider Example \ref{ex:rev2}: there, the elements of the semimatroid are the hyperplanes $H(\{i\},j)$, and the action of $\mathbb Z^2$ is by standard translation, i.e., such that $k\in \mathbb Z^2$  sends $H(\{i\},j)$ to $H(\{i\},{j+\langle k \mid a_i \rangle})$ (compare Remark \ref{rem:nuovazione}). %We see that there are three orbits which we call $H_1,H_2,H_3$ where, for $i=1,2,3$, $H_i:= H^i_{\langle \mathbb Z^2 \vert a_i\rangle}$.
%. In fact, one of our aims here is to construct, in the full generality of semimatroids, the objects and invariants described in Section \ref{sec:MaEx}.
\end{remark}

\subsection{Group actions on semimatroids}

We now discuss group actions on a set $S$ that carries the structure of a semimatroid. In order to get a sense of the objects and notions introduced in the following definition the reader may already keep an eye on Example \ref{ex:running2} and Figure \ref{fig:Casistica}.

\begin{definition}[$G$-semimatroids]\label{def:prinzipal}
  An {\em action} of $G$ on a semimatroid $\SS:=(S,\CC,\rk)$ is an action of $G$ on the set $S$, whose induced action on $2^S$ preserves rank and centrality. 
  A $G$-semimatroid
$$\GS= G\circlearrowright (S,\CC,\rk)$$ is a semimatroid together with a
$G$-action. We define then
$$
\ES:= S/G ; \quad\quad \CC_{\GS}=\CC/G; \quad\quad 
\underline{\CC}:=\{\underline X \mid X\in \CC\};
$$
where we take quotients of sets, i.e., $\ES$ and $\CC_{\GS}$ are families of orbits.
  We call such an action
  \begin{itemize}
  \item[--] {\em centered} if there is an $X\in \CC$ with $\underline X = \ES$,
  \item[--] {\em \WT} if, for all $g\in G$ and all $x\in S$,  $\{x,g(x)\}\in \CC$ implies 
    $\rk(\{x,g(x)\})=\rk(\{x\}).$
  \item[--] {\em \ST} if, for all $g\in G$ and all $x\in S$, $\{x,g(x)\}\in \CC$ implies $g(x)=x$.
\end{itemize}
Moreover, for $A\subseteq \ES$ define
$$
\rkE(A):=\max \{\rk_{\mathcal C} (X) \mid \underline X \subseteq A\}
$$
and write $\rk(\GS):=\rkE(\ES)=\rk(\SS)$ for the rank of the $G$-semimatroid $\GS$.

\end{definition}

\begin{remark}
We call  a $G$-semimatroid $\GS$ {\em representable} if it arises from a periodic affine arrangement (see beginning of Section \ref{sec:MaEx}). In particular, $\SS$ is representable in the sense of Example \ref{ex:ur}.
\end{remark}

\begin{remark} Every \ST action is \WT. Moreover, every \WT action on a simple semimatroid is \ST.
\end{remark}

\begin{remark}\label{rem:CSposet}
We will sometimes find it useful to consider the set system $\CC_{\GS}$ as a poset, with the natural order defined by $GX \leq GY$ if $X\subseteq gY$ for some $g\in G$ (notice that this is well-defined: in fact, it is the poset-quotient of the poset of simplices of $\CC$ ordered by inclusion).
\end{remark}

\begin{defass} The action is called {\em cofinite} if the set $\CC_{\GS}$ is finite (in particular, $\ES$ is finite). We will assume this throughout without further mention.
\end{defass}

\begin{Mthm}\label{thm:polymat}
Every $G$-action on $\SS$ gives rise to a polymatroid on the ground
set $\ES$ with rank function $\rkE$ (see Remark \ref{def:polymat}).
This
polymatroid is a matroid if and only if the action is \WT: in
this case the triple 
$$\SS_\GS:=(\ES,\CS,\rkE)$$
is locally ranked and satisfies (CR2).
The triple $\SS_{\GS}$ is a matroid if and only if $\GS$ is centered. \end{Mthm}
\begin{proof}
  The first part of the claim is Proposition
  \ref{prop:unten}. The second part follows from Proposition \ref{prop:unten}
  and Proposition \ref{prop:quot}.
\end{proof}

\begin{example} \label{ex:running2}
As an illustration consider the semimatroid $\SS$ described in Example \ref{ex:running1} (and Figure \ref{fig:PseudoFirst}) with an action of the group $\mathbb Z^2$ given by
\begin{eqnarray*}
\eee_1(a_i) = a_{i+2},\,\, \eee_1(b_i)=b_{i+2},\,\, \eee_1(c_i)=c_i, \,\,\eee_1(d_i)=d_{i+1},\,\, \eee_1(e_i)=e_i \,\,\,\,\\
\eee_2(a_i) = a_{i+1},\,\, \eee_2(b_i)=b_{i-1},\,\, \eee_2(c_i)=c_{i+1},\,\, \eee_2(d_i)=d_{i},\,\, \eee_2(e_i)=e_{i+1}
\end{eqnarray*}
where, as above, $\eee_1, \eee_2$ is the standard basis of $\mathbb Z^2$.

This action gives rise to a well-defined $\mathbb Z^2$-semimatroid $\GS$, with
$$
\ES = \{a,b,c,d,e\},\,\,\, \underline{\CC}= 2^{\{a,b,c,d\}} \cup 2^{\{a,b,e\}} \cup 2^{\{e,d\}}
$$
and rank function defined  via $\rkE(\emptyset)=0$ and, for $A\subseteq \ES$,  $\rkE(A)=1$ if $\vert A \vert =1$, else $\rkE(A)=2$. A sketch of the fundamental region of this action is given in Figure \ref{fig:FR1}, and the associated $\CC_{\GS}$ is shown in Figure \ref{fig:CS}.

In this case, $\SS_{\GS}$ does not satisfy (CR1). For instance, with $X:=\{a,b,c\}$ and $Y:=\{a,b,e\}$, we have $X,Y \in \underline{\CC}$ with $\rkE(X\cap Y)=\rkE(\{a,b\})=2 = \rkE(X)$, but $X\cup Y=\{a,b,c,e\}\not\in \underline{\CC}$.
\end{example}

\begin{remark}
Notice that $\SS_{\GS}$ not being a semimatroid is not a consequence of $\GS$ not being representable. In fact, Figure \ref{fig:Casistica} shows that the properties of being representable, centered and $\SS_{\GS}$ being a semimatroid can appear in any combination not explicitly covered in Theorem \ref{thm:polymat}.
\end{remark}

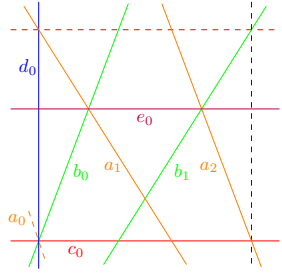
\begin{figure}[h]
\scalebox{.7}{%
\begin{tikzpicture}[
extended line/.style={shorten >=-#1,shorten <=-#1},
extended line/.default=1.5em]
    \node at (0,0) (O) {};
    \node at (-4,0) (W) {};
    \node at (-4,-4) (SW) {};
    \node at (0,-4) (S) {};
    \node at (-2.5,0) (Wl) {};
    \node at (-1.5,0) (Wr) {};       
    \node at (-2.5,-4) (SWl) {};
    \node at (-1.5,-4) (SWr) {};    
    \draw [extended line,-,dashed,red] (W.center) -- (O.center);
    \draw [extended line,-,red] (SW.center) -- (S.center);
    \draw [extended line,-,green] (SW.center) -- (Wl.center);
    \draw [extended line,-,green] (SWl.center) -- (O.center);
    \draw [extended line,-,dashed,blue] (O.center) -- (S.center);
    \draw [extended line,-,blue] (W.center) -- (SW.center);                
    \draw [extended line,-,orange] (Wr.center) -- (S.center);
    \draw [extended line,-,orange] (W.center) -- (SWr.center);
    \draw [extended line,-,purple] (-4,-1.5) -- (0,-1.5);
    \node[text=red] at (-3.3,-4.2) (c0) {$c_0$}; 
    \node[text=blue] at (-4.2,-.7) (d0) {$d_0$};     
    \node[text=green] at (-1.3,-2.7) (b0) {$b_1$}; 
    \node[text=green] at (-3.2,-2.7) (b1) {$b_0$}; 
    \node[text=orange] at (-2.6,-2.6) (a1) {$a_1$}; 
    \node[text=orange] at (-.8,-2.6) (a2) {$a_2$}; 
    \node[text=orange] at (-4.4,-3.6) (a0) {$a_0$}; 
    \node[text=purple] at (-2,-1.7) (e0) {$e_0$}; 
    \draw [extended line,dashed,shorten >=2em,shorten <=-1em, orange] (-4,-4) -- (-3.99,-4.03);
\end{tikzpicture}
}
\caption{A picture of the fundamental region of the $\mathbb Z^2$-semimatroid of Example \ref{ex:running2}, obtained from the natural action by translations on the pseudoline arrangement of Figure \ref{fig:PseudoFirst}. }
\label{fig:FR1}
\end{figure}

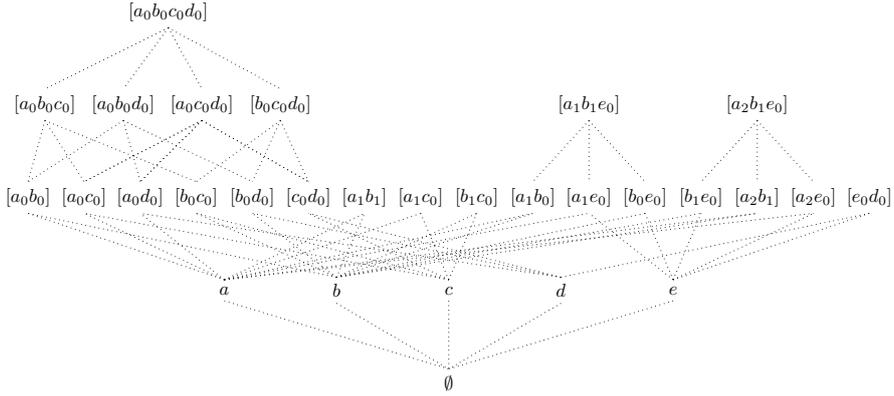
\begin{figure}[h]
\scalebox{.7}{%
\begin{tikzpicture}[x=3em, y=5em]
\node at (-4,0) (a) {$a$};
\node at (-2,0) (b) {$b$};
\node at (0,0) (c) {$c$};
\node at (2,0) (d) {$d$};
\node at (4,0) (e) {$e$};
\node at (-7.5,1) (a0b0) {$[a_0b_0]$};
\node at (-6.5,1) (a0c0) {$[a_0c_0]$};
\node at (-5.5,1) (a0d0) {$[a_0d_0]$};
\node at (-4.5,1) (b0c0) {$[b_0c_0]$};
\node at (-3.5,1) (b0d0) {$[b_0d_0]$};
\node at (-2.5,1) (c0d0) {$[c_0d_0]$};
\node at (-1.5,1) (a1b1) {$[a_1b_1]$};
\node at (-.5,1) (a1c0) {$[a_1c_0]$};
\node at (.5,1) (b1c0) {$[b_1c_0]$};
\node at (1.5,1) (a1b0) {$[a_1b_0]$};
\node at (2.5,1) (a1e0) {$[a_1e_0]$};
\node at (3.5,1) (b0e0) {$[b_0e_0]$};
\node at (4.5,1) (b1e0) {$[b_1e_0]$};
\node at (5.5,1) (a2b1) {$[a_2b_1]$};
\node at (6.5,1) (a2e0) {$[a_2e_0]$};
\node at (7.5,1) (e0d0) {$[e_0d_0]$};
\node at (-7.2,2) (a0b0c0) {$[a_0b_0c_0]$};
\node at (-5.8,2) (a0b0d0) {$[a_0b_0d_0]$};
\node at (-4.4,2) (a0c0d0) {$[a_0c_0d_0]$};
\node at (-3,2) (b0c0d0) {$[b_0c_0d_0]$};
\node at (2.5,2) (a1b0e0) {$[a_1b_1e_0]$};
\node at (5.5,2) (a2b1e0) {$[a_2b_1e_0]$};
\node at (-5,3) (a0b0c0d0) {$[a_0b_0c_0d_0]$};
    \tikzset{
    every path/.style={
        dotted
        }
    }
\draw (a.north) -- (a0b0.south);
\draw (a.north) -- (a0c0.south);
\draw (a.north) -- (a0d0.south);
\draw (a.north) -- (a1b1.south);
\draw (a.north) -- (a1c0.south);
\draw (a.north) -- (a1b0.south);
\draw (a.north) -- (a2b1.south);
\draw (a.north) -- (a2e0.south);
\draw (a.north) -- (a1e0.south);
\draw (b.north) -- (a0b0.south);
\draw (b.north) -- (b0c0.south);
\draw (b.north) -- (b0d0.south);
\draw (b.north) -- (a1b1.south);
\draw (b.north) -- (b1c0.south);
\draw (b.north) -- (a1b0.south);
\draw (b.north) -- (b0e0.south);
\draw (b.north) -- (b1e0.south);
\draw (b.north) -- (a2b1.south);
\draw (c.north) -- (a0c0.south);
\draw (c.north) -- (b0c0.south);
\draw (c.north) -- (c0d0.south);
\draw (c.north) -- (a1c0.south);
\draw (c.north) -- (b1c0.south);
\draw (d.north) -- (a0d0.south);
\draw (d.north) -- (b0d0.south);
\draw (d.north) -- (c0d0.south);
\draw (d.north) -- (e0d0.south);
\draw (e.north) -- (a1e0.south);
\draw (e.north) -- (b0e0.south);
\draw (e.north) -- (b1e0.south);
\draw (e.north) -- (a2e0.south);
\draw (e.north) -- (e0d0.south);
\draw (a0b0c0.south) -- (a0b0.north);
\draw (a0b0c0.south) -- (b0c0.north);
\draw (a0b0c0.south) -- (a0c0.north);
\draw (a0b0d0.south) -- (a0b0.north);
\draw (a0b0d0.south) -- (b0d0.north);
\draw (a0b0d0.south) -- (a0d0.north);
\draw (a0c0d0.south) -- (a0c0.north);
\draw (a0c0d0.south) -- (a0d0.north);
\draw (a0c0d0.south) -- (c0d0.north);
\draw (a0c0d0.south) -- (a0c0.north);
\draw (a0c0d0.south) -- (a0d0.north);
\draw (a0c0d0.south) -- (c0d0.north);
\draw (b0c0d0.south) -- (b0c0.north);
\draw (b0c0d0.south) -- (b0d0.north);
\draw (b0c0d0.south) -- (c0d0.north);
\draw (a1b0e0.south) -- (a1b0.north);
\draw (a1b0e0.south) -- (b0e0.north);
\draw (a1b0e0.south) -- (a1e0.north);
\draw (a2b1e0.south) -- (a2b1.north);
\draw (a2b1e0.south) -- (a2e0.north);
\draw (a2b1e0.south) -- (b1e0.north);
\draw (a0b0c0d0.south) -- (a0b0c0.north);
\draw (a0b0c0d0.south) -- (a0b0d0.north);
\draw (a0b0c0d0.south) -- (a0c0d0.north);
\draw (a0b0c0d0.south) -- (b0c0d0.north);
\node at (0,-1) (O) {$\emptyset$};
\draw (a.south) -- (O.north);
\draw (b.south) -- (O.north);
\draw (c.south) -- (O.north);
\draw (d.south) -- (O.north);
\draw (e.south) -- (O.north);
\end{tikzpicture}
}

\caption{The set system $\CC_{\GS}$, with dotted lines representing the Hasse diagram of the associated poset. We use shorthand notation, where we write, e.g., $[a_0b_0c_0]$ for the orbit $\mathbb Z^2\{a_0,b_0,c_0\}$.}
\label{fig:CS}
\end{figure}

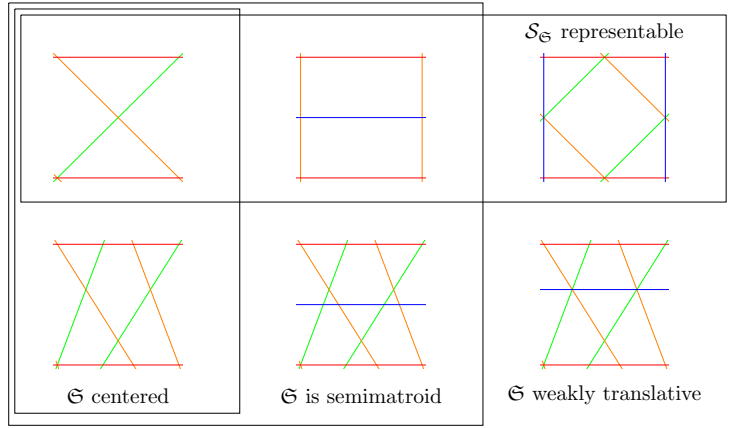
\begin{figure}[h]
\scalebox{.8}{%
\begin{tikzpicture}
\draw (-5.6,1.6) -- (6,1.6) -- (6,-1.5) -- (-5.6,-1.5) -- (-5.6,1.6);
\draw (-5.8,1.8) -- (2,1.8) -- (2,-5.2) -- (-5.8,-5.2) -- (-5.8,1.8);
\draw (-5.7,1.7) -- (-2,1.7) -- (-2,-5) -- (-5.7,-5) -- (-5.7,1.7);
\node at (-4,-4.7) (C) {$\GS$ centered};
\node at (0,-4.7) (R) {$\GS$ is semimatroid};
\node at (4,-4.7) (R) {$\GS$ weakly translative};
\node at (4,1.3) (S) {$\SS_\GS$ representable};
    \node at (-4,-3.2) (C_NR_S) {
        \includegraphics[scale=.5]{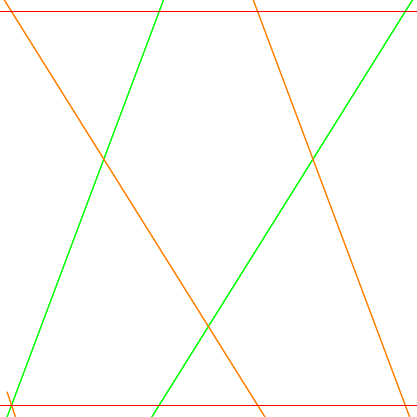}
    };

    \node at (-4,-.1) (C_R_S) {
        \includegraphics[scale=.5]{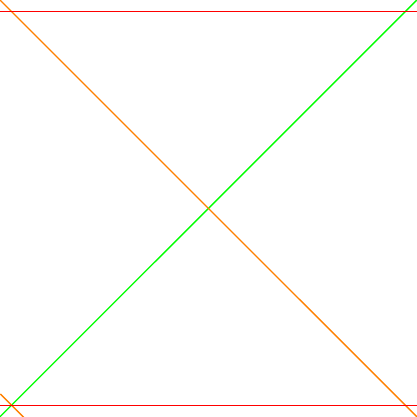}
    };

    \node at (4,-3.2) (NC_NR_NS) {
        \includegraphics[scale=.5]{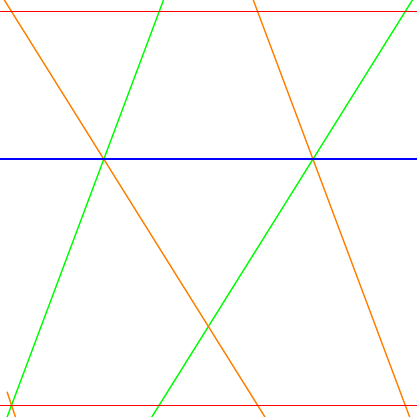}
    };

    \node at (4,-.1) (NC_R_NS) {
        \includegraphics[scale=.5]{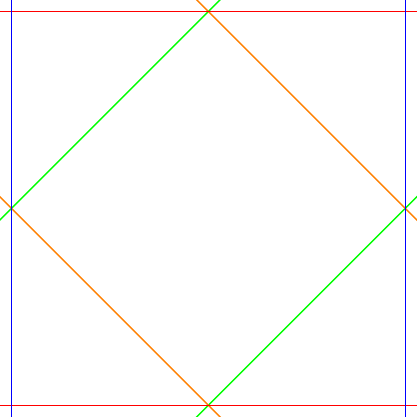}
    };

    \node at (0,-.1) (NC_R_S) {
        \includegraphics[scale=.5]{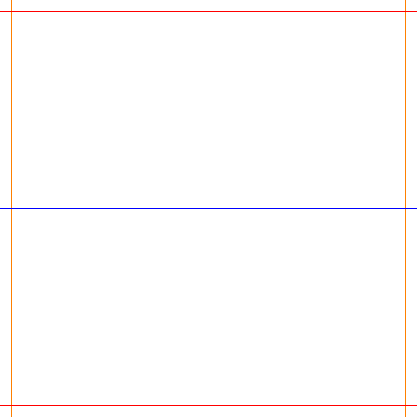}
    };

    \node at (0,-3.2) (NC_NR_S) {
        \includegraphics[scale=.5]{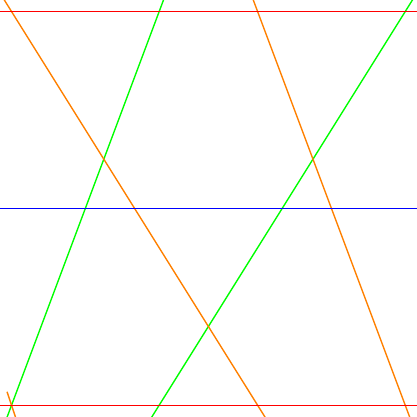}
    };
\end{tikzpicture}
}

\caption{This diagram depicts the fundamental regions of different cases of $\mathbb Z^2$-actions on arrangements of (PL-)pseudolines. 
The full arrangement can be recovered in each case by tiling the plane with copies of the respective picture, identifying the edges of adjacent squares, just as the arrangement of Figure \ref{fig:PseudoFirst} is obtained from the fundamental region of Figure \ref{fig:FR1}.\newline
%This diagram shows examples (arising from pseudoline arrangements)
 These examples realize every combination of centered, representable and ``$\SS_\GS$ is semimatroid'', within \WT actions (with the only constraint that centered actions always afford $\SS_\GS$ to be semimatroid - indeed in this case $\SS_\GS$ is a matroid).}
\label{fig:Casistica}
\end{figure}

We want to study the ``sets of orbits that give rise to central sets'': the following definition makes this sentence precise, and Example \ref{ex:running_precise} below illustrates it.

\begin{definition}\label{def:mm}
Let $\GS$ be a $G$-semimatroid. Given $A\subseteq \ES$ we define 
$$
\oC{A}:= \{X\in\CC\mid\underline X=A\}\subseteq \CC.
$$
For any given $A\subseteq \ES$, the set $\oC{A}$ carries a natural $G$-action, and we
will be concerned with the study of its orbit set, i.e., the set
$$
\oC{A} /G = \{\TT \in \CC_{\GS} \mid \lfloor \TT \rfloor = A\}
$$ 
where, for any orbit $\TT = G\{t_1,\ldots, t_k\}\in \CC_{\GS}$ we write
$$\lfloor \TT \rfloor := \{Gt_1,\ldots, Gt_k\},$$
so that $\lfloor \cdot \rfloor$ defines a map $\CGS \to \CS$. 
For every $A\subseteq \ES$, let then
$$
\mm(A) := \vert \oC{A} /G \vert.
$$
\end{definition}

\begin{remark}
We illustrate the relationships between the previous definitions by
fitting them into a diagram.
$$
\begin{tikzcd}[row sep=-0.1cm, column sep=huge]
  \CC \arrow{r}{/G} & 
  \CC_{\GS} \arrow{r}{\lfloor\cdot\rfloor}& 
  \CS \subseteq 2^{\ES} \\
  \rotatebox{90}{$\subseteq$} &   \rotatebox{90}{$\subseteq$} &   \rotatebox{90}{$\subseteq$}\\
    \oC{A} \arrow[mapsto, dashed]{r}{\textrm{preimage of}} & 
  \oC{A}/G \arrow[mapsto, dashed]{r}{\textrm{preimage of}}& A\\
    \rotatebox{90}{$\in$} &   \rotatebox{90}{$\in$} &   \rotatebox{90}{$\in$}\\
  X \arrow[mapsto]{r} & GX \arrow[mapsto]{r} & \underline X \\
\end{tikzcd}
$$
\end{remark}

The number $\mm(A)$ is nonzero if and only if $A\in \CS$. We will often tacitly consider the
restriction of $\mm$ to its support, which in the cofinite case 
defines a multiplicity function $\mm: \CS \to \mathbb N_{>0}.$

\begin{example}\label{ex:running_precise}
In our running example (the $\mathbb Z^2$-semimatroid $\GS$ of Example \ref{ex:running2}), we consider for instance the set $\{a,b\}\in\underline{\CC}$. Then,
$$
\oC{\{a,b\}} = \{\{a_i,b_j\} \mid i,j\in \mathbb Z\}
$$
and so
$$
\oC{\{a,b\}}/\mathbb Z^2 = \left\{\mathbb Z^2\{a_0,b_0\},
\mathbb Z^2\{a_1,b_0\},\mathbb Z^2\{a_1,b_1\},\mathbb Z^2\{a_2,b_1\}\right\},
$$
thus $\mm(\{a,b\})=4$. Repeating this procedure for all elements of $\underline{\CC}$ we obtain the multiplicities written as ``exponents'' next to the corresponding sets in Figure \ref{fig:Cmult}. 
\end{example}

\begin{figure}[h]
\scalebox{.8}{%
$\underline{\mathcal C}=$
\begin{tikzpicture}[baseline=(L.base),x=4em]
    \node at (0,1.9) (L) {};
    \node at (0,0) (0) {$\emptyset^{(1)}$};
    \node at (-4,1) (a) {$a^{(1)}$};
    \node at (-2,1) (b) {$b^{(1)}$};
    \node at (0,1) (c) {$c^{(1)}$};
    \node at (2,1) (d) {$d^{(1)}$};
    \node at (4,1) (e) {$e^{(1)}$};    
    \node at (-4,2) (ab) {$\{a,b\}^{(4)}$};
    \node at (-3,2) (ac) {$\{a,c\}^{(2)}$};
    \node at (-2,2)  (bc) {$\{b,c\}^{(2)}$};
    \node at (-1,2)  (ad) {$\{a,d\}^{(1)}$};
    \node at (0,2)  (bc) {$\{b,c\}^{(1)}$};
    \node at (1,2) (bd) {$\{b,d\}^{(1)}$};
    \node at (2,2) (ae) {$\{a,e\}^{(2)}$};
    \node at (3,2) (be) {$\{b,e\}^{(2)}$};
    \node at (4,2) (ed) {$\{e,d\}^{(1)}$};
    \node at (-3,3)  (abc) {$\{a,b,c\}^{(1)}$};
    \node at (-1.5,3)  (abc) {$\{a,b,d\}^{(1)}$};
    \node at (0,3)  (abd) {$\{a,c,d\}^{(1)}$};
    \node at (1.5,3)  (abd) {$\{b,c,d\}^{(1)}$};
    \node at (3,3)  (abe) {$\{a,b,e\}^{(2)}$};    
    \node at (0,4)  (abcd) {$\{a,b,c,d\}^{(1)}$}; 
    \draw [decoration={brace,amplitude=0.5em},decorate,ultra thick,gray]
 (-4.5,-.2) --  (-4.5,4.2);
 \draw [decoration={brace,amplitude=0.5em},decorate,ultra thick,gray]
 (4.5,4.2) --  (4.5,-.2);       
\end{tikzpicture}   
}

\caption{The set $\underline{\CC}$ for Example \ref{ex:running2}, with the multiplicity $\mm(A)$ written as a superscript of every set $A\in\underline{\CC}$.}
\label{fig:Cmult}
\end{figure}

\begin{definition}\label{def:normalA}
  We call the action of $G$
  \begin{itemize} 
  \item[-] {\em normal} if, for all $x\in S$, $\stab(x)$
    is a normal subgroup of $G$,
  \item[-] {\em almost arithmetic} if it is \ST and normal.
  \end{itemize}
\end{definition}

\begin{remark}
The two above-defined conditions are independent from each other and from the previous definitions. Indeed: %\begin{itemize}
%\item[--] 
the action of the symmetric group on its associated braid arrangement (see e.g. \cite[Example 1.9]{OT}) is neither normal nor \ST;
%\item[--] 
the permutation action of the symmetric group on $n$ distinct points in $\mathbb R$ is \ST but not normal;
%\item[--] 
the nontrivial action of $\mathbb Z_2$ on the uniform matroid of rank $1$ on two elements is normal but not \ST;
%\item[--] 
every realizable $G$-semimatroid is \ST.
%\end{itemize}

\end{remark}

\begin{Mthm}\label{thm:almost}
  If $\GS$ is a $G$-semimatroid associated to an almost-arithmetic
  action, then the pair $(\SS_\GS,\mm)$ is
  pseudo-arithmetic (see Definition \ref{def:AM}).
  If $\SS_{\GS}$ is a semimatroid, $\mm$ defines a pseudo-arithmetic semimatroid whose arithmetic Tutte polynomial, which we will call  $T_{\GS}(x,y)$ (cf.\ Definition \ref{def:Gtutte}), satisfies
  an analogue of Crapo's decomposition formula  (Theorem
  \ref{thm:craponew}) generalizing the combinatorial interpretation of  \cite[Theorem 6.3]{BM}. 
\end{Mthm}

\begin{proof}
This is proved as Proposition \ref{prop:almost} and Theorem \ref{thm:craponew}.
\end{proof}

\begin{remark}
If, in addition to satisfying the conditions of Theorem \ref{thm:almost}, $\GS$ is also centered, then $\SS_\GS$ is a matroid and $\mm$ defines a pseudo-arithmetic matroid on $E_\GS$ in the sense of \cite{BM}. Notice that this way we can produce a natural class of nonrealizable arithmetic matroids, e.g., by the action associated to non-stretchable pseudoarrangements (see Figure \ref{fig:FR1}).
\end{remark}

\begin{definition}\label{not:puntino}
If the action of $G$ is \ST, for every $X\subseteq S$ we
have that $\stab(X)=\cap_{x\in X} \stab(x)$. If, moreover, the action
is normal, it follows that, for every $X\in \CC$, $\stab(X)$ is a normal subgroup of $G$. We can then
  define the group
$$
\Gamma(X) := G/ \stab (X)
$$
and, for $g\in G$, write $[g]_X:=g+ \stab(X) \in \Gamma(X)$.
For any
$X\subseteq S$ consider then the group 
$$ 
\Gamma^X:=\prod_{x\in X} \Gamma(x) 
$$ 
and the natural map  
$$h'_X:G\to \Gamma^X, \quad h_X(g)=([g]_x)_{x\in X}.$$
Given $\gamma\in \Gamma^X$, let 
$$
\gamma.X := \{\gamma_{x} (x) \mid x\in X\}
$$
and, for all $X\in \CC$,  define
$$
W(X) := \{ \gamma\in \Gamma^X \mid \gamma.X\in \CC \}. 
$$
Since $X\in \CC$ implies $\im(h'_X)\subseteq W(X)$, we can restrict $h'_X$ to $W(X)$ as follows.
$$h_X: G\to W(X), \quad h_X(g):=h'_X(g).$$
\end{definition}

\begin{remark}
In order to help the intuition, notice that this definition of $W(X)$ coincides, in the realizable case, with that given in Equation \eqref{WXfirst}.
\end{remark}

\begin{example}
In our running example (from Example \ref{ex:running1} and \ref{ex:running2}) we can illustrate the construction of $W(X)$ by taking, e.g., $X=\{a_0,b_0,c_0\}\in \CC$. We have
$$
\stab(a_0)= \mathbb Z 
\left(\begin{smallmatrix} -1 \\ 2 \end{smallmatrix}\right),\quad
\stab(b_0)= \mathbb Z 
\left(\begin{smallmatrix} 1 \\ 2 \end{smallmatrix}\right),\quad
\stab(c_0)= \mathbb Z 
\left(\begin{smallmatrix} 1 \\ 0 \end{smallmatrix}\right),
$$
hence
\begin{align*}
\Gamma(a_0) &=\mathbb Z^2 / \stab(a_0) = 
\{\left(\begin{smallmatrix} 0 \\ k  \end{smallmatrix}\right) + \stab(a_0) \mid k\in \mathbb Z\}\simeq \mathbb Z\\
\Gamma(b_0) &=\mathbb Z^2 / \stab(b_0) = 
\{\left(\begin{smallmatrix} 0 \\ -k  \end{smallmatrix}\right) + \stab(b_0) \mid k\in \mathbb Z\}\simeq \mathbb Z\\
\Gamma(c_0)&=\mathbb Z^2 / \stab(c_0) = 
\{\left(\begin{smallmatrix} 0 \\ k \end{smallmatrix}\right) + \stab(c_0) \mid k \in \mathbb Z\}\simeq \mathbb Z
\end{align*}
where we take the isomorphism with $\mathbb Z$ to send $k\in \mathbb Z$ to the element listed in the braces.

Then, $\Gamma^X = \Gamma (a_0) \times \Gamma(b_0) \times \Gamma(c_0) \simeq \mathbb Z^3$ and for $\gamma\in \Gamma^X$, say $\gamma=(i,j,l)\in \mathbb Z^3$, our choice of the isomorphisms with $\mathbb Z$ above implies that 
$$
\gamma.\{a_0,b_0,c_0\}=\{a_{i},b_j, c_l\}
$$
and thus we see that $\gamma.\{a_0,b_0,c_0\}\in \CC$ if and only if
$i-l=j+l$ is an even number
 (compare Example \ref{ex:running1}). Therefore
$$
W(X)=\{(2h+l,2h-l,l) \mid h,l\in \mathbb Z\}
$$
is clearly seen to be a subgroup of $\Gamma^X$. We leave it to the reader to check that this applies to every $X$, thus the $\mathbb Z^2$-semimatroid $\GS$ is arithmetic (though not centered, neither representable, and $\SS_{\GS}$ is not a semimatroid).
\end{example}

\begin{definition}\label{def:args}
  An almost-arithmetic action is called {\em arithmetic} if $W(X)$ is a
  subgroup of $\Gamma^X$ for all $X\in \CC$.
\end{definition}

\begin{Mthm}\label{thm:arithm}
  If $\GS$ is an arithmetic $G$-semimatroid, then the pair $(\SS_\GS,\mm)$ is arithmetic. If, moreover, $\GS$ is centered, then $(E_\GS,\rkE,\mm)$ is an arithmetic matroid.
\end{Mthm}
\begin{proof} This is a combination of Proposition \ref{prop:almost} and Lemma \ref{lem:lari}.
\end{proof}

\subsection{Matroids over $\mathbb Z$} Under appropriate circumstances, the objects defined in Notation \ref{not:puntino} give rise to a matroid over $\mathbb Z$ 
%which, in the representable case, corresponds to the one associated to periodic arrangements in Section \ref{...}. 
%
%This matroid over $\mathbb Z$ is
 defined on the ground set $\ES$. In fact, for every arithmetic $G$-semimatroid, the groups $\Gamma(X)$ and $\Gamma^X$ from Definition \ref{not:puntino} do not depend on the choice of $X$ in $\oC{\underline{X}}$ (Lemma \ref{lem:unabrep}). Moreover, if we assume that all groups $\Gamma(x)$ are cyclic, then every group $\Gamma^X$ is abelian, and in particular all notions introduced in Definition \ref{not:puntino} above do not depend on the choice of $X$ inside $\oC{\underline{X}}$ (Lemma \ref{lem:Wunabrep}). So given $A\in \underline{\CC}$ it makes sense to write $\Gamma(A)$, $\Gamma^A$, $W(A)$ etc., see Section \ref{sec:arithm} for a more thorough discussion.

\begin{definition}\label{def:Emme}
Let $\GS$ denote an arithmetic and centered $G$-semimatroid such that, for all $a\in \ES$, the group $\Gamma(a)$ is cyclic. Given $A\subseteq E_\GS$ we will write $\Ac:=E_\GS\setminus A$ and, if $\Ac\in \underline{\CC}$, define
$$
  M_\GS(A):= \Gamma^{\Ac}/h'_{\Ac}(G).
  $$
\end{definition}

\begin{Mthm}\label{thm:RAM}
  Let $\GS$ denote an arithmetic and centered $G$-semimatroid such that, for all $a\in \ES$, the group $\Gamma(a)$ is cyclic. Then the abelian groups $M_\GS(A)$, where $A$ runs over all subsets of $\ES$, define a realizable matroid over $\mathbb Z$. 
  Moreover, if the groups $\Gamma(a)$ are infinite cyclic, the underlying matroid of $M_\GS$ is the dual to $(\ES,\rkE)$. If, additionally,  $W(A)$ is a pure subgroup of $\Gamma^A$ we have an isomorphism
$$
M_{\GS}(A) \simeq \mathbb Z^{\vert \Ac \vert - \rkE(\Ac)} \oplus W(\Ac)/h_{\Ac}(G)
$$
and the underlying arithmetic matroid is dual to $(\ES,\rkE,\mm)$.
\end{Mthm}

\begin{proof} This statement combines those of Proposition \ref{prop:D1}, Corollary \ref{cor:D2}, Corollary \ref{cor:D7}, Proposition \ref{prop:D3} and Corollary \ref{cor:D8}.
\end{proof}

%\begin{remark}
%  The modules $M_{\GS}(A)$ are well-defined because, by Lemma
%  \ref{lem:unabrep},
%  the group $W(X)$ does not depend on the choice of $X\in \oC{\underline{X}}.$
%\end{remark}

\begin{remark}\label{rem:toric2}
  In general, a {toric arrangement} in $(\mathbb C^*)^d$ is given as a family of level sets of characters of $(\mathbb C^*)^d$ (see e.g.\ \cite[\S 2.1]{dAD2}). By lifting the toric arrangement to the universal covering space of the torus one recovers a periodic affine hyperplane arrangement $\AA$. If $\GS$ is the $\mathbb Z^d$-semimatroid associated to this action as in Section \ref{sec:MaEx}, then $M_\GS$ is dual to the matroid over $\mathbb Z$ associated to the characters defining the toric arrangement (see Theorem \ref{thm:dualMR}).
\end{remark}

\subsection{Group actions on finitary geometric semilattices}

The main tool allowing us to establish a poset-theoretic formulation of the theory
of $G$-semimatroids is the following cryptomorphism result between finitary semimatroids and finitary geometric semilattices. Its proof is the object of Section \ref{sec:FSGS}.

\begin{Mthm} A poset $\LL$ is a finitary geometric semilattice if and only if it is isomorphic to the poset of flats of a finitary semimatroid. Furthermore, each finitary geometric semilattice is the poset of flats of an unique simple\footnote{See Definition \ref{df:simple}.} finitary semimatroid (up to isomorphism).
\label{thm:fsl}\end{Mthm}

We now discuss some basics about group actions on
finitary geometric semilattices.

\begin{definition}\label{def:LS}
  An action of $G$ on a geometric semilattice $\LL$ is given by a
  group homomorphism of
  $G$ in the group of poset automorphisms of $\LL$. We define
$$
\PS:=\LL/G,
$$ 
the set of orbits of elements of $\LL$ partially ordered such that $GX \leq GY$
if there is $g$ with $X\leq gY$ (where as usual we identify a group element in $G$ with the automorphism to which it corresponds). 
\end{definition}

\begin{remark}
 The fact that automorphisms of $\LL$ preserve rank implies that the above binary relation on $\PS$ is indeed a partial order. For another appearance of this definition of a ``quotient poset" see, e.g., \cite{ThWe}.
\end{remark}

\begin{example}[Toric arrangements]
If $\GS$ arises from a periodic arrangement of hyperplanes as in Section \ref{sec:MaEx}, then $\PS$ is the poset of layers of the associated toric arrangement (cf.\ Remark \ref{rem:toric:first} and Remark \ref{rem:nuovazione}).
  \end{example}
  \begin{example}[Toric pseudoarrangements]
 If $\GS$ is the $\mathbb Z^2$-semimatroid associated to a periodic 
  arrangement of pseudolines (see, e.g., Example \ref{ex:running1}) then $\PS$ is the poset of layers of the associated pseudoarrangement on the torus. 
  
The higher-dimensional analogue of this construction needs a (combinatorial) notion of a ``periodic affine arrangement of pseudoplanes'' whose intersection poset is a geometric semilattice. A forthcoming paper \cite{DK} will provide such a notion by defining finitary affine oriented matroids and studying their topological representation. If $\GS$ arises from an appropriate $\mathbb Z^d$-action on a rank $d$ finitary affine oriented matroid, then $\PS$ is the poset of layers of the associated pseudoarrangement on the torus. 
\end{example}

\begin{remark}
 It is clear that every action on a semimatroid induces an action on
  its semilattice of flats, and every action on a geometric
  semilattice induces an action on the associated simple semimatroid.  It is an exercise to reformulate the
  requirements of the different kinds of actions in terms of the
  poset - where, however, the distinction between \WT and
  \ST does not show. In our proofs we will mostly use the semimatroid language,
  in order to treat the most general case, and will call an action on a geometric semilattice {\em cofinite, \WT, \ST, normal, arithmetic}, etc., if the corresponding $G$-semimatroid is.
\end{remark}

\begin{example} The poset $\PP_{\GS}$ for the $\mathbb Z^2$-semimatroid of Example \ref{ex:running2} can be read off the picture of the fundamental region in Figure \ref{fig:FR1}, and gives the poset depicted in Figure \ref{fig:PSrunning}.
\end{example}

\begin{figure}[h]
\scalebox{.7}{%
\begin{tikzpicture}[y=3em]
\node at (0,0) (E) {$\emptyset$};
\node at (-4,2) (a) {$a$};
\node at (-2,2) (b) {$b$};
\node at (0,2) (c) {$c$};
\node at (2,2) (d) {$d$};
\node at (4,2) (e) {$e$};
\node at (0,4) (O) {$[a_0b_0c_0d_0]$};
\node at (-4,4) (S) {$[b_1c_0]$};
\node at (-2,4) (T) {$[a_1c_0]$};
\node at (2,4) (P) {$[a_1b_0e_0]$};
\node at (4,4) (Q) {$[a_2b_1e_0]$};
\node at (-6,4) (R) {$[a_1b_1]$};
\node at (6,4) (U) {$[d_0e_0]$};
\draw (E.north) -- (a.south);
\draw (E.north) -- (b.south);
\draw (E.north) -- (c.south);
\draw (E.north) -- (d.south);
\draw (E.north) -- (e.south);
\draw (a.north) -- (O.south);
\draw (a.north) -- (T.south);
\draw (a.north) -- (P.south);
\draw (a.north) -- (Q.south);
\draw (a.north) -- (R.south);
\draw (b.north) -- (O.south);
\draw (b.north) -- (S.south);
\draw (b.north) -- (P.south);
\draw (b.north) -- (Q.south);
\draw (b.north) -- (R.south);
\draw (c.north) -- (O.south);
\draw (c.north) -- (S.south);
\draw (c.north) -- (T.south);
\draw (d.north) -- (O.south);
\draw (d.north) -- (U.south);
\draw (e.north) -- (U.south);
\draw (e.north) -- (P.south);
\draw (e.north) -- (Q.south);
\end{tikzpicture}
}
\caption{The poset $\PP_{\GS}$ for the (nonrepresentable) $\mathbb Z^2$-semimatroid $\GS$ of our running Example \ref{ex:running2}, where we use the same conventions as in Figure \ref{fig:CS}.}
\label{fig:PSrunning}
\end{figure}
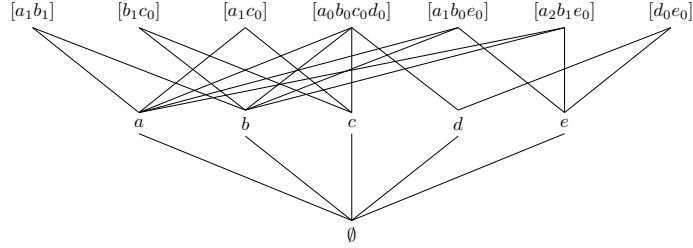

%In order to highlight the parallelism with the formulation in terms of rank-preserving actions, we state one additional definition.

The poset $\PS$ can also be obtained through a ``closure operator'' on $\CC_{\GS}$.

\begin{definition} Given a $G$-semimatroid $\GS: G\circlearrowright (S,\CC,\rk)$, define the function
$$
\kg: \CC_{\GS} \to \PP_{\GS}, \quad GX \mapsto G\cl(X)
$$ 
where $\cl$ denotes the closure operator associated to $(S,\CC,\rk)$ (see Remark \ref{rem:monotone}).
\end{definition}
The function $\kg$ is independent from the choice of representatives
(since the action is rank-preserving) and thus defines a ``closure
operator'' $\kappa_{\GS}:\CC_{\GS}\to \CC_{\GS}$ whose closed sets are
exactly the elements of $\PP_{\GS}$. 

Think of $\CC_{\GS}$ as a
poset with the natural order given by $GX \leq GY$ if there is $g\in
G$ with $gX\subseteq Y$, and let $\CC$ and $\underline\CC$ be ordered by inclusion. Then, for every \WT $\GS$-semimatroid we have the following commutative diagram of order-preserving functions.

\begin{center}
\begin{tikzcd}
\CC \arrow{r}{/G}  \arrow{d}{\cl}& 
\CC_{\GS} \arrow{r}{\lfloor\cdot \rfloor} \arrow{d}{\kg}&
\underline{\CC} \arrow{r}{\subseteq} &
2^{\ES}\arrow{d}{\underline{\cl}}\\
\LL \arrow{r}{/G}& 
\PP_{\GS} \arrow{rr}{\underline{\cl}\lfloor\cdot \rfloor}&
&
\LL_0
\end{tikzcd}
\end{center}

\subsection{Tutte polynomials of group actions}

\begin{definition}\label{def:Gtutte}
To every $G$-semimatroid $\GS$ we associate the polynomial 
 \begin{equation*}
  T_\GS(x,y):=\sum_{A\in \CS} \mm (A) (x-1)^{\rkE(\ES)-\rkE(A)}(y-1)^{\vert A \vert -
    \rkE(A)}. %\label{eq:4}
\end{equation*}
\end{definition}

This definition is natural in its own right, as can be seen in Section
\ref{ss:CP} and Section \ref{ss:CN}. If the action is centered (so in particular
$\SS_{\GS}$ is a matroid),  we recover Definition
\ref{caramai} and in particular, in the realizable, resp.\ arithmetic case, Moci's arithmetic Tutte polynomial \cite{Moc1}. 

Our first result is valid in the full generality of \WT actions, and concerns the characteristic polynomial of the poset $\PS$: we point, e.g., to \cite{Sta} for background on characteristic polynomials of posets, and to our Section \ref{ss:CP} for the precise definition.

\begin{Mthm}\label{thm:CP}
  Let $\GS$ be any \WT and loopless $G$-semimatroid, and let $\chi_{\GS}(t)$ denote the characteristic polynomial of the poset $\PS$. Then, 
  $$
  \chi_\GS(t) = (-1)^rT_\GS(1-t,0).
  $$
\end{Mthm}
\begin{proof}
The proof is given at the end of Section \ref{ss:CP}
\end{proof}

\begin{example}\label{ex:TChi}
For our running example we have (e.g., from Figure \ref{fig:Cmult})
\begin{align*}
T_{\GS}(x,y) &= (x-1)^2 + 5 (x-1) + 16 + 6(y-1) + (y-1)^2 \\
&=x^2+y^2+3x+4y+7
\end{align*}
and, from Figure \ref{fig:PSrunning}, 
$$
\chi_{\GS}(t) = t^2 -5t + 11.
$$
An elementary computation now verifies Theorem \ref{thm:CP} in this case.
\end{example}

The polynomials $T_{\GS}(x,y)$ associated to \ST actions satisfy a deletion-contraction recursion. Deletion and contraction for $G$-semimatroids correspond, in the representable case, to removing a set of orbits of hyperplanes, respectively vonsidering the periodic arrangement induced on any (nonempty) intersection of hyperplanes.

%To this end, we need to introduce the operations of deletion and for $G$-semimatroids. 

\begin{definition}
  For every $T\subseteq \ES$, $G$ acts on $\SS \setminus \cup T$. We denote
  the associated $G$-semimatroid by $\GS \setminus T$ and call this
  the {\em deletion} of $T$. We follow established matroid terminology
  and denote by $\GS [T] :=\GS\setminus (S\setminus \cup T)$ the {\em
    restriction} to $T$.
\end{definition}

\begin{remark}\label{lem:deletion}
A comparison with Definition \ref{def:DeRedef} shows that $\SS_{\GS [T]}= \SS_{\GS}[T]$ and that, for every $A\subseteq T$, $m_{\GS[T]}(A) = m_{\GS}(A)$. 
\end{remark}

\begin{definition}
  Recall $\CGS:=\CC/G$. For all $\mathcal T \in \CGS$ define the {\em
    contraction} of $\GS$ to $\mathcal T$ by choosing a representative
  $T\in \mathcal T$ and considering the action of $\stab(T)$ on the
  contraction $\SS / T$. This defines the $\stab(T)$-semimatroid
  $\GS/\mathcal T$.
\end{definition}

\renewcommand{\TT}{\mathcal T}
\begin{remark}
Clearly $\GS/\TT$ does not depend on the choice of the representative $T\in \TT$. Moreover, for all $e\in \ES$  we will abuse notation and write $\GS/e$ as a shorthand for $\GS/\{e\}$.
\end{remark}

\begin{remark}
By Proposition \ref{prop:CoDe}, \WTy, \STy, normality and atithmeticity of actions are preserved under taking contractions and restrictions.
\end{remark}

\begin{Mthm}\label{thm:MainCD}
  Let $\GS$ be a \ST $G$-semimatroid and let $e\in \ES$. Then
  \begin{itemize}
  \item[(1)] if $e$ is neither a loop nor an isthmus\footnote{See Definition \ref{df:simple}.} of $\SS_{\GS}$,
 $$
 T_\GS(x,y) = T_{\GS/e} (x,y) + T_{\GS \setminus e} (x,y);
 $$
 \item[(2)]
 if $e$ is an isthmus, $ T_\GS(x,y) =  (x-1)T_{\GS\setminus e} (x,y) + T_{\GS/ e} (x,y)$;
  \item[(3)]
 if $e$ is a loop, $ T_\GS(x,y) =  T_{\GS\setminus e} (x,y)+(y-1)T_{\GS/ e} (x,y)$.
 \end{itemize}
\end{Mthm}
\begin{proof}
The proof is given at the end of Section \ref{ss:TG}.
\end{proof}

\begin{example} If $\GS$ is the $\mathbb Z^2$-semimatroid of our running example, then $\GS \setminus e$ is given by the induced $\mathbb Z^2$-action on the semimatroid $\SS \setminus \{e_i\}_{i\in \mathbb Z}$ associated to the periodic arrangement of Figure \ref{fig:DeCo}.(a). Moreover, $\GS / e$ is the $\mathbb Z$-semimatroid given by the action of $\stab(e_0)= \mathbb Z \left(\begin{smallmatrix} 1 \\ 0 \end{smallmatrix}\right) \simeq \mathbb Z$ on the finitary semimatroid associated to the periodic arrangement of Figure \ref{fig:DeCo}.(b). A picture of the fundamental regions of these two actions is given in Figure \ref{fig:DeCoFR}, from which we can compute
\begin{align*}
T_{\GS \setminus e}(x,y) &= (x-1)^2 + 4(x-1) + 11 + 4(y-1) + (y-1)^2\\
&= x^2 + y^2 +2x +2y +5 \\
T_{\GS / e}(x,y) & = (x-1)+ 5 + 2(y-1) = x + 2y +2
\end{align*}
and easily verify that the sum of these polynomials equals $T_{\GS}(x,y)=x^2+y^2+3x+4y+7$ (Example \ref{ex:TChi}).
\begin{figure}[h]
\scalebox{.7}{%
\begin{tikzpicture}[
extended line/.style={shorten >=-#1,shorten <=-#1},
extended line/.default=1.5em]
    \node at (0,0) (O) {};
    \node at (-4,0) (W) {};
    \node at (-4,-4) (SW) {};
    \node at (0,-4) (S) {};
    \node at (-2.5,0) (Wl) {};
    \node at (-1.5,0) (Wr) {};       
    \node at (-2.5,-4) (SWl) {};
    \node at (-1.5,-4) (SWr) {};    
    \draw [extended line,-,dashed,red] (W.center) -- (O.center);
    \draw [extended line,-,red] (SW.center) -- (S.center);
    \draw [extended line,-,green] (SW.center) -- (Wl.center);
    \draw [extended line,-,green] (SWl.center) -- (O.center);
    \draw [extended line,-,dashed,blue] (O.center) -- (S.center);
    \draw [extended line,-,blue] (W.center) -- (SW.center);                
    \draw [extended line,-,orange] (Wr.center) -- (S.center);
    \draw [extended line,-,orange] (W.center) -- (SWr.center);
    \node[text=red] at (-3.3,-4.2) (c0) {$c_0$}; 
    \node[text=blue] at (-4.2,-.7) (d0) {$d_0$};     
    \node[text=green] at (-1.3,-2.7) (b0) {$b_1$}; 
    \node[text=green] at (-3.2,-2.7) (b1) {$b_0$}; 
    \node[text=orange] at (-2.6,-2.6) (a1) {$a_1$}; 
    \node[text=orange] at (-.8,-2.6) (a2) {$a_2$}; 
    \node[text=orange] at (-4.4,-3.6) (a0) {$a_0$}; 
    \draw [extended line,dashed,shorten >=2em,shorten <=-1em, orange] (-4,-4) -- (-3.99,-4.03);
\end{tikzpicture}
}
\scalebox{.7}{%
\begin{tikzpicture}
    \node at (0,0) (N) {};
    \node at (0,-2.5) (S) {};
    \draw [-,blue] (-3.9,-.1) -- (-3.9,.1);
    \draw [-,green] (-2.45,-.1) -- (-2.45,.1);
    \draw [-,orange] (-2.55,-.1) -- (-2.55,.1);    
    \draw [-,green] (-1.45,-.1) -- (-1.45,.1);
    \draw [-,orange] (-1.55,-.1) -- (-1.55,.1);    
    \draw [-,blue] (0,-.1) -- (0,.1);
    \draw [-,purple] (-4.2,0) -- (0.2,0);
    \node[text=blue] at (-4,.3) (d0) {$d_0$};     
    \node[text=green] at (-2.5,.7) (b0) {$b_0$};  
    \node[text=green] at (-1.5,.7) (b1) {$b_1$}; 
    \node[text=orange] at (-2.5,.3) (a1) {$a_1$}; 
    \node[text=orange] at (-1.5,.3) (a2) {$a_2$}; 
\end{tikzpicture}
}
\caption{}
\label{fig:DeCoFR}
\end{figure}
\end{example}

\begin{table}[h]
\begin{center}
  
\newcommand{\spazio}{$\,$\\[-15pt]}
\resizebox{\linewidth}{!}{
{
\begin{tabular}{c||c|c|c|c|c|}
$G$-semimatroid&
Loc.\ ranked triple&
Multiplicity&
Poset&
Polynomial&
Modules\\
$\mathfrak S$&
$\mathcal S_{\mathfrak S}$&
$m_{\mathfrak S}$ &
$\mathcal P_{\mathfrak S}$ &
$T_{\mathfrak S}(x,y)$
& $M_{\mathfrak S}$ \\[2pt] \hline
Weakly translative & 
\begin{minipage}{0.16\linewidth}
\begin{center} well-defined \\
(Theorem \ref{thm:polymat})
\end{center}
\end{minipage} & & 
\multicolumn{2}{|c|}{ 
\begin{minipage}{0.32\linewidth}
\begin{center}
$\chi_{\mathcal P _{\mathfrak S}} (t)=(-1)^rT_{\mathfrak S} (1-t,0)$\\
(Theorem \ref{thm:CP})
\end{center}
\end{minipage}}
& 
\begin{minipage}{0.2\linewidth}
\spazio
\begin{center}
$\,$\\ $\,$
\end{center}
\spazio
\end{minipage} 
\\\hline 
\begin{minipage}{0.15\linewidth}
\begin{center}
Translative
\end{center}
\end{minipage}
&\multicolumn{2}{|c|}{
\begin{minipage}{0.23\linewidth}
\begin{center}
Pseudo-arithmetic \\ (Proposition \ref{lem:P})
\end{center}
\end{minipage} 
} &  &
\begin{minipage}{0.23\linewidth}
\begin{center}
Deletion--contraction\\ recursion \\ (Theorem \ref{thm:MainCD})
\end{center}
\end{minipage} &
\begin{minipage}{0.2\linewidth}
\spazio
\begin{center}
$\,$\\ $\,$\\ $\,$
\end{center}
\spazio
\end{minipage} 
\\\hline 
\begin{minipage}{0.15\linewidth}
\begin{center}
Translative \\
and normal
\end{center}
\end{minipage} &
\multicolumn{2}{|c|}{
\begin{minipage}{0.25\linewidth}
\begin{center}
Almost-arithmetic\\ (P, A.1.2, A2) \\ (Theorem \ref{thm:almost})
\end{center}
\end{minipage} 
} & & &\\
\begin{minipage}{0.15\linewidth}
\begin{center}
...and $\SS_{\GS}$ a \\
semimatroid
\end{center}
%\spazio
\end{minipage}
& 
\multicolumn{2}{|c|}{} & &
\begin{minipage}{0.18\linewidth}
%\spazio
\begin{center}
Activity decomposition \\ (Theorem \ref{thm:craponew})
\end{center}
%\spazio
\end{minipage}  & 
\begin{minipage}{0.2\linewidth}
\spazio
\begin{center}
$\,$\\ $\,$ \\ $\,$
\end{center}
\spazio
\end{minipage} \\\hline
Arithmetic
&  \multicolumn{2}{|c|}{
\begin{minipage}{0.3\linewidth}
%\spazio
\begin{center}
Arithmetic \\
(Theorem \ref{thm:arithm})
\end{center}
\end{minipage}}
&&&
\begin{minipage}{0.2\linewidth}
\spazio
\begin{center}
$\,$\\ 
\end{center}
\spazio
\end{minipage} \\\hline\hline
Centered & Matroid & & & & 
\begin{minipage}{0.2\linewidth}
\spazio
\begin{center}
$\,$\\ 
\end{center}
\spazio
\end{minipage} \\\hline\hline
\begin{minipage}{0.15\linewidth}
$\,$
\begin{center}
Realizable and centered
\end{center}
$\,$
\end{minipage}&
\multicolumn{2}{|c|}{
\begin{minipage}{0.27\linewidth}
%\spazio
\begin{center}
Arithmetic matroid \\ dual to that of \cite{BM}
\end{center}
%\spazio
\end{minipage}}&
\begin{minipage}{0.17\linewidth}
%\spazio
\begin{center}
Poset of layers of toric arrangement
\end{center}
%\spazio
\end{minipage} &
\begin{minipage}{0.15\linewidth}
%\spazio
\begin{center}
Arithmetic Tutte polynomial
\end{center}
%\spazio
\end{minipage} &
\begin{minipage}{0.15\linewidth}
\begin{center}
Representable
matroid over $\mathbb Z$
\end{center}
%\spazio
\end{minipage}\\\hline
\end{tabular}}
}
\end{center}
\caption{A tabular overview of our setup and our results}\label{table}
\end{table}

\section{Some examples}
\label{sec:exs}

\begin{example}[Reflection groups]
    \label{ex:nosemimat} Let $G$ be a finite or affine complex
    reflection group acting on the intersection poset of its
    reflection arrangement. This setting has been considered
    extensively, especially in the finite case (see  e.g.\ the
    treatment of Orlik and Terao \cite{OT}). These actions
    are not translative, and thus  fall at the margins of our present
    treatment. Still, we would like to mention 
    them as a motivation for %token of the fact that
	further investigation of non-translative actions ---
    e.g., the case where $(E,\rkE)$ is a polymatroid.
    % - might bear some interest.
    %is clearly  warranted. 
    %, for instance in order to characterize a class of posets associated to representations of finite groups that were recently used in computations of Motivic classes \cite{DeMa}.
\end{example}

\begin{example}[Toric arrangements]
\label{ex:not centered}
  The natural setting in order to develop a combinatorial framework
  for toric arrangements is that of the group $\mathbb Z^d$ acting by translations on an
  affine hyperplane arrangement on $\mathbb C^d$ (see Section \ref{sec:MaEx}). Such actions will
  often fail to be centered. Therefore we will try to state our results as much as possible without centrality assumptions, adding them only when needed in order to establish a link to the arithmetic and algebraic matroidal structures appeared in the literature.
\end{example}

\def\Ro{{\color {red} c}}
\def\Bo{{\color {blue} d}}
\def\Go{{\color {LimeGreen} b}}
\def\Oo{{\color {Orange} a}}

The next examples will refer to Figure \ref{fig:DeCoFR} and Figure
\ref{FigPseudolines}. These are to be interpreted as the depiction of
a fundamental region for an action of $\mathbb Z^2$ by unit
translations in orthogonal directions (vertical and horizontal) on an
arrangement of pseudolines in $\mathbb R^2$  (see Example \ref{ex:running1}) which, then, can be recovered by 'tiling' the plane by translates of the depicted squares. Notice that the intersection poset of any arrangement of pseudolines is trivially a geometric semilattice, and thus defines a simple semimatroid. We will call $\Oo$, $\Go$,  $\Ro$, $\Bo$ the orbits of the respective colors.

\begin{figure}[h]
  \centering
  \begin{minipage}{.45\textwidth}
  \centering
  \scalebox{.8}{
\begin{tikzpicture}[
extended line/.style={shorten >=-#1,shorten <=-#1},
extended line/.default=1.5em]
    \node at (0,0) (O) {};
    \node at (-4,0) (W) {};
    \node at (-4,-4) (SW) {};
    \node at (0,-4) (S) {};
    \node at (-3,0) (Wl) {};
        \node at (-2.1,0) (Wcl) {};
        \node at (-1.9,0) (Wcr) {};
    \node at (-1,0) (Wr) {};
        \node at (-2.5,0) (Wrr) {};       
    \node at (-3,-4) (SWl) {};
        \node at (-2.1,-4) (SWcl) {};
        \node at (-1.9,-4) (SWcr) {};
    \node at (-1,-4) (SWr) {};    
    \draw [extended line,-,dashed,red] (W.center) -- (O.center);
    \draw [extended line,-,red] (SW.center) -- (S.center);
    \draw [extended line,-,green] (SW.center) -- (Wl.center);
    \draw [extended line,-,green] (SWl.center) -- (Wcl.center);    
    \draw [extended line,-,green] (SWcl.center) -- (Wr.center);
    \draw [extended line,-,green] (SWr.center) -- (O.center);
    \draw [extended line,-,dashed,blue] (O.center) -- (S.center);
    \draw [extended line,-,blue] (W.center) -- (SW.center);                
    \draw [extended line,-,orange] (Wr.center) -- (S.center);
    \draw [extended line,-,orange] (Wcr.center) -- (SWr.center);
    \draw [extended line,-,orange] (Wl.center) -- (SWcr.center);
    \draw [extended line,-,orange] (W.center) -- (SWl.center);
    \draw [extended line,dashed,shorten >=2em,shorten <=-1em, orange] (-4,-4) -- (-3.99,-4.03);
\end{tikzpicture}
}
  \end{minipage}
    \caption{Figure for Example \ref{ex:noarithm} 
  }
  \label{FigPseudolines}
\end{figure}
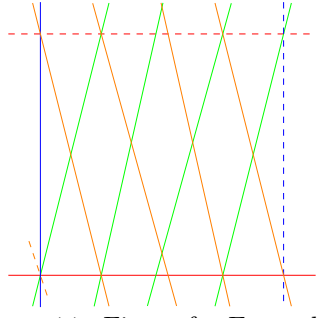

\begin{example}
  \label{ex:noarithm}
The $\mathbb Z^2$-semimatroid described in Figure \ref{FigPseudolines} is clearly almost-arith\-metic, but cannot be arithmetic, because the multiplicity $m_{\GS}(\{\Ro,\Go,\Oo\})=3$ does not divide $m_{\GS}(\{\Ro,\Oo\})=4$, violating (A.1.1).
\end{example}

\begin{example}
  \label{ex:arithm-noalg}
  One readily verifies that the $\mathbb Z^2$-semimatroid described at the left-hand side of
  Figure \ref{fig:DeCoFR} is arithmetic. However, $M_{\GS}$ is not a matroid over $\mathbb Z$. Indeed, the requirement of Definition \ref{def:mator} fails for the square 
  $$
  \begin{CD}
  M_{\GS} (\{\Go\})\cong\mathbb Z @>{?}>>   M_{\GS} (\{\Go,\Ro\})\cong\mathbb Z_2 \\
  @V{?}VV @V{}VV\\
    M_{\GS} (\{\Oo,\Go\})\cong\mathbb Z_4 @>{}>>   M_{\GS} (\{\Oo,\Go,\Ro\}) \cong\{0\}
  \end{CD}
  $$
  where the the condition that the maps be surjections with cyclic kernel determines everything up to leaving two possibilities for the left-hand side vertical map: neither of these gives the required pushout. 
%  which clearly cannot be made to be a pushout square of surjections with cyclic kernel.
\end{example}

\begin{remark}
  Examples where $M_\GS$ is a nonrealizable matroid over $\mathbb Z$
  can easily be generated in a trivial way, e.g.\ by considering
  trivial group actions on nonrealizable matroids. We do not know
  whether there is a periodic pseudoarrangement for which $M_\GS$ is a
  nonrealizable matroid over $\mathbb Z$.
\end{remark}

\begin{example}[The realizable case]
  \label{ex:realizable}
%We close with the realizable case: t
The arrangement on the top left of Figure \ref{fig:Casistica} is a periodic affine arrangement in the sense
of Section \ref{sec:MaEx}: thus, the associated $M_\GS$ is a realizable matroid over $\mathbb Z$.
\end{example}

\begin{example}[Crystallographic root systems] An important family of realizable examples is that of the periodic hyperplane arrangements arising as the reflection arrangements of the affine Coxeter groups associated to crystallographic root systems, where the weight lattice acts by translation. In this setting, some enumerative results in terms of Dynkin diagrams were obtained by Moci \cite{phi}.
\end{example}

\section{Finitary geometric semilattices}\label{sec:FSGS}
In this section we study posets associated to finitary semimatroids. This leads us to consider geometric semilattices in the sense of Wachs and Walker \cite{WW}. Our goal is to prove a finitary version of the equivalence between simple semimatroids and geometric semilattices given in \cite{Ard}.

%\begin{remark}[Basic poset terminology]\label{rem:basicposets} 
We start by recalling some basic terminology about partially ordered sets. The reader already familiar with poset theory may skip to Definition \ref{df:GS}. We refer to Stanley's book \cite{Sta} for a comprehensive introduction to this topic. 

A partially ordered set (for short {\em poset}) is a set $\mathcal P$ endowed  with a partial order relation, i.e., a transitive, antisymmetric and reflexive binary relation which we denote by $\leq$. 
As is customary, we write $p<q$ if $p\leq q$ and $p\neq q$. Given $x\in \mathcal P$ we write $\mathcal P_{\leq x} :=\{p\in \mathcal P \mid p \leq x\}$ for the set of elements below $x$, and define $\mathcal P_{\geq x}$ analogously. We say that the poset $\mathcal P$ is {\em bounded below} (resp.\ bounded above) if it possesses a unique minimal (resp.\ maximal) element, that is an element $\hat{0}\in \mathcal P$ (resp.\ $\hat{1}\in \mathcal P$) with $\mathcal P _{\geq \hat{0}} = \mathcal P$ (resp.\ $\mathcal P _{\leq \hat{1}} = \mathcal P$). If $\mathcal P$ is bounded below and bounded above, we call it simply {\em bounded}.

The {\em join} of a subset $X\subseteq \mathcal P$, written $\vee X$, if it exists, is defined by 
$$
\mathcal P _{\geq \vee X} = \{p\in \mathcal P \mid p\geq x\textrm{ for all }x\in X\}.
$$
Analogously the {\em meet} $\wedge X$, if it exists, is defined by
$$
\mathcal P _{\geq \wedge X} = \{p\in \mathcal P \mid p\leq x\textrm{ for all }x\in X\}.
$$
If $X=\{x,y\}$, we write $x\vee y := \vee X$ and $x\wedge y := \wedge X$.

If the meet of any two elements exists, then so does the meet of every finite set of elements, and $\mathcal P$ is called {\em meet-semilattice}. {\em Join-semilattices} are defined accordingly. If $\mathcal P$ is both a meet- and a join- semilattice, then it is called a {\em lattice}.

A {\em chain} in $\mathcal P$ is any totally ordered subset, i.e., any $\omega=\{p_1,\ldots,p_k\}\subseteq \mathcal P$ such that $p_0 < p_1 < \cdots < p_k$. The {\em length} of such a chain is $\ell (\omega) = \vert \omega\vert - 1$. In this paper {\bf we assume throughout that all posets are chain-finite}, i.e., all chains have finite length. The (closed) {\em interval} between $p,q\in \mathcal P$ is the set $$[p,q]:=\{x\in \mathcal P \mid p\leq x \leq q\}.$$
We say that $q$ {\em covers} $p$ if $[p,q]=\{p,q\}$. The {\em atoms} of a bounded below poset $\mathcal P$ are the elements that cover $\hat{0}$.  A bounded below poset $\mathcal P$ is called {\em atomic} if every element is a join of atoms, i.e., if for every $p\in \mathcal P$ there is a set $A$ of atoms of $\mathcal P$ such that $p=\vee A$.

A poset $\mathcal P$ is called {\em ranked} if there is a function $\rk:\mathcal P \to \mathbb N$ such that $\rk(q)=\rk(p)+1$ whenever $q$ covers $p$. If $\mathcal P$ is bounded below we assume $\rk(\hat{0})=0$, and the condition above is equivalent to the fact that, for every $x\in \mathcal P$,  all maximal chains of $\mathcal P_{\leq x}$ have the same (finite) length. In general, a poset $\mathcal P$ is called {\em graded} if all maximal chains have the same (finite) length. % (such a poset is called {\em graded}). 
A ranked lattice $\mathcal P$ is called {\em semimodular} if, for all $x,y\in \mathcal P$, $\rk(x\vee y) + \rk(x\wedge y) = \rk(x) + \rk(y)$.
\begin{definition}\label{df:GL}
A {\em geometric lattice} is a finite, atomic and semimodular lattice.
\end{definition}

A set $A$ of atoms of a ranked, bounded below poset, is called {\em independent} if the join $\vee A$ exists and satisfies $\rk(\vee A) = \vert A \vert$.

A {\em morphism of posets} is an order preserving map, i.e., a morphism between posets $(\mathcal P, \leq)$ and $(\mathcal Q, \preccurlyeq)$ is a function $f: \mathcal P \to \mathcal Q$ such that $f(p_1\leq p_2)$ implies $f(p_1)\preccurlyeq f(p_2)$, for all $p_1,p_2\in \mathcal P$. An {\em isomorphism} of posets is a bijective morphism of posets with order-preserving inverse.

\begin{definition}[See Theorem 2.1 in \cite{WW}]\label{df:GS}
 %Let $\LL$ be 
 A (chain-finite) ranked meet-semilattice $\LL$ 
 %with bounded rank function $\rl$ %. If $\LL$ satisfies the following conditions, it 
 is called a \textit{finitary geometric semilattice} if it satisfies the following conditions.
% \todo{simple $\SS$ unique up to iso + Kontrolle ranked}
\begin{itemize}
  \item[(G3)] {\em There is $N\in \mathbb N$ such that} every (maximal) interval in $\mathcal L$ is a (finite) geometric lattice {\em with at most $N$ atoms}.
  \item[(G4)] For every independent set $A$ of atoms of $\LL$ and every $x\in\LL$ such that $\rl(x)<\rl(\vee\! A)$, there is $a\in A$ with $a\nleq x$ and such that $x\vee a$ exists. 
\end{itemize}\end{definition}

\begin{remark}\label{gslAltDef}
The definition given in \cite{WW} of a finite geometric semilattice is that of a
finite ranked meet-semilattice which satisfies:
\begin{itemize}
	\item[(G1)] Every element is a join of atoms.
	\item[(G2)] The collection of independent sets of atoms is the set of independent sets of a matroid.
\end{itemize}
In the finite case, Wachs and Walker prove that this is equivalent to Definition \ref{df:GS}, which we choose to take as our definition because of its more immediate generalization to the infinite case.
  We nevertheless keep, for consistency, the labeling of the conditions as in \cite{WW}. 
  
  In passing to the infinite case we have added the part of (G3) that is written in italic. If the poset is finite, then this addition is redundant, and it does not appear in \cite{WW}. Let us note here that the proof Theorem \ref{thm:fsl} remains valid if the italic part of (G3) and the requirement finite-dimensionality of $\mathcal C$ in Definition \ref{def:FS} are simultaneously dropped.
\end{remark}

%\begin{proposition} In a finitary geometric semilattice $\LL$, the following properties are satisfied:
%\begin{itemize}
%	\item[(G1')] If $x,y\in\LL$ and $y$ covers $x$ then there is an atom $a\in\LL$ such that $x\vee a=y.$
%	\item[(G1'')] Every element $x\in\LL$ is a join of an independent set of atoms, which we call \textit{basis for} $x.$ 
%\end{itemize}
%\end{proposition}
%
%\begin{proof} In finitary geometric semilattices, the property (G1) is satisfied by (G3) and from there the proof follows \cite{WW}.
%\end{proof}

    \begin{remark}\label{CiErPr}
In view of the proof of Theorem \ref{thm:fsl} and for later reference we note that finitary semimatroids satisfy the following properties (i.e., a ``local'' version of (R2) and a stronger version of (CR1) and (CR2)). 
\begin{itemize}
  \item[(R2')] If $X\cup x\in\CC$ then $\rk(X\cup x)-\rk(X)$ equals 0 or 1.
  \item[(CR1')] If $X,Y\in\CC$ and $\rk(X)=\rk(X\cap Y),$ then $X\cup Y\in\CC$ and $\rk(X\cup Y)=\rk(Y).$
  \item[(CR2')] If $X,Y\in\CC$ and $\rk(X)<\rk(Y),$ then $X\cup y\in\CC$ and $\rk(X\cup y)=\rk(X)+1$ for some $y\in Y-X.$
\end{itemize}
The proof is analogous to that in the finite case given in \cite[Section 2]{Ard}.
       \end{remark}

\begin{proof}[Proof of Theorem \ref{thm:fsl}] Let $\SS=\scr$ be a finitary semimatroid. Recall from Definition \ref{def:latticeofflats} the closure operator $\cl$ and the poset of flats $\lc$ of $\SS$. We begin by showing that $\lc$ is a geometric semilattice.
\begin{itemize}
\item[-] {\em $\lc$ is a chain-finite ranked meet semilattice.}
Given flats $X,Y$ of $\SS$, the subset $X\cap Y$ is also central and its closure $\cl(X\cap Y)\in\lc$ is a lower bound of $X$ and $Y$ in $\lc$ by Remark \ref{rem:monotone}. Now suppose $A\in\lc$ is a lower bound of $X,Y$ in $\lc$, thus $A\ssq X,Y.$ In particular, this means $A\ssq X\cap Y\ssq\cl(X\cap Y)$. Therefore, the set $\cl(X\cap Y)$ is the meet of $X$ and $Y$ in $\lc$. 
%Since we consider flats of $\SS$, 
Now, (CR2') implies that $\lc$ is ranked with rank function $\rl:=\rc$. Moreover, an infinite chain in $\lc$ must arise from an infinite chain of simplices in $\CC$, which would violate local finiteness.
% the length of any chain, is bounded by the (finite) dimension of $\CC$, . 

\item[-] {\em Condition (G3).} 
%We have to show that every maximal interval is a geometric lattice of width bounded as in the claim. 
If $X$ is a maximal flat of $\SS$, then in particular $\rc$ is defined for every subset of $X$ and satisfies axioms (R1-R3). Thus $\rc$ defines a matroid $M$ on $X$ whose closure operator coincides with $\cl$ (since $X$ is closed, $\cl$ restricts to a function $2^X \to 2^X$), and thus the lattice of flats of $M$ is isomorphic to the interval $[\hat0 , X]$ in $\lc$, proving that this interval is indeed a geometric lattice.

For the bound on the number of atoms of intervals, notice that a top simplex $X$ of $\CC$ is a maximal flat of $\SS$, hence its cardinality is at least the number of atoms in $\LL(\mathcal S)_{\leq X}$.
%, say $n$. The maximum cardinality of an antichain in an atomic poset with $n$ atoms is reached in the boolean case, where it equals ${n \choose \lceil n/2 \rceil}$. Thus, we have that the number 
%The width of a maximal interval is bounded by t
%The cardinality of a top simplex of $\CC$ , \todo{explain} thus we can take, e.g., 
%$N={d+1 \choose \lceil d+1 /2 \rceil}$, where 
Thus, if $d$ is the (finite) dimension of the simplicial complex $\CC$, the poset $\mathcal L(\mathcal S)$ satisfies (G3) with $N=d+1$.

\item[-] {\em Condition (G4).} Now let $A$ be an independent set of atoms in $\lc$ and $X$ a flat of $\SS$ such that $\rc(X)<\rc(\vee A)=\rc(\cl(\cup A))=\rc(\cup A).$ By (CR2), there is an element $a\in \cup A \setminus X$ such that  $X\cup a\in \CC$. In particular, $\cl(\{a\})$ is an atom from $A$ such that $\cl(\{a\})\nleq X$ in $\lc$. Furthermore, by Remark \ref{rem:monotone} the set $X\cup\cl(\{a\})$ is a subset of $\cl(X\cup a)$  -- and hence central as well. So the join $X\vee\cl(\{a\})=\cl(X\cup\cl(\{a\}))$ exists and (G4) is satisfied.

 This concludes the proof that $\lc$ is a finitary geometric semilattice.
\end{itemize}

Conversely, let $\LL$ be a finitary geometric semilattice. Let $S_\LL$ denote the set of atoms of $\LL$ and set $$\CCl=\{X\ssq S_\LL\mid\vee X\in\LL\}.$$
Moreover, we define the function $$\rcl:\CCl\rightarrow\nn, \quad X\mapsto\rl(\vee X).$$ Now suppose $Y\ssq X\in\CCl$. Then $\vee X$ is  an upper bound for $Y$ and thus the join $\vee Y$ exists (since $\LL$ is a meet-semilattice). Hence, the collection $\CCl$ is an abstract simplicial complex. Since  $\vert X \vert \leq \vert S_\mathcal L \cap \mathcal L_{\leq \vee X} \vert$ for all $X\in \mathcal C_{\mathcal L}$,  the cardinality of any simplex is bounded by $N$; thus $\CC$ is finite-dimensional. We will show that $\SS_{\LL}:=(S_\LL, \CCl, \rcl)$ is a finitary semimatroid with semilattice of flats $\LL(\SS_{\LL})$ isomorphic to $\LL$.

\begin{itemize}

\item[-] {\em Axioms (R1) - (R3).}
For every $X\in\CCl,$ the join $\vee X$ exists and the interval $[\hat0,\vee X]$ is a geometric lattice by (G3). Thus (e.g., by Remark \ref{rem:GL}) it defines a matroid $M_X$ with ground set the atoms in $[\hat0,\vee X]$, whose rank function is a restriction of $\rl$ (resp. $\rcl$). Thus (R1) holds for $X$ because it holds in $M_X$. Moreover,  (R2) holds for every $X\subseteq Y \in \CCl$ because it holds in $M_Y$,
 and  (R3) holds for every $X,Y$ with $X\cup Y \in \CCl$ because it does in $M_{X\cup Y}$.

\item[-] {\em Axiom (CR1).} Take $X,Y\in\CCl$ with $\rcl(X)=\rcl(X\cap Y)$, i.e., $\rl(\vee X)=\rl(\vee(X\cap Y))$. %Evidently, %the join of $X$ is also an upper bound of $X\cap Y$, thus 
Since $\LL$ is a ranked poset, the former rank equality and the evident relation $\vee (X\cap Y)\leq\vee X$ %, since $\LL$ is ranked, %Since by assumption  and $\LL$ is ranked, 
imply $\vee(X\cap Y)=\vee X$. So  
 $$\vee X=\vee(X\cap Y)\leq\vee Y,$$ that is to say every upper bound of $Y$ is also an upper bound of $X.$ Hence $\vee (X\cup Y)=\vee Y$ and $X\cup Y\in\CC$, and (CR1) is  satisfied. 

\item[-] {\em Axiom (CR2)}. Let $X,Y$ be in $\CCl$ and such that $\rcl(X)<\rcl(Y)$. Choose an independent set $A\ssq Y$ with  $\vee A = \vee Y$. 
%(see (G1'') - since $[\hat0,\vee Y]$ is a geometric lattice we can find this basis in $Y$). 
Property (CR2) for $X$ and $Y$ now follows applying (G4) to $X$ and $A$.

\item[-] {\em There is a poset isomorphism $\LL \simeq \lcl$.} Let $\varphi:\LL\rightarrow\lcl$ be defined by 
\begin{equation}\label{eq:phidef}
\varphi(x):=\{a\in S_\LL\mid a\leq x\}.
\end{equation}
 For $\varphi$ to be well-defined, we must check that, for all $x\in \LL$, $\varphi(x)$ is a flat of $\SS_\LL$. Let  $x\in \LL$. First, by (G3) we have $\vee\varphi(x)=x$ and thus $\varphi(x)\in\CCl.$ Now suppose $b$ is an element of $S$ such that $\varphi(x)\cup \{b\}\in \CC_{\LL}$  and $\rcl(\varphi(x)\cup b)=\rcl(\varphi(x)).$ This means that $\rl(\vee(\varphi(x)\cup b))=\rl(\vee\varphi(x))$, and since clearly
% Since by assumption $\rl(\vee(\varphi(x)\cup b))=\rl(\vee\varphi(x))$ and
 $\vee(\varphi(x)\cup b)\leq\vee\varphi(x)=x$, from the fact that $\LL$ is ranked we conclude $\vee(\varphi(x)\cup b)\leq\vee\varphi(x)=x$. In particular, $b\leq x$, so $b\in\varphi(x)$. 
 %it follows that equality holds here. Then $b\leq x,$ that is to say $b\in\varphi(x).$ 
 This proves that the set $\varphi(x)$ is closed, hence a flat of $\SS_\LL$.

The function $\varphi$ is clearly injective. %We need to show surjectivity.
To check surjectivity, let $Y$ be a flat of $\SS_\LL$. We have to find some $x\in\LL$ with $\varphi(x)=Y$, and indeed $x=\vee Y$ will do.
% But this is the same as to say $x=\vee Y$, which exists by definition of $\SS_\LL$. 

%To prove  
Moreover, comparing the definition of $\varphi$ in Equation \eqref{eq:phidef} one readily checks the following equivalences
$$
\varphi(x) \leq \varphi(y) \Leftrightarrow
\varphi(x) \subseteq \varphi(y) \Leftrightarrow x \leq y
$$

Thus, both $\varphi$ and its inverse are order-preserving, and $\varphi$ is the required isomorphism.% and the theorem is proven.  
\end{itemize}

%For a geometric semilattice $\LL$ and atoms $x\neq y$ of $\LL$ with $x\vee y\in \LL$ holds $\rk_{\CCl}(x)=1$ and $\rk_{\CCl}(x,y)=2$. Thus, t

The semimatroid $\SS_{\LL}=(S_\LL, \CCl, \rcl)$ is simple by construction. We are left with showing that for every simple semimatroid  $\SS=(S,\CC,\rk)$ with a poset-isomorphism $$\psi: \LL(\SS)\stackrel{\cong}{\longrightarrow}\LL$$ we can construct an isomorphism between $\SS$ and $\SS_\LL$. 

Since $\SS$ is simple, for every $x\in S$ the set $\{x\}$ is closed. Thus, $\psi$ induces a natural bijection $$\psi_S: S\to S_\LL,\quad\quad \{\psi_S(x)\}=\psi(\{x\})$$ %Now consider $X=\{x_1,...,x_k\}\in\CC$ and

To see that $\psi_S$ induces a well-defined function $\CC \to \CC_{\LL}$, consider any $X=\{x_1,\ldots,x_k\}\in \mathcal C$. Then, using the definition of $\psi_S$ and the fact that $\psi$ is an isomorphism,
$$
\bigvee_{i=1}^k\{\psi_S(x_i)\} =
\bigvee_{i=1}^k\psi(\{x_i\}) = \psi(\bigvee_{i=1}^k\{x_i\}), $$
hence the right-hand side exists in $\mathcal L$, and thus  $\psi_S(X)\in \mathcal C_\LL$. 

 An analogous argument using $\psi_S^{-1}$ (together with the fact that $\psi_S$ is monotone by definition) shows that in fact $\psi_S$ induces an {\em isomorphism} of simplicial complexes $\mathcal C \cong \mathcal C_\LL$.

It remains to show that $\psi_S$ preserves ranks of central sets. 
For this consider any  $X=\{x_1,\ldots,x_k\}\in \mathcal C$ and compute
$$
\rc(X) = \rk_{\LL(\SS)} (\bigvee_i \{x_i\}) = \rk_{\LL} (\bigvee_i \psi(\{x_i\})) = \rk_{\CCl} (\psi_S(X)).
$$

%, for $i=1,\ldots , k$ the atom $a_i=\varphi (\{x_i\})$ of $\LL$.
% be the corresponding atoms under the semilattice isomorphism. 
%Because $\varphi$ is an isomorphism  $\bigvee_{i=1}^ka_i\in\LL$ corresponding to $\cl(X)\in\LL(\SS).$ Furthermore, $A=\{a_1,...,a_k\}\ssq\CCl$ and $\rc(X)=\rl(\bigvee_{i=1}^ka_i)=\rk_{\CCl}(A).$ 
%Conversely, for $A=\{a_1,...,a_k\}\in\CCl$ there exist $x_1,...,x_k\in S$ corresponding to $a_1,...,a_k$ and a flat $F\in\LL(\SS)$ corresponding to $\bigvee_{i=1}^ka_i\in\LL$. By the isomorphism, it follows that $X:=\{x_1,...,x_k\}\ssq F\in\CC$, $\cl(X)=F$ and $\rk_{\CCl}(A)=\rl(\bigvee_{i=1}^ka_i)=\rc(F)=\rc(X).$
\end{proof}

\section{The underlying matroid of a group action}
\label{sec:orb}
This section is devoted to the proof of Theorem \ref{thm:polymat}. Let $\GS$ be  a $G$-semimatroid associated to an action of $G$ on a
semimatroid $(S,\CC,\rk)$.
Recall from Section \ref{sec:defs} the set $\ES:=S/G$ of orbits of elements, the
family $\CS=\{\unten{X}\subseteq \ES \mid X\in \CC \}$, and that we
only consider actions for which $E_\GS$ is finite. 

For every
$A \subseteq \ES$ define
$$
J(A) :=\{X\in \CC \mid \underline X \subseteq A\}
$$
\def\max{\operatorname{max}}
and write $J_{\max} (A)$ for the set of inclusion-maximal elements of
$J(A)$.

\begin{lemma}\label{lem:rank}
  For every $X,Y\in J_{\max} (A)$, $\rk(X)=\rk(Y)$.
\end{lemma}
\begin{proof}
  By way of contradiction assume $\rk(X) > \rk (Y)$. Then with (CR2)
  we can find $x\in X\setminus Y$ with $Y\cup x  \in \CC$
  and $\underline{Y\cup x} \subseteq A$, contradicting
  maximality of $Y$.
\end{proof}

\begin{definition}\label{def:rkE}
  For any $A\subseteq \ES$ choose $X\in J_{\max}(A)$ and let $\rkE(A) :=
  \rk(X)$, in agreement with Definition \ref{def:prinzipal}. Lemma \ref{lem:rank} shows that this is well-defined and independent on the choice of $X$.
\end{definition}

\begin{remark}
  For all $A\subseteq \ES$ we have $$\rkE(A)=\max\{\rkE(A') \mid
  A'\subseteq A, A'\in \CS\},$$
  because $A'\subseteq A$ implies $J(A')\subseteq J(A)$.
\end{remark}

\begin{proposition}\label{prop:unten}
  The pair $(\ES,\rkE)$ always satisfies (R2) and (R3), and thus defines
  a polymatroid on $\ES$.  %when $\ES$ is finite. \\
  Moreover, $(\ES,\rkE)$ satisfies (R1) if and only if the action is \WT.
\end{proposition}
\begin{proof} $\,$\\
%\todo{Lista senza rientro per sostituire i bullets}
%\begin{itemize}
%\item 
\noindent $\bullet$ {\em $(\ES,\rkE)$ is a polymatroid.}
  Property (R2) is trivial, we check (R3). Consider $A,A'\subseteq \ES$, and
  choose $B_0\in J_{\max}(A\cap A')$. By Lemma \ref{lem:rank},
  \begin{equation}\label{eqn1}
    \rk(B_0)=\rkE(A\cap
  A').\tag{*}
  \end{equation}
%  \eqref{eqn:einstein}
  In particular,  $B_0\in J(A)$ and thus we can find $B_1\in J(A)$ such
  that 
  \begin{equation}B_0\cup B_1 \in J_{\max } (A)\tag{*}
  \end{equation}
 and a maximal $B_2\in J(A')$ such that
  $B_0\cup B_1 \cup B_2 $ is in $ J(A'\cup A)$. Then,
  
  \begin{equation}\label{eqn2}
   B_0\cup B_1 \cup B_2 \in J_{\max} (A'\cup A), \tag{*}
  \end{equation}
  because otherwise we could complete it with some $B_2'\in J(A)$ in
  order to get an element of $J_{\max}(A\cup A')$ -- but then,
  $B_0\cup B_1 \cup B_2' \supseteq B_0\cup B_1 \in J_{\max}(A)$, thus
  $B_2' =\emptyset$ by the choice of $B_1$. Using the identities \eqref{eqn1} and axiom (R3) for $(S,\CC,\rk)$ we obtain
  \begin{align*}
    \rkE(A\cap A') + \rkE(A\cup A') -\rkE(A) &= \rk(B_0) + \rk(B_0\cup
    B_1\cup B_2) - \rk(B_0\cup B_1)\\ &\leq \rk(B_0\cup B_2) \leq
    \rkE(A'),
  \end{align*}
where the last inequality follows from $\underline{B_0\cup B_2} \subseteq A'$. This proves that $\rkE$ satisfies (R3).

\noindent $\bullet$ {\em Weakly translative implies (R1)}

Suppose that the action is \WT. For (R1) we need to show that $0\leq
\rkE(A) \leq \vert A \vert$ for every $A\subseteq \ES$. The left hand
side inequality is trivial. Consider $A\subseteq \ES$ and choose $X\in
J_{\max}(A)$. 

\noindent {\bf Claim.} {\em For every $x\in X$ with $g(x)\in X$ we
have $\rk(X) = \rk(X\setminus g(x))$.}

\noindent {\em Proof of claim.} Using (R3) in
$(S,\CC,\rk)$ on the sets $X\setminus g(x)$ and $\{x,g(x)\}$, we obtain
$$\rk(X) + \rk(x)\leq \rk(X\setminus g(x)) + \rk(\{x,g(x)\}) =
\rk(X\setminus g(x)) + \rk(x)$$ where in the last equality we used
\WTy of the action. Thus we get $\rk(X)\leq \rk(X\setminus g(x))$
and, the other inequality being trivial from (R2), we have the claimed
equality. \hfill $\square$

%With it, we can then choose 
Choose a system $X'$ of representatives
of the orbits in $\unten{X}$. We obtained the claimed inequality by computing
\begin{equation}
\rkE(A)=\rk(X)=\rk(X')\leq \vert X' \vert = \vert \unten{X} \vert \leq \vert A \vert
\label{eq:obliqua}
\end{equation}
where the second equality holds because of the claim above.

\noindent $\bullet$ {\em (R1) implies \WT.} By contraposition.
%Conversely, i
If the action is not \WT, choose $x\in X$ and $g\in G$
violating the \WTy condition and consider $A:= \{Gx\}$. First
notice that $x$ cannot be a loop, since if $\rk(x)=0$ then
$\rk(g(x))=0$ and $\rk(\{x,g(x)\})$ must equal $0$ (otherwise it would
contain an independent set of rank $1$), implying that $x$ is not a loop, thus 
$\rk(\{x,g(x)\})=\rk(x)$ in agreement with the \WTy
condition. Hence it must be $\rk(x)=1$, and we have $\rkE(A) \geq \rk(\{x,g(x)\}) > \rk(x) =1=\vert \{A\} \vert$, contradicting (R1). 
%\end{itemize}
\end{proof}

\begin{corollary}\label{cor:rank}
If the action is \WT, for all $X\in\CC $ we have \linebreak$\rk(X) = \rkE(\underline{X})$.
\end{corollary}
\begin{proof}
This is a consequence of Equation \eqref{eq:obliqua} in the previous proof, and of the discussion preceding it.
\end{proof}

\begin{remark}
  The matroid $(\ES,\rkE)$ is, in some sense an `artificial'
  construct, although in some cases useful. For instance, when $(S,\CC,\rk)$
  is the semimatroid of a periodic arrangement of hyperplanes in
  real space associated to a toric arrangement $\AA$, then $(\ES,\rkE)$ is the matroid of the arrangement
  $\AA_0$ which plays a key role in the techniques used in \cite{CaDe,dAD2,dAD1}.
\end{remark}

\begin{proposition}\label{prop:quot}
  Let $\GS$ be \WT. 
Then 
  $\SS_{\GS}:=(\ES,\underline\CC, \rkE)$ is a locally ranked triple satisfying (CR2). 
\end{proposition}

%    \begin{remark}\label{CiErPr} Before delving into the proof, notice  that a finitary semimatroid satisfies the following stronger version of (CR2).
%\begin{itemize}
%%  \item[(R2')] If $X\cup x\in\CC$ then $\rk(X\cup x)-\rk(X)$ equals 0 or 1.
%%  \item[(CR1')] If $X,Y\in\CC$ and $\rk(X)=\rk(X\cap Y),$ then $X\cup Y\in\CC$ and $\rk(X\cup Y)=\rk(Y).$
%  \item[(CR2')] If $X,Y\in\CC$ and $\rk(X)<\rk(Y),$ then $X\cup y\in\CC$ and $\rk(X\cup y)=\rk(X)+1$ for some $y\in Y-X.$
%\end{itemize}
%The proof of this fact is analogous to that given for the finite case in \cite[Section 2]{Ard}.
%       \end{remark}

\begin{proof}
%[Proof of Proposition \ref{prop:quot}]
  Proposition \ref{prop:unten} implies that (R1), (R2), (R3) hold.
  
  For (CR2), let $A,B\in \CS$ with $\rkE(A)<\rkE(B)$ and
  choose $X\in \oC{A}$ 
  and $Y\in \oC{B}$. Then, by Corollary \ref{cor:rank},
 $\rk(X)<\rk(Y)$. Using (CR2') in $\SS$ (cf.\ Remark \ref{CiErPr}) we find
  $y\in Y\setminus X$ with $X\cup y \in \CC$ and $\rk(X\cup y)>\rk(X)$. Set $b:=\unten{y}$. Then, $A\cup b = \unten{X\cup  y}\in \CS$ and $b\in B\setminus A$ (otherwise $b\in A$, thus -- using Corollary \ref{cor:rank} -- $\rk(X\cup y)=\rkE(A\cup b)=\rkE(A) =\rk(X)$, a contradiction).  
\end{proof}

\section{Translative actions}\label{sec:AAA}
We now proceed towards establishing Theorem \ref{thm:almost}. 
The main idea in this section is to associate a diagram of finite sets and injective maps to every molecule of the quotient triple $\SS_{\GS}$ (see Example \ref{ex:diagram} below). 
In the realizable case, this structure specializes to the inclusion pattern of integer points in semiopen parallelepipeds as well as to that of layers of the associated toric arrangement. In general, these diagrams will allow us in later sections to extend to the general (nonrealizable, non-arithmetic) case some combinatorial decompositions given in \cite{dAM} for realizable arithmetic matroids, most notably Theorem \ref{thm:craponew}.

%We will start with some general considerations on translative actions in Section \ref{sec:pt}. Section \ref{sec:laborb} contains some (rather technical) considerations needed to show that the  above-mentioned are well defined. These diagrams are defined precisely in Section \ref{sec:orbcou}, where we will use them in order to prove the required statements about the function $\mm$, as well as some other facts used later on.  

Recall the definitions in Section \ref{sec:defs}, and in particular that $\GS$ denotes a $G$-semi\-matroid corresponding to the
action of a group $G$ on a semimatroid $\SS=(S,\mathcal C, \rk)$. In
this section we suppose this action always to be cofinite and \ST. In particular, 
we can consider the associated locally ranked triple 
$\SS_\GS = (\ES,\underline{\mathcal C}, \rkE)$ with 
multiplicity function $\mm$. 

\subsection{Maps between sets of ``central orbits''}\label{sec:pt}

\begin{definition}\label{def:WA}
Given $A\in \underline \CC$ %, $X\in\oC{A}$
 and $a_0\in A$ define 
\begin{equation}
  \label{eq:1}
  \setmap{A}{a_0}: \oC{A} \to \oC{A\setminus a_0}, \quad X\mapsto
  X\setminus a_0,
\end{equation}
and notice that, since it is $G$-equivariant,
it induces a function
\begin{equation}
\setmapG{A}{a_0}: \oC{A}/G \to \oC{A\setminus a_0}/G.\label{eq:3}
\end{equation}
\end{definition}

\begin{remark}\label{rem:Winj}
  When
  $\setmap{A}{a_0}$ is injective then
  $\setmapG{A}{a_0}$ also is. This can be seen in many ways -- for instance, by noting that any injective map of $G$-sets is a split monomorphism (e.g., see \cite{ZIM}), and the splitting $G$-map induces a splitting of  $\setmapG{A}{a_0}$.
\end{remark}

\setlist[enumerate,1]{start=1}
\begin{lemma}\label{lem:inj-surj} Let $\GS$ be \ST. 
  \begin{enumerate}[label={(\alph*{})},ref={(\alph*{})}]
  \item If $x_0\in X\in \CC$ with $\rk(X)=\rk(X\setminus x_0)+1$, then $Y
    \cup g(x_0) \in \CC$ for all $g\in G$ and all $Y\in\CC$ with
    $\underline Y = \underline{X\setminus x_0}$. \label{lem:geo-loop}
  \item If $a_0\in A \in \underline\CC$ with $\rkE(A)=\rkE(A\setminus a_0)+1$, then $\setmap{A}{a_0}$ is
    surjective and, for any choice of $x_0\in a_0$,
    a right inverse of $\setmap{A}{a_0}$ is given by
      \begin{equation}
        \setmapS{A}{a_0}: \oC{A\setminus a_0} \to \oC{A},\,\, Y\mapsto
        Y\cup x_0.
        \label{eq:2}   
      \end{equation}
Moreover, $\setmapG{A}{a_0}$ is surjective. In particular,
      $\mm(A)\geq \mm(A\setminus a_0)$. \label{lem:surj}
  \item If $a_0 \in A\in \underline\CC$ with $\rkE(A)=\rkE(A\setminus a_0)$, then $\setmap{A}{a_0}$ is
    injective and thus  
    $\mm(A) \leq \mm(A\setminus a_0)$. \label{lem:inj}
  \end{enumerate}
\end{lemma}

\begin{proof}$\,$
  \begin{itemize}
  \item[\ref{lem:geo-loop}] Let $X,x_0$ be as in the claim. For all $g\in G$ consider
    the central set $g(X)$ of rank $\rk(g(X)) = \rk(X) >
    \rk(X\setminus x_0)$. By (CR2) there is some $y\in g(X) \setminus (X\setminus x_0)$ with
    $y\cup (X \setminus x_0) \in \CC$ and $\rk(y\cup (X \setminus x_0))=\rk(X)$. This $y$ must be $g(x_0)$ because every other
    element $y'\in g(X)\setminus (X\cup g(x_0))$ is of the form $y'=g(x')$ ($\not\in X$)
    for some $x'\in X$, thus $y'\cup (X \setminus x_0) \in \CC$ would imply $\{x',g(x')\}\in \CC$
which, since by construction $x'\neq g(x')$, is forbidden by the fact
that the action is \ST. Thus $(X\setminus x_0) \cup g(x_0) \in \CC$
for all $g\in G$.
    Now consider any $Y$ with $\underline Y = \underline{ X\setminus x_0}$ and notice that     with Lemma \ref{lem:rank} 
 we have the first equality in the following expression $$
 \rk(Y) = \rk(X\setminus x_0) < \rk(X) = \rk((X\setminus x_0) \cup
    g(x_0))$$ (where the inequality holds by assumption and the last equality derives from the choice of $y=g(x_0)$ above). Thus by (CR2) there must be
    $x\in (X\setminus x_0) \cup g(x_0)$ with $Y\cup x\in \CC$ and $\rk(Y\cup x)=\rk(Y)+1$. Since
    $Y$ consists of translates of elements of $X$, as
    above the fact that the action is \ST forces $x=g(x_0)$.
  \end{itemize}
  Towards \ref{lem:surj} and \ref{lem:inj}, choose any $X\in \oC{A}$ and let $x_0\in X$ be
    a representative of $a_0$. By the definition of $\rkE$ (Definition \ref{def:rkE}) and since \STy allows us to apply Corollary \ref{cor:rank}, we conclude that 
    $\rk(X\setminus x_0)=\rk(X)$ 
if and only if $\rkE(A\setminus a_0)=\rkE(A)$.
\begin{itemize}
    \item[\ref{lem:surj}] Suppose 
     $\rkE(A \setminus a_0)=\rkE(A)-1$.
      Part \ref{lem:geo-loop}
    ensures that the function $\setmapS{A}{a_0}$
      is well-defined. Clearly, it is injective and $ \setmap{A}{a_0}
      \circ\setmapS{A}{a_0} = \id$. In particular, $\setmap{A}{a_0}$ is
      surjective. Moreover, if we fix a representative $Y^{\mathcal O}$ of every element $\mathcal O \in \oC{A\setminus a_0}/G$ we see that the assignment 
      \begin{equation} \label{ssms}
      \oC{A\setminus a_0}/G \to \oC{A}/G, \quad
       \mathcal{O} \mapsto G\setmapS{A}{a_0}(Y^{\mathcal{O}})
      \end{equation}
       defines a (noncanonical) section of $\setmapG{A}{a_0}$. This proves surjectivity of  $\setmapG{A}{a_0}$, which implies the stated inequality.

    \item[\ref{lem:inj}] Suppose now $\rkE(A\setminus a_0)=\rkE(A)$ 
    	  and consider $X_1, X_2\in \oC{A}$. 
      Since the action is \ST the sets $X_1\cap
      a_0$ and $X_2\cap a_0$ each consist of a single element, say
      $x_{0,1}$ and $x_{0,2}$ respectively.  If moreover $\setmap{A}{a_0}$ maps both $X_1$, $X_2$
       to the same $Y=X_1\setminus a_0 = X_2\setminus a_0$, then
      $Y\cup x_{0,1}$ and $Y\cup x_{0,2}$ are both central and of the
      same rank, equal to the rank of $Y$. By (CR1) then $Y\cup
      \{x_{0,1},x_{0,2}\} \in \CC$, thus $\{x_{0,1},x_{0,2}\}\in \CC$
      and since the action is \ST we must have
      $x_{0,1}=x_{0,2}$, hence $X_1=X_2$. This proves that $\setmap{A}{a_0}$ is injective and, with Remark \ref{rem:Winj}, the stated inequality. 
  \end{itemize}
\end{proof}

\begin{remark}\label{lk}
More generally, for every $A\in \underline \CC$ and every $A'\subseteq A$ we can consider
$$
\setmap{A}{A'}: \oC{A} \to \oC{A\setminus A'}, \quad X \mapsto X\setminus {\cup A'}
$$
and the associated map $\setmapG{A}{A'}: \frac{\oC{A}}{G} \to \frac{\oC{A\setminus A'}}{G}$.

Notice that, given any enumeration $a'_1,\ldots,a'_l$ of $A'$, we have
$$
\setmap{A}{A'}=\setmap{A}{a'_1} \circ \cdots \circ \setmap{A}{a'_l}
,\quad
\setmapG{A}{A'}=\setmapG{A}{a'_1} \circ \cdots \circ \setmapG{A}{a'_l}.
$$
\end{remark}

\begin{corollary}\label{cor:winj}
For every molecule $(R,F,T)$ of $\SS_{\GS}$,
\begin{itemize}
\item[(a)] $\setmap{R\cup F\cup T}{T}$ and $\setmapG{R\cup F \cup T}{T}$ are injective,
 \item[(b)] $\setmap{R\cup F \cup T}{F}$ and $\setmapG{R\cup F \cup T}{F}$ are surjective.
 \end{itemize}
\end{corollary}

\subsection{Labeling orbits}\label{sec:laborb} The purpose of this section is to provide the groundwork for proving that the objects that will be introduced in Section \ref{sec:orbcou} are well-defined. The reader wishing to acquire a general view of our setup without delving into technicalities may skip this section with no harm.

Our main task here will be to specify canonical representatives for orbits supported on a given molecule, in order for Equation \eqref{eq:2} to
induce a well-defined function between sets of orbits.

Again, we consider
throughout a $G$-semimatroid $\GS$ defined by an action on
$\SS=(S,\CC,\rk)$, and we assume \STy. 
\begin{AN}
For this section we fix a molecule $\mol:= (R,F,T)$ of $\SS_{\GS}$ %and a numbering $F=\{f_1,\ldots,f_k\}$ of the elements of $F$. Moreover, choose 
a linear extension $\prec$ of the partial order defined by inclusion on $2^{F}$, the set of subsets of $F$.\footnote{E.g., 
represent the elements of $2^{F}$ as ordered zero-one-tuples and take the lexicographic order.}. In particular, $I\subseteq I'\subseteq F$ implies $I\preceq I'$.
\end{AN}

\newcommand{\ord}{\lhd}
\begin{definition} 
%Given a molecule $(R,F,T)$ of $\SS_{\GS}$
%fix a numbering $F=\{f_1,...,f_k\}$ 
%If we consider the elements of $2^{[k]}$ as ordered zero-one-tuples we obtain a total order $\prec$ on $2^{[k]}$ by the lexicographic order on the tuples. Notice that, then, $I\subseteq I'$ implies $I\preceq I'$.

We choose representatives $X_R^{(1)}\!\!,\ldots,X_R^{(k_R)}$ of the orbits in $\oC{R}/G$ and extend $\prec$ to a total order on the index set  $\{(i,I) \mid i=1,\ldots, k_R, \, I\in 2^{F}\}$ via 
\begin{equation}\label{eq:ordiJ} 
(i,I)\prec (i',I') \Leftrightarrow\begin{cases} i<i',\\
\mbox{or } i=i' \text{ and } I \prec I'.\end{cases}
\end{equation}

Moreover, choose and fix an element $x_f\in f$ for every $f\in F$.
%$i=1,...,k$. 
%For all $I\in2^{F}$ set $F_I=\{f_i\mid i\in I\}$ and f
Then, for all $F'\subseteq F$ define $X_{F'}=\{x_f\mid f\in F'\}$.
\end{definition} 

We now can recursively define the blocks of an ordered partition of $\oC{R\cup F}/G$ as follows.
\begin{definition} \label{def:12}
Set $\mathscr Y^{(1,\emptyset)}:=\{G (X^{(1)}_R \cup X_F)\}$, and for each $(i,I) \succ (1,\emptyset)$ let
$$
\mathscr Y^{(i,I)}:=\left\{
\oo \in \frac{\oC{R\cup F}}{G} 
\left\vert
\begin{array}{ll}
\textrm{(i)} & \oo \not\in\bigcup_{(j,J)\prec (i,I)} \mathscr Y^{(j,J)} \\[5pt]
\textrm{(ii)} & X_R^{(i)}\cup X_{F\setminus I} \subseteq Y \textrm{ for some }Y\in \oo
\end{array}\right.
\right\}.
$$
Choose an enumeration
$$
\mathscr Y^{(i,I)}=\{\oo_1,\ldots, \oo_{h_{(i,I)}}\}
$$
thereby defining the numbers $h_{(i,I)}$ (and setting $h_{(i,I)}=0$ if $\mathscr Y^{(i,I)}=\emptyset$).
\end{definition}

\begin{remark}
The sets $\mathscr Y ^{(i,I)}$ do partition $\oC{R\cup F}/G$. First,  (i) ensures that they have trivial intersections. Moreover, for every $\mathcal O \in \oC{R\cup F}/G$ there is a unique $i$ with $X^{(i)}_R\subseteq Y$ for some $Y\in \mathcal O$. Now let $I$ be $\prec$-minimal such that the expression in part (ii) holds, and we have $\mathcal O \in \mathscr Y^{(i,I)}$. 
\end{remark}

\begin{remark}\label{rem:zeri}
%If $\oo \in \mathscr Y^{(i,I)}$ then there exists a $Y\in \oo$ with
%$X_{F\setminus F_I}\ssq Y$. Moreover, if 
If $X_{F\setminus J}\ssq Y$ for some $Y\in \oo\in\mathscr Y^{(i,I)}$, then $J\succeq I$. In particular,
$J\subsetneq I$ implies  $X_{F\setminus F_J}\not\subseteq Y$ for all
$Y\in \oo$.  
\end{remark}

Now we are ready to define representatives for orbits in $\oC{R\cup F}/G$.

\begin{definition}\label{def:Yrf}
Define the set 
$$\mathscr Z_{R,F}:=\{(i,I,j) \mid i=1,\ldots,k_R;\, I\in 2^{F};\, j=1,\ldots, h_{(i,I)}\}$$ 
and consider on it the total ordering $\lhd$ given by
$$(i,I,j)\ord (i',I',j')\Leftrightarrow\begin{cases} (i,I) \prec (i',I') \textrm{ or}\\
 (i,I)=(i',I') \textrm{ and } j < j'.
\end{cases}$$
For every $(i,I,j)\in \mathscr Z_{R,F}$ consider the corresponding orbit $\oo_j\in \mathscr Y^{(i,I)}$ and choose a representative $ Y_{R\cup F}^{(i,I,j)}$ of $\oo_j$ with 
\begin{equation}
X_R^{(i)} \cup X_{F\setminus I}\subseteq  Y_{R\cup F}^{(i,I,j)}
\in \oo_j\label{eq:yifr}
\end{equation}
(such a representative exists by requirement (2) of Definition
\ref{def:12}). 
\end{definition}

\begin{lemma}\label{lem:ipreciso}
  We have  $Y^{(i,I,j)}_{R\cup F} \cap X_F = X_{F\setminus I}$.
\end{lemma}
\begin{proof}
  Let $J$ be such that $Y^{(i,I,j)}_{R\cup F} \cap X_F = X_{F\setminus
    J}$. Then $J\subseteq I$ by Equation \eqref{eq:yifr}. Moreover, if
  $J\subsetneq I$ then $J\prec I$, a contradiction to Remark
  \ref{rem:zeri}. Hence $I=J$ as desired.
\end{proof}

For each $F'\subseteq F$ we now fix representatives of the orbits in $\oC{R\cup F'}/G$.

\begin{definition}\label{def:Yrff}
Given  $F'\subseteq F$, for every $\oo\in \oC{R\cup F'}/G$ let 
$$
z(\oo) :=\min_{\lhd} \{z\in \mathscr Z_{R,F} \mid \oo \leq GY_{R\cup F}^z \textrm{ in }\mathcal C_{\GS}\}
$$
and let $Y_{R\cup F'}^{\oo} \in \oo$ be the (unique) representative with 
$$
Y_{R\cup F'}^{\oo} \subseteq Y_{R\cup F}^{z(\oo)}.
$$
With these choices, let
\begin{equation}
\begin{array}{lcll}
\setmapSG{R\cup F}{F\setminus F'}: &
\oC{R\cup F'}/G &\to &\oC{R\cup F}/G \\
& \mathcal{O}
 & \mapsto & G(Y^{\oo}_{R\cup F'}\cup X_{F\setminus F'}).
\end{array}
\label{def:whatund}
\end{equation}
\end{definition}

\begin{lemma}\label{lem:whatcomp}
%Let $\GS$ be \ST, let $(R,F,\emptyset)$ be a molecule of $\SS_{\GS}=(\ES,\underline{\CC}, \rkE)$
%and consider
Let $F''\ssq F'\ssq F$. Then
\begin{itemize}
\item[(a)] for every $\oo \in \oC{R\cup F'}/G$
$$
Y^{z(\oo)}_{R\cup F} = Y^{\oo}_{R\cup F'} \cup X_{F\setminus F'};
$$
\item[(b)] for every $\oo \in \oC{R\cup F''}/G$ 
$$
Y_{R\cup F'}^{G(Y^{\oo}_{R \cup F''}\cup X_{F'\setminus F''})} = Y^{\oo}_{R\cup F''} \cup X_{F'\setminus F''}.
$$
\item[(c)] Furthermore, 
$$\setmapSG{R\cup F}{F\setminus F'}\circ\setmapSG{R\cup F'}{F'\setminus F''}=\setmapSG{R\cup F}{F\setminus F''}.$$
\end{itemize}
\label{ref:Comb}\end{lemma}

\begin{proof} In this proof, given any $\oo \in \oC{R\cup F'}/G$ let us for brevity call $\mathscr Z(\oo)$ the set over which the minimum is taken in Definition \ref{def:Yrff} in order to define $z(\oo)$. 

  \begin{enumerate} %[itemindent=-1em]
  \item[(a)] It is enough to show that
    $X_{F\setminus F'}\ssq Y^{z(\oo)}_{R\cup F}$. In order to prove
    this, we consider 
$$
Y':= (Y^{z(\oo)}_{R\cup F}\setminus X')\cup X_{F\setminus F'}
$$
where $X' \in \oC{F\setminus F'}$ is defined by $X'\subseteq
Y^{z(\oo)}_{R\cup F}$ (notice that $\vert X' \vert = \vert X_F \vert $ since $Y^{z(\oo)}_{R\cup F} \in \oC{R\cup F}$ and the action is \ST). The set $Y'$ is central by Lemma
\ref{lem:inj-surj}.(a), because $\rk(Y^{z(\oo)}_{R\cup F'}) =
\rk(Y^{z(\oo)}_{R\cup F}\setminus X') + \vert X' \vert $. Moreover,
$GY' \geq \oo$ in $\CC_{\GS}$ since $Y^{\oo}_{R\cup F} \subseteq Y'$.

If $X_{F\setminus F'} \subseteq Y^{z(\oo)}_{R\cup F}$, then
$Y'=Y^{z(\oo)}_{R\cup F}$ and we are done. We will prove that if this is not the case,
then $z(\oo)\neq \min \mathscr Z (\oo)$, reaching a contradition.
Suppose then $X_{F\setminus F'} \not\subseteq Y^{z(\oo)}_{R\cup F}$, and
write $z(\oo)=(i,I,j)$.
By Lemma \ref{lem:ipreciso}, we have $I=\{f\mid x_{f}\notin Y_{R\cup F}^{z(\oo)}\}$. Hence, setting
$$I_{Y'}:=\{f \mid x_{f}\notin Y'\}$$
we have that $I_{Y'}=I\cap F'\subseteq I$, where  the last containment is strict (otherwise $Y'= Y^{z(\oo)}_{R\cup F}$, hence $X_{F\setminus F'} \subseteq Y^{z(\oo)}_{R\cup F}$, contrary to our assumption). By definition, $I_{Y'} \subsetneq I$ implies $I_{Y'} \precneq I$.
Moreover, for $z'=(i,I',j')$ defined by $GY'=\oo_{j'}\in\mathscr Y^{(i,I')}$
we have in fact by Remark \ref{rem:zeri} that $ I' \preceq
I_{Y'}$. Therefore, $I'\preceq I_{Y'}\precneq I$.
This implies that $z'=(i,I',j') \unlhd (i,I,j) =z(\oo)$ and $z'\neq
z(\oo)$. Thus, $GY^{z'}_{R\cup F} \in \mathscr Z (\oo)$ but $z'$
strictly precedes $z(\oo)$, and we reach the annouced contradiction.

\item[(b)] Let $\oo$ be as in the claim, and set $\mathcal U :=
G(Y^{\oo}_{R \cup F''}\cup X_{F'\setminus F''})$. Then $\oo \leq
\mathcal U$ in $\CC_{\GS}$, thus $\mathscr Z (\oo) \supseteq \mathscr
Z (\mathcal U)$ and therefore $z(\oo) \unlhd z(\mathcal U)$. Now,
since $Y^{z(\oo)}_{R\cup F}= Y^{\oo}_{R\cup F''} \cup
X_{F\setminus F''}$ by part (a), we see that $\mathcal U \leq
GY^{z(\oo)}_{R\cup F}$ in $\CGS$, thus $z(\mathcal U) \unlhd
z(\oo )$. In summary, $z(\mathcal U) = z(\oo )$ and, as a subset of
$Y^{\oo}_{R\cup F''} \cup X_{F\setminus F''}$, we see that
$Y^{\mathcal U}_{R\cup F'} =Y^{\oo}_{R\cup F''} \cup X_{F'\setminus
  F''}$ as claimed.

\item[(c)] For every $\oo \in \oC{R \cup F''}/G$ we compute
\begin{center}
 $ \setmapSG{R\cup F}{F\setminus F'}\circ\setmapSG{R\cup F'}{F'\setminus F''}(\oo)=\setmapSG{R\cup F}{F\setminus F'}(G(Y^{\oo}_{R\cup F''} \cup X_{F'\setminus F''} )) $
 $ =G(Y_{R\cup F'}^{G(Y^{\oo}_{R \cup F''}\cup X_{F'\setminus F''})}\cup
  X_{F\setminus F'})
  =G(Y^{\oo}_{R\cup F''}\cup X_{F\setminus F''})=\setmapSG{R\cup
    F}{F\setminus F''}(\oo)$,\end{center}
where in the third equality we used (b) and all other equalities hold by definition.
\end{enumerate}
\end{proof}

\begin{corollary}\label{cor:whatinj}
For every $F'\subseteq F$, the function $\setmapS{R\cup F}{F\setminus F'}$ is injective.
\end{corollary}
\begin{proof}
Let $f_1,\ldots,f_m$ be an enumeration of the elements of $F\setminus F'$ and for every $j=1,\ldots, m$ set $F_j:= F' \cup \{f_1,\ldots,f_j\}$. Then by Lemma \ref{lem:whatcomp}.(c) 
$$\setmapS{R\cup F}{ F \setminus F'} =
\setmapS{R\cup F}{ f_m} \circ \setmapS{R\cup F_{m-1}}{ f_{m-1}} \cdots \circ \setmapS{R\cup F_2}{ f_1} $$
and each of the functions on the right-hand side is injective because it is an instance of the function described in Equation \eqref{ssms}. The latter is 
%injective because it is a section of  (in fact these maps are 
used as a left-inverse to prove the surjectivity claim of Lemma \ref{lem:inj-surj}.(b) and, as such, is injective.
\end{proof}

\begin{definition}\label{def:Yrft}
Given $F'\ssq F$, $T'\ssq T$, as a representative of the orbit $\oo\in\oC{R\cup F'\cup T'}/G$ we choose
\begin{equation}\label{eq:repdef}
Y^\oo_{R\cup F' \cup T'}:= \setmap{R\cup F' \cup T'}{ T'}^{-1}(Y^{\setmapG{R\cup F'\cup T'}{T'}(\oo)}_{R\cup F'})
\end{equation}

as the representative of $\mathcal O$, and let 
$Y^{\mathcal O}_{T'}:= Y^\oo_{R\cup F' \cup T'} \setminus Y^{\setmapG{R\cup F'\cup T'}{T'}(\oo)}_{R\cup F'} $

%
%
%
%  representative
%$Y^\oo_{R\cup F' \cup T'}\in\oo$ as
%\begin{equation}
%Y^\oo_{R\cup F' \cup T'}=Y^\oo_{R\cup F'}\cup Y^\oo_{T'}=X^\oo_R\cup Y_{F'}^\oo\cup Y^\oo_{T'},
%\label{eq:Yrft}
%\end{equation}
%where $Y^\oo_{R\cup F'}$ is given as above
%and $Y^\oo_{T'}\in\oC{T'}$ is uniquely determined by $X^\oo_R$ since $\setmapG{R\cup T'}{T'}$ is injective by Lemma \ref{lem:inj-surj}.\ref{lem:inj} and Remark \ref{rem:Winj}.
\end{definition}

\begin{remark} In order to prove that $Y^{\mathcal O}_{R\cup F'\cup T'}$ is well defined, we have to show that the right-hand side of Equation \eqref{eq:repdef} is not empty; uniqueness will then follow from injectivity of $\setmap{R\cup F' \cup T'}{ T'}$ (see Corollary \ref{cor:winj}). To see this, it is enough to notice that the function $\setmap{R\cup F' \cup T'}{ T'}$  is onto when restricted to $\setmapG{R\cup F'\cup T'}{T'}(\oo)$, which by definition is $\setmap{R\cup F'\cup T'}{T'}(\oo)$, clearly part of the image of $\setmap{R\cup F'\cup T'}{T'}(\oo)$.
\end{remark}

\begin{example}\label{ex:choice} We go back to our running example (Example \ref{ex:running1}), for which we depict in Figure \ref{fig:OtherFund} a piece of the associated periodic arrangement,
and consider there the molecule $(\emptyset,F,\emptyset)$, where $F=\{f_a,f_b\}$ is the set of orbits of the orange and green lines. 

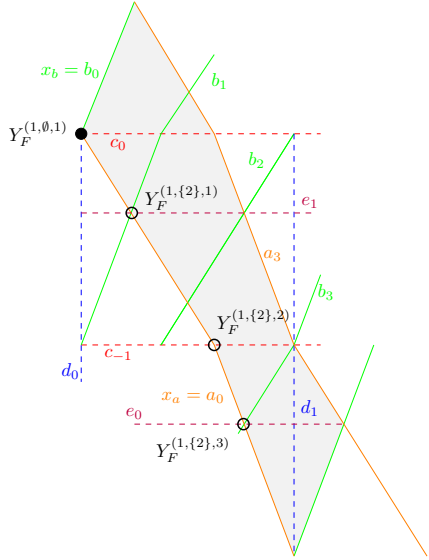
\begin{figure}[h]
\scalebox{.7}{%
\begin{tikzpicture}
    \node at (0,0) (O) {};
    \node at (4,0) (E) {};
    \node at (4,4) (NE) {};
    \node at (0,4) (N) {};
    \node at (-4,0) (W) {};
    \node at (-4,-4) (SW) {};
    \node at (0,-4) (S) {};
    \node at (4,-4) (SE) {};
    \node [label=left:{$Y_F^{(1,\emptyset,1)}$}] at (-4,4) (NW) {};
    \node at (2.5,0) (Er) {};
    \node at (1.5,0) (El) {};
    \node at (-2.5,0) (Wl) {};
    \node at (-1.5,0) (Wr) {};   
    \node at (-2.5,4) (NWl) {};
    \node at (-1.5,4) (NWr) {};    
    \node at (-2.5,-4) (SWl) {};
    \node at (-1.5,-4) (SWr) {};    
    \node at (2.5,4) (NEr) {};
    \node at (1.5,4) (NEl) {};
    \node at (2.5,-4) (SEr) {};
    \node at (1.5,-4) (SEl) {};
    \node at (-4,-1.5) (EW) {};
    \node at (4,-1.5) (EE) {};    
    \fill [fill=gray, fill opacity=0.1] (-3,6.5) -- (NWr.center)
    -- (O.center) -- (.95,-1.5) -- (S.center) -- (Wr.center) -- (NW.center) -- cycle;
    \draw [dashed,red] (W.center) -- (.5,0);
 %   \draw [-,red] (SW.center) -- (SE.center);
    \draw [dashed,red] (NW.center) -- (.5,4);    
%    \draw [-,green] (SW.center) -- (Wl.center);
    \draw [-,green] (Wl.center) -- (N.center) ;
    \draw [-,green] (-1.05,-1.68) -- (O.center) -- (.5,1.33);
    \draw [-,green] (S.center) -- (El.center); % -- (NE.center);
    \draw [-,green] (Wl.center) -- (N.center);
%    \draw [-,green] (SEl.center) -- (E.center);   
    \draw [green] (NW.center) -- (-3,6.5) ;
    \draw [-,green] (W.center) -- (NWl.center) -- (-1.5,5.5);    
    \draw [dashed,blue] (N.center) -- (0,-4);
  %  \draw [-,blue] (NE.center) -- (SE.center);
    \draw [dashed,blue] (NW.center) -- (-4,-.7);                
    \draw [-,orange] (NW.center) -- (Wr.center) -- (S.center); %(-.87,-1.68);
    \draw [-,orange] (-3,6.5) -- (NWr.center) -- (O.center) -- (SEr.center);
  %  \draw [-,orange] (N.center) -- (Er.center) -- (SE.center);
  %  \draw [-,orange] (NEr.center) -- (E.center);
%    \draw [-,orange] (W.center) -- (SWr.center);
    \draw [dashed,purple] (-3,-1.5) -- (1,-1.5);
    \draw [dashed,purple] (-4,2.5) -- (.4,2.5);    
%    \node[text=red] at (-3.3,-4.2) (c0) {$c_0$}; 
    \node[text=red] at (-3.3,-.2) (c1) {$c_{-1}$}; 
    \node[text=red] at (-3.3,3.8) (c2) {$c_0$}; 
    \node[text=blue] at (-4.2,-.5) (d0) {$d_0$};     
    \node[text=blue] at (.3,-1.2) (d1) {$d_1$}; 
  %  \node[text=blue] at (3.8,-1.2) (d2) {$d_2$}; 
    \node[text=green] at (-.7,3.5) (b0) {$b_2$}; 
    \node[text=green] at (-1.4,5) (bm) {$b_{1}$}; 
    \node[text=green] at (.6,1) (b1) {$b_3$}; 
    \node[text=green] at (-4.2,5.2) (b2) {$x_b=b_0$}; 
%    \node[text=green] at (2.7,-2.7) (b3) {$b_3$}; 
%    \node[text=orange] at (-2.6,-2.6) (c0) {$a_1$}; 
    \node[text=orange] at (-1.9,-1) (c0) {$x_a=a_0$}; 
    \node[text=orange] at (-.4,1.7) (c0) {$a_3$}; 
  %  \node[text=orange] at (2.6,-1) (c0) {$a_4$}; 
  %  \node[text=orange] at (3.5,.5) (c0) {$a_5$};
    \node[text=purple] at (.3,2.7) (e1) {$e_1$};
    \node[text=purple] at (-3,-1.3) (e2) {$e_0$}; 
    \filldraw (NW) circle (3pt);
        \draw [thick] (-3.05,2.5) circle (3pt);
        \node [label=above right:{$Y_F^{(1,\{2\},1)}$}] at (-3.05,2.3) (O1) {};
        \draw [thick] (Wr) circle (3pt);
        \draw [thick] (-.95,-1.5) circle (3pt);
        \node [label=below left:{$Y_F^{(1,\{2\},3)}$}] at (-.95,-1.5) (O3) {};
        \node [label=above right:{$Y_F^{(1,\{2\},2)}$}] at (-1.7,-.05) (O2) {};    
\end{tikzpicture}
}
\caption{An illustration for Example \ref{ex:choice}.}
\label{fig:OtherFund}
\end{figure}

Choose representatives $x_{a}=a_0$ for the orange lines, $x_b=b_0$ for the green lines and denote their $(0,k)$-translate by $a_k$ (resp. $b_k$).

By Definitions \ref{def:12} and \ref{def:Yrf}, we get the following partition of $\oC{F}/G,$  
$$\mathscr Y^{(1,\emptyset)}=\{\oo_0\},\ \mathscr Y^{(1,\{2\})}=\{\oo_1,\oo_2,\oo_3\},\ \mathscr Y^{(1,\{1\})}=\mathscr Y^{(1,\{1,2\})}=\emptyset,$$ with representatives 
$$Y_F^{(1,\emptyset,1)}\mkern-4mu=\{a_0,b_0\}, 
Y_F^{(1,\{2\},1)}\mkern-4mu=\{a_0,b_{k_1}\}, 
Y_F^{(1,\{2\},2)}\mkern-4mu=\{a_0,b_{k_2}\},
Y_F^{(1,\{2\},3)}\mkern-4mu=\{a_0,b_{k_3}\},$$
where $k_1\neq 0\! \mod 4;$ $k_2\neq 0,k_1\!\mod 4;$ and $k_3\neq 0,k_1,k_2\!\mod 4.$
Without loss of generality, one could assume $k_1=1$, $k_2=2$, $k_3=3$, and we get the situation depicted in Figure \ref{fig:OtherFund}. 
Moreover, by Definition \ref{def:Yrff} we get $Y_a^{\oo_a}=a_0$ (where $\oC{f_a}/G=\{\oo_a\}$), $Y_b^{\oo_b}=b_0$ (where $\oC{f_b}/G=\{\oo_b\}$), and $Y_\emptyset^{\emptyset}=\emptyset$ where $\oC{\emptyset}/G=\{\emptyset\}.$

Thus, $$\setmapSG{F}{F}(\emptyset)=\setmapSG{F}{f_b}(\setmapSG{f_a}{f_a}(\emptyset))=\setmapSG{F}{f_a}(\setmapSG{f_b}{f_b}(\emptyset))=G(a_0b_0)=\oo_0.$$

Notice that an accurate choice of representatives is of the
essence. For example, choosing $Y^{\oo_a}_{a}=a_0$ and $Y^{\oo_b}_b=b_{1}$
as representatives of $\oo_a$, resp.\ $\oo_b$, $$\im\setmapSG{F}{f_b}=G(a_0x_b)=G(a_0b_0)\neq G(a_0b_1)= G(x_{a}b_{1})=\im\setmapSG{F}{f_a}.$$
\end{example}

\begin{lemma}\label{lem:comm}
For every $F'\subseteq F$ and $T'\subseteq T$, 
%the following square commutes.

$$
\setmapSG{R\cup F}{F\setminus F'} \circ \setmapG{R\cup F' \cup T'}{T'}
= \setmapG{R\cup F \cup T'}{T'}
\circ \setmapSG{R\cup F \cup T'}{F\setminus F'}.
$$

%\begin{center}
%\begin{tikzpicture}
%\node (NW) at (-3,1) {$\displaystyle{\frac{\oC{R\cup F' \cup T'}}{G}}$};
%\node (NE) at (3,1) {$\displaystyle{\frac{\oC{R\cup F'}}{G}}$};
%\node (SW) at (-3,-1) {$\displaystyle{\frac{\oC{R \cup F \cup T'}}{G}}$};
%\node (SE) at (3,-1) {$\displaystyle{\frac{\oC{R\cup F}}{G}}$};
%\draw [->] (NW) -- (NE);
%\node[anchor=south] (N) at (0,1) {$\setmapG{R\cup F' \cup T'}{T'}$};
%\draw [->] (NE) -- (SE);
%\node[anchor=west] (E) at (3,0) {$\setmapS{R\cup F}{F\setminus F'}$};
%\draw [->] (NW) -- (SW);
%\node[anchor=east] (W) at (-3,0) {$\setmapS{R\cup F \cup T'}{F\setminus F'}$};
%\draw [->] (SW) -- (SE); 
%\node[anchor=south] (S) at (0,-1) {$\setmapG{R\cup F \cup T'}{T'}$};
%\end{tikzpicture}
%\end{center}

\end{lemma}

\begin{proof}
We check equality on every $\oo\in\oC{R\cup F' \cup T'}$. On the right-hand side, using the definitions, we find
$$
\setmapG{R\cup F \cup T'}{T'}(\setmapSG{R\cup F \cup T'}{F\setminus F'}(\oo))
$$
$$=\setmapG{R\cup F \cup T'}{T'} (G(Y^\oo_{R\cup F'\cup T'} \cup X_{F\setminus F'})) = G((Y^\oo_{R\cup F'\cup T'}\setminus \cup T' ) \cup X_{F\setminus F'})
$$
while on the left-hand side we compute
$$
\setmapSG{R\cup F}{F\setminus F'} ( \setmapG{R\cup F' \cup T'}{T'} (\oo)) $$
$$= G (Y_{R\cup F'}^{\setmapG{R\cup F' \cup T'}{T'} (\oo)}\cup X_{F\setminus F'}) = G((Y_{R\cup F' \cup T'}^{\oo}\setminus \cup T')\cup X_{F\setminus F'})
$$
where the last equality uses Definition \ref{def:Yrft}.
\end{proof}

\subsection{Orbit count for molecules}\label{sec:diagrams}\label{sec:orbcou}
\begin{definition}
  Given a molecule $(R,F,T)$ of a ranked triple, define the following boolean poset
\begin{align*}
P[R,F,T]&:=\{(F',T') \mid F'\subseteq F,\, T'\subseteq T\} \text{ with order}\\ (F',T')&\leq(F'',T'') \Leftrightarrow F' \subseteq F'', \, T'\supseteq T''.
\end{align*}
Thus, the maximal element is $(F,\emptyset)$ and the minimal element $(\emptyset, T)$.
\end{definition}

%Consider now a \ST $G$-semimatroid $\GS$ and a molecule $(R,F,T)$ of
%$\SS_{\GS}$. Let $A\in[R,R\cup F\cup T]$, say $A=R\cup F'\cup T'$. 
%Recall that for every $t\in T'$ we have an injective function $$\setmapG{A}{t}: \oC{A}/G \to \oC{A\setminus t}/G$$ by Lemma \ref{lem:inj-surj}.\ref{lem:inj}, and for all $f\in F\setminus F'$ we have the injective function $$\setmapSG{R\cup F'\cup f}{f}: \oC{R\cup F'}/G \to \oC{R\cup F'\cup f}/G$$ by Equation 
%\eqref{def:whatund} in Definition \ref{def:Yrff}. 

 \begin{definition} Let $\GS$ be a \ST $G-$semimatroid and ${\mol}:=(R,F,T)$ be a molecule of $\SS_{\GS}$. By composing the above functions we obtain, for every $(F',T')\in P[R,F,T]$, a function 
\begin{equation}
f^{\mol}_{(F',T')}:= \widehat{\underline{w}}_{R\cup F, F'} \circ \underline{w}_{R\cup F'\cup T', T'}
\label{eq:fctf}\end{equation}
\label{def:fctf}
\end{definition}

\begin{remark}
Explicitly,
$$\begin{array}{lcll}
 f^{\mol}_{(F',T')}
: & \oC{R\cup F' \cup T'}/G &\to &\oC{R\cup F}/G, \\
 & \oo & \mapsto & G((Y^{\oo}_{R \cup F'\cup T'}\setminus\cup T')\cup X_{F\setminus F'}).
\end{array}$$

\end{remark}

\begin{remark}\label{rem:BDMF} 
%Let $A\in[R,R\cup F\cup T]$, say $A=R\cup F'\cup T'$. 
The functions $f^{\mol}_{(F',T')}$ are
  %well-defined by Lemma \ref{lem:whatcomp}.(c), and 
  injective by  Corollary \ref{cor:winj} and Corollary \ref{cor:whatinj}.
  % and Remark \ref{rem:Winj}. 
  In particular, with $A:=R\cup F'\cup T'$,
%Moreover, injectivity implies that  
\begin{center}
$\displaystyle{\mm(A)= \left\vert\frac{\oC{R\cup F' \cup T'}}{G} \right\vert = \vert \im f^{\mol}_{(F',T')}\vert}$.
\end{center}
\end{remark}

\begin{example}\label{ex:diagram} In the context of our running example, Example \ref{ex:running1}, we have that $\mol:=(\emptyset,\{a,b\}, \emptyset)$ is a molecule of $\SS_{\GS}$. Figure \ref{fig:Pmaps} depicts the associated poset and maps.
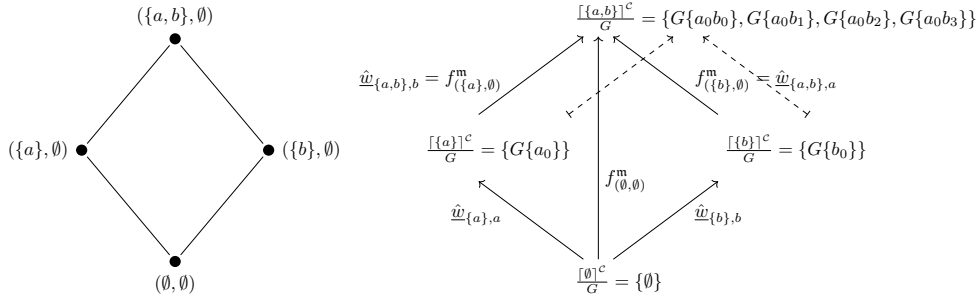
\begin{figure}[h]
\scalebox{.7}{
\begin{tikzpicture}[x=5em,y=6em]
    \node at (0,-1) (S) {};
    \node at (1,0) (E) {};
    \node at (-1,0) (W) {};
    \node at (0,1) (N) {};
    \fill (N) circle (3pt);
    \fill (S) circle (3pt);    
    \fill (E) circle (3pt);
    \fill (W) circle (3pt);
    \draw [-] (N) -- (W) -- (S) -- (E) -- (N);
    \node [label=below:{$(\emptyset,\emptyset)$}] at (S) {};
    \node [label=left:{$(\{a\},\emptyset)$}] at (W) {};
    \node [label=right:{$(\{b\},\emptyset)$}] at (E) {};
    \node [label=above:{$(\{a,b\},\emptyset)$}] at (N) {};
\end{tikzpicture}
\begin{tikzpicture}[x=8em,y=6em]
    \node [anchor=south west] at (0,1) (N) {$\frac{\oC{\{a,b\}}}{G}=\{G\{a_0b_0\},G\{a_0b_1\},G\{a_0b_2\},G\{a_0b_3\}\}$};
    \node [anchor=west] at (1,0) (E) {$\frac{\oC{\{b\}}}{G}=\{G\{b_0\}\}$};
    \node [anchor=west] at (-1,0) (W) {$\frac{\oC{\{a\}}}{G}=\{G\{a_0\}\}$};
    \node [anchor=north west] at (0,-1) (S) {$\frac{\oC{\emptyset}}{G}=\{\emptyset\}$};
    \draw [->] (-.6,.3) -- (.1,1);
    \draw [|->, dashed] (0,.3) -- (.7,1);
    \draw [->] (1,.3) -- (.3,1);
    \draw [|->, dashed] (1.6,.3) -- (.9,1);
    \draw [->]  (.1,-1) -- (-.6,-.3);    
    \draw [->]  (.3,-1) -- (1,-.3);
    \draw [->]  (.2,-1) -- (.2,1);
    \node [anchor=east] at (-.4,.6) (fab) {$\hat{\underline{w}}_{\{a,b\},b}=f^{\mol}_{(\{a\},\emptyset)}$};
    \node [anchor=west] at (.8,.6) (fba) {$f^{\mol}_{(\{b\},\emptyset)}=\hat{\underline{w}}_{\{a,b\},a}$};
    \node [anchor=east] at (-.4,-.6) (fa) {$\hat{\underline{w}}_{\{a\},a}$};
    \node [anchor=west] at (.8,-.6) (fb) {$\hat{\underline{w}}_{\{b\},b}$};
    \node [anchor=west] at (.2,-.3) (fo) {$f^{\mol}_{(\emptyset,\emptyset)}$};
\end{tikzpicture}
}
\caption{The Hasse diagram of the poset $P[\emptyset, \{a,b\}, \emptyset]$ in the context of Example \ref{ex:running1} and, on the right-hand side, the associated diagram of sets.}
\label{fig:Pmaps}
\end{figure}
\end{example}

\begin{lemma}\label{lem:mocposet} Let $\GS$ be \ST and consider a
  molecule $\mol:=(R,F,T)$ of $\SS_{\GS}$.
\begin{enumerate}[label={(\alph*{})},ref={(\alph*{})}]
\item  For $(F',T'),(F'',T'')\in P[R,F,T]$ we have 
\begin{align*}
\im(f^{\mol}_{(F',T')\wedge(F'',T'')})&=\im(f^{\mol}_{(F'\cap F'',T'\cup T'')})\\
&=\im f^{\mol}_{(F',T')}\cap\im f^{\mol}_{(F'',T'')}.
\end{align*}\label{lem:bildf}\vspace{-.5cm}
\item In particular, $$\im f^{\mol}_{(F',T')}\ssq\im f^{\mol}_{(F'',T'')}\text{ if }(F',T')\leq (F'',T'').$$\label{lem:bildf2}\vspace{-.3cm}
\item The function
$$m_{\GS} : P[R,F,T] \to \mathbb N,\quad (F',T') \mapsto m_{\GS}(R\cup F'\cup T')$$
is (weakly) increasing.\label{lem:increasing}
\end{enumerate}
\end{lemma}

\begin{proof} 
%The function $f_{(F',T')}^R$ is by Definition
%  \ref{def:fctf} a composition of functions of the type exhibited in
%  Equations \eqref{eq:3} and \eqref{def:whatund}. Thus using Lemma
%  \ref{ref:Comb}.(c) we obtain part \ref{lem:bildf}. 

    Part \ref{lem:bildf2} is an immediate consequence of \ref{lem:bildf} and by Remark \ref{rem:BDMF} it implies \ref{lem:increasing}. Thus it is enough to prove part (a), where the first equality is the definition of greatest lower bound in $P[R,F,T]$. We turn then to the second equality and consider the following diagram, where we write $F^*:= F'\cap F''$ and $T^*:= T'\cup T''$.
    \begin{center}
    \def\stile{}
    \def\picc{\scriptstyle}
\begin{tikzpicture}
\node (NW) at (-4,1) {$\stile{\frac{\oC{R\cup F^* \cup T^*}}{G}}$};
\node (NE) at (4,1) {$\stile{\frac{\oC{R\cup F' \cup T'}}{G}}$};
\node (SW) at (-4,-1) {$\stile{\frac{\oC{R \cup F'' \cup T''}}{G}}$};
\node (SE) at (4,-1) {$\stile{\frac{\oC{R\cup F}}{G}}$};
\draw [->] (NW) -- (NE);
\node[anchor=south] (N) at (0.4,1) {$\picc{\setmapSG{R\cup F'}{F'\setminus F''}\circ
\setmapG{R\cup F^* \cup T^*}{T''\setminus T'}}$};
\draw [->] (NE) -- (SE);
\node[anchor=west] (E) at (4,0) {$\picc{f^{\mol}_{(F',T')}}$};
\draw [->] (NW) -- (SW);
\node[anchor=east] (W) at (-4,0) {\begin{minipage}{.22\textwidth}$\picc{\setmapSG{R\cup F'' }{F''\setminus F'}} $\\
\hspace*{1em}$\picc{\circ \setmapG{R\cup F^*\cup T^*}{T'\setminus T''}}$\end{minipage}};
\draw [->] (SW) -- (SE); 
\node[anchor=south] (S) at (0,-1) {$\picc{f^{\mol}_{(F'',T'')}}$};
\draw [->] (NW) -- (SE); 
\node[anchor=west] (M) at (0,0.2) {$\picc{f^{\mol}_{(F'\cap F'',T' \cup T'')}}$};
\end{tikzpicture}
\end{center}
This diagram is commutative by Lemma \ref{lem:comm} and Remark \ref{lk}, and we conclude that $\im(f^{\mol}_{(F'\cap F'',T'\cup T'')})\subseteq
\im f^{\mol}_{(F',T')}\cap\im f^{\mol}_{(F'',T'')}$. For the reverse inclusion 
let $\oo\in \im f^{\mol}_{(F',T')}\cap\im f^{\mol}_{(F'',T'')}$ and choose $\oo'\in \frac{\oC{R \cup F' \cup T'}}{G}$, $\oo''\in \frac{\oC{R \cup F'' \cup T''}}{G}$ such that 
 $\oo=f^{\mol}_{(F',T')} (\oo') = f^{\mol}_{(F'',T'')} (\oo'')$. In particular,
 $$
 Y^\oo = Y^{\oo'}\setminus Y^{\oo'}_{T'} \cup X_{F\setminus F'} 
 = Y^{\oo''}\setminus Y^{\oo''}_{T''} \cup X_{F\setminus F''}. 
 $$
 Thus, the set 
 $$\widehat{Y}:= (Y^{\oo'}\setminus Y^{\oo'}_{T'}) \cap (Y^{\oo''} \setminus Y^{\oo''}_{T''})$$
 is a subset of $Y^{\oo}$, hence it is a central set and generates an orbit
 $$
 \widehat{\oo} := G\widehat{Y} \in \frac{\oC{R\cup (F'\cap F'')}}{G}.
 $$
 Since $Y^\oo = \widehat{Y} \cup X_{F\setminus (F'\cap F'')}$, Lemma \ref{lem:whatcomp}.(a) implies $z(\oo)=z(\widehat{\oo'})$, hence $\widehat{Y}=Y^{\widehat{\oo}}$ by definition of preferred representatives. 
 
 Now notice that $Y^{\widehat{\oo}} \cup Y_{T'}^{\oo'}$ is central because it is a subset of $Y^{\oo'}$. Similarly, also $Y^{\widehat{\oo}} \cup Y_{T''}^{\oo''}\subseteq Y^{\oo''}$ is central. Moreover, $\rk(Y^{\widehat{\oo}} \cup Y_{T''}^{\oo'})=\rk(Y^{\widehat{\oo}})=\rk(Y^{\widehat{\oo}} \cup Y_{T'}^{\oo'})$ and thus, by (CR1) (see Definition \ref{def:FS}),
% 
% NOTICE $$
% Y^{\widehat{\oo}} \cup Y_{T'}^{\oo'} = Y^{\oo'} \setminus X_{F'\setminus F''}
% $$
% is central and maps to $\widehat{\oo}$ under $\setmapG{}{}$ same for $T''$. By (CR1), 
$Y^{\widehat{\oo}} \cup Y^{\oo'}_{T'}\cup Y^{\oo''}_{T''}$ is central, and we can compute $$
 f^{\mol}_{(F'\cap F''), (T'\cup T'')}( G( Y^{\widehat{\oo}} \cup Y^{\oo'}_{T'}\cup Y^{\oo''}_{T''} )) = G(Y^{\widehat{\oo}}\cup X_{F\setminus (F'\cap F'')}) = GY^\oo = \oo
 $$
 proving $\oo\in \im f^{\mol}_{(F'\cap F''), (T'\cup T'')}$, as was to be shown.
 \end{proof}

\begin{definition}\label{def:lezete}
  Let $\mol :=(R,F,T)$ be a molecule of $\SS_{\GS}$. For every $(F',T')\in P[R,F,T]$ define the
  sets
  $$
  Z^{\mathfrak m} (F',T'):= \im f^{\mathfrak m}_{(F',T')},\quad \overline{Z}^\mol (F',T'):=
  Z^\mol (F',T')\setminus \bigcup_{(F'',T'')<(F',T')} Z^\mol (F'',T''),
  $$
  and let $$\nm(F',T'):= \vert \overline{Z}^\mol (F',T')\vert.$$
\end{definition}

The following equality holds then by Lemma \ref{lem:mocposet}.(a).
\begin{equation}
\mm(R\cup T'\cup F') = \vert \im f^\mol _{(F',T')}\vert = \sum_{p\leq (F',T')} \nm (p)\label{eq:MI}
\end{equation}

\begin{lemma}\label{lem:PP}
  If $\GS$ is \ST,  
  then for every molecule $\mol:= (R,F,T)$  in $\SS_\GS$ we have
  $$
  \rho(R,R\cup F\cup T) = \nm (F,\emptyset).
  $$
\end{lemma}

\begin{proof}
  Let $(R,F,T)$ be a molecule in $\SS_\GS$ and in this proof let us write $P$ for $P[R,F,T]$. We start by rewriting Definition
  \ref{def:rho} as a sum over elements of $P$ as follows.
\begin{align*}
\rho(R,R\cup F\cup T)&:= (-1)^{\vert T \vert}
\sum_{R\subseteq A\subseteq R_1} (-1)^{\vert R\cup F\cup T\vert -
  \vert A \vert} \mm(A) \\
&=
\sum_{F'\subseteq F}\sum_{ T'\subseteq T}
(-1)^{\vert F\setminus F' \vert + \vert T' \vert} \mm(R\cup F' \cup T')
\end{align*}

Then the poset $P$ has
rank function
$\rk(F',T')=\vert F' \vert + \vert T\setminus T'\vert$, and by
M\"obius inversion (where we call $\mu_P$ the M\"obius function of $P$) we can write explicitly the value of
$\nm (F,\emptyset)$ from Equation \eqref{eq:MI}. 

\begin{align*}
\nm (F,\emptyset)&=
\sum_{(F',T')\in P} \mu_{P}((F',T'), (F,\emptyset)) \mm(R\cup T'\cup
F')\\
&=\sum_{A\in [R,R\cup F\cup T]} 
(-1)^{\vert F \vert + \vert T \vert - \vert F'\vert - \vert T\setminus
T'\vert} \mm(A)\\
&=\sum_{A\in [R,R\cup F\cup T]} 
(-1)^{\vert F \setminus F' \vert + \vert T' \vert} \mm(A) \\
&=\rho(R,R\cup F\cup T)
\end{align*}

\end{proof}

Since the function $\nm$ is - by definition - never negative,
as an easy corollary we obtain the following.

\begin{proposition}\label{lem:P}
  If $\GS$ is \ST,  
  then the pair $(\SS_\GS,\mm)$
  satisfies property (P) of
  Definition \ref{def:AM} (and is thus called ``pseudo-arithmetic'').
\end{proposition}

\begin{definition}\label{def:eta}
  Fix $A\subseteq \ES$, recall the poset $\PS$ from Definition \ref{def:LS} and define the function 
$$
\eta_A: \CC_{\GS} \to \mathbb N,\quad \eta_A(\oo):=\vert \{a\in A \mid a \leq_{\PS}\kg(\oo)\}\vert.
$$
\end{definition}

\begin{proposition}\label{cor:eta}
  Let $(R,\emptyset,T)$ be a molecule. 
Then, 
$$
\sum_{L\subseteq T} \rho(R\cup L, R\cup T)x^{\vert L \vert} =
\sum_{\oo \in\oC{R}/G} x^{\eta_T(\oo)}.
$$
\end{proposition}

\begin{remark}
  Notice that, in terms of the poset $\PS$, 
$$
\eta_T(\oo) = \vert\{t\in T \mid \kg(t) \leq_{\PS}\kg(\oo)\}\vert.
$$
Thus, in the realizable case we recover  the number defined in
\cite[Section 6]{BM}.
\end{remark}

\begin{proof}[Proof of Proposition \ref{cor:eta}]
  First notice that, for every $L\subseteq T$, $\mL:=(R\cup L, \emptyset,
  T\setminus L)$ is also a molecule, and that by Equation \eqref{eq:fctf}
  we have immediately $$f^{\mol}_{(\emptyset,L\cup L')} =
  f^{\mol}_{(\emptyset, L)} \circ f^{\mL}_{(\emptyset,L')}$$
  for every $L'\subseteq T\setminus L$. Therefore, with Lemma \ref{lem:mocposet}.(b) and Lemma \ref{lem:PP}
  we can write the following
\begin{align*}
\rho(R\cup L , R\cup T) &= \left\vert \oC{R\cup L}/G \,\,\,\bigg\backslash 
\bigcup_{t\in T\setminus L} \im f^{\mL}_{(\emptyset, \{t\})}\right\vert\\
&= \left\vert \im f^{\mol}_{(\emptyset, L)} \bigg\backslash 
\bigcup_{t\in T\setminus L} \im f^{\mol}_{(\emptyset, L\cup \{t\})}\right\vert =
\vert \overline{Z}^{\mol}(\emptyset, L)\vert
\end{align*}
where the second equality follows from injectivity of the functions $f^{\mol}$ and $f^{\mL}$.

\begin{itemize}
\item[{\em Claim.}]
For all $\oo\in \oC{R}/G$, if
 $\oo \in \overline{Z}^{\mol}(\emptyset, L)$ then
$$\{t\in T\mid t \leq_{\CC_\GS} \kg(\oo)\}=L.$$ 
In particular, we have that $\eta_{T}(\oo) = \vert L \vert$. 
\item[{\em Proof.}]
Let $\oo \in \overline{Z}^{\mol}(\emptyset, L)$ then for every $t\in T$
we have 
$\oo \in \im f^{\mol}_{(\emptyset,t)}$ if and only if there is a
representative $X_R$ of $\oo$ and some $x_t\in t$ such that
$X_R \cup x_t \in \mathcal C$. Since we know that $\rkE(R\cup t) =
\rkE(R)$, the latter is equivalent to saying that $x_t\in
\cl_{\CC}(X_R)$, i.e., $t\leq\kg(\oo)$ in $\CC_\GS$. 
Now, by Lemma \ref{lem:mocposet}.\ref{lem:bildf} we have
$$ \im f^{\mol}_{(\emptyset,L)} = \bigcap_{t\in L} \im f^{\mol}_{(\emptyset,t)}$$
and thus we see that $t\leq \kappa_{\GS}(\oo)$ if and only if $t\in L$. 
$\square$
\end{itemize}

We can now return to the statement to be proved and write
\begin{align*}
\sum_{L\subseteq T} \rho(R\cup L, R\cup T)x^{\vert L \vert}
&=\sum_{L\subseteq T} \vert \overline{Z}^{\mol}(\emptyset,L)\vert x^{\vert L \vert}
= \sum_{L\subseteq T} \sum_{\oo\in \overline{Z}^{\mol}(\emptyset,L)
} x^{\vert L \vert}\\
&= \sum_{\oo \in \oC{R}/G} x^{\eta_T(\oo)}
\end{align*}
where the last equality uses the Claim we just proved.
\end{proof}

\section{Almost-arithmetic actions}\label{sec:almost}
We now turn to what we call ``almost-arithmetic'' actions (see Definition \ref{def:normalA}). The name is reminiscent of the fact that one additional condition on top of \STy (i.e., normality) already ensures that the multiplicity function satisfies ``almost all'' of the requirements for arithmetic matroids (see Definition \ref{def:AM}): this is the gist of the main result of this section (Proposition \ref{prop:almost}).

We keep the notation $\GS$ to signify a $G$-semimatroid arising from an action on a semimatroid $\SS=(S,\CC,\rk)$. 

\begin{lemma}\label{lem:stabs}
  Let $\GS$ be almost-arithmetic
and let $X\in \CC$. Then 
  \begin{itemize}
  \item[(a)] for all $X'\in \oC{\underline X}$ we have
    $\stab(X)=\stab(X')$, 
  \item[(b)] if $x_0\in X$ and $\rk(X\setminus x_0)=\rk(X)$, then
    $\stab(X)=\stab(X\setminus x_0)$.
  \end{itemize}
\end{lemma}

\begin{proof}
  Item (a) is an immediate consequence of normality. In the claim of item (b), the inclusion $\stab(X) \subseteq \stab(X\setminus x_0)$ is evident. For the reverse inclusion, consider $g\in \stab(X\setminus x_0)$. Then
     $X\setminus x_0 \subseteq gX \cap X $, which justifies the first inequality in 
 \begin{equation}\label{eqed}
 \rk(X\setminus x_0) \leq \rk(gX\cap X) \leq \rk(X),
\tag{$\ast$} 
 \end{equation}
 where the second inequality holds by (R2). Since by assumption $\rk(X) = \rk(X\setminus x_0)$, equality must hold throughout \eqref{eqed} above, proving that  $\rk(X)=\rk(gX \cap
    X)$. By (CR1), the latter implies $X\cup g(X) \in \CC$ and, in
    particular, $\{x_0,g(x_0)\} \in \CC$. Translativity of the action
    then ensures $g\in \stab(x_0)$ and thus $g\in \stab (X)$.

\end{proof}

\begin{definition}\label{def:fA}
  Given $X_1,\ldots,X_k \in \CC$ define
 $$
 \theta_{X_1,\ldots,X_k}:G\to \prod_{i=1}^k \Gamma(X_i),\quad g\mapsto ([g]_{X_1},\ldots,[g]_{X_k}).
 $$
\end{definition}
\begin{remark}\label{rem:idx}
By Lemma \ref{lem:stabs}.(a), in an almost-arithmetic $G$-semimatroid
%Notice that if the action is normal, 
this map does not depend on the choice of the $X_i$ in
$\oC{\underline{X_i}}$ for $i=1,\ldots,k$.
\end{remark}

\begin{lemma}\label{lem:quot}
Given an almost-arithmetic $G$-semimatroid $\GS$, consider  $ A\subseteq \ES$ and $a_1,\ldots,a_k\in \ES$ with $\rkE(A\cup\{ a_1,\ldots, a_k\} ) = \rkE(A) + k$. For every choice of $X\in\oC{A}$ and of $x_i\in a_i$, $i=1,\ldots,k$,
$$
\frac{\mm(A\cup\{ a_1,\ldots, a_k\})}{\mm(A)} = 
[\Gamma(X) \times \prod_{i=1}^k\Gamma(x_i):
\theta_{X,x_1,\ldots, x_k}(G)]
$$
\end{lemma}
\begin{proof}
  Let $A$ and $a_1,\ldots,a_k$ be as in the statement and, in this proof, let us write  $A':=A\cup \{a_1,\ldots, a_k\}$. Since the
  action is \ST, with Lemma
  \ref{lem:inj-surj}.\ref{lem:geo-loop} we obtain the following equality of sets.
  $$\oC{A' } = \oC{A}\times \prod_{i=1}^k\oC{a_i}$$

 The projection 
  $$p_A : \oC{A'} \to \oC{A}, \quad Y \mapsto Y\setminus \bigcup_{i=1}^k a_i $$ 
maps each of the $\mm(A')$ orbits of the action of $G$ on $\oC{A'}$ to one of the $\mm(A)$ orbits of the action on $\oC{A}$.  Thus, it is enough to prove that the number of $\oC{A'}$-orbits mapped to a fixed $\oC{A}$-orbit equals the right-hand side of the equation in the claim.

To this end, choose $X\in \oC{A}$ and consider the set of orbits in $\oC{A'}$ which project to $GX$,
  i.e., the orbits of
  the action of $G$ on $$p_A^{-1}(GX)=\{(g(X),x_1,\ldots,x_k) \mid
  g\in G,\, \forall i=1,\ldots,k: x_i\in
  \oC{a_i}\}.$$

  Notice that for every $a\in \ES$ and every $x\in a$ we have an equality
  $a=Gx=\oC{a}$ and a natural bijection of this set with $\Gamma(x)$. In
  fact, any choice of $x_i\in \oC{a_i}$ for $i=1,\ldots,k$ and $X\in
  \oC{A}$ fixes a bijection $p_A^{-1}(GX) \to
  \Gamma(X)\times \prod_{i=1}^k\Gamma(x_i)$, and under this bijection the action of
  $G$ on the right-hand side is  given by composition with elements of the
  subgroup $\theta_{X,x_1,\ldots,x_k}(G)$ defined above. Therefore we
have a bijection $$p_A^{-1}(GX)/G \to (\Gamma(X)\times
\prod_{i=1}^k\Gamma(x_i))/\theta_{X,x_1,\ldots,x_k}(G).$$
By Lemma \ref{lem:stabs}.(a) and Remark \ref{rem:idx}, the group on the right hand side does not
depend on the choice of $X\in\oC{A}$ and $x_i\in a_i$. This concludes the proof.
\end{proof}

\begin{lemma}\label{lem:A}
  The multiplicity function associated to an almost-arithmetic
  $G$-semi\-matroid $\GS$ satisfies
$$\mm(R)\mm(R\cup F \cup T)=\mm(R\cup T) \mm(R\cup F)$$
for every molecule $(R,F,T)$ of $\SS_{\GS}$.
\end{lemma}

\begin{proof} We choose $X_{R\cup T} \in \oC{R\cup T}$ and let 
  $X_R:= X_{R\cup T}\setminus \cup T$, so that $X_R\in \oC{R}$. Moreover, write
  $F=\{f_1,\ldots,f_k\}$ and choose $x_i\in f_i$ for all
  $i=1,\ldots,k$. From Lemma \ref{lem:quot} we obtain the following equalities.
$$
\frac{m(R\cup F)}{m(R)} = 
\Big[\Gamma(X_R) \times \prod_{i=1}^k
  \Gamma(x_i) : \theta_{X_R,x_1,\ldots,x_k}(G)\Big]
$$
$$
\frac{m(R\cup T \cup F)}{m(R\cup T)} = 
\Big[\Gamma(X_{R\cup T}) \times \prod_{i=1}^k \Gamma(x_i)
  : \theta_{X_{R\cup T}, x_1,\ldots,x_k}(G) \Big] 
$$
Since $\rkE(R\cup T) = \rkE(R)$, by Lemma \ref{lem:stabs}.(b) we have $\stab(X_R)=\stab(X_{R\cup T})$, so (e.g., by inspection of Definition \ref{not:puntino}) the right-hand sides are equal.  
\end{proof}

\begin{proposition}\label{prop:almost}
  If $\GS$ is an almost-arithmetic action on a semimatroid, then $\mm$ satisfies properties (P), (A.1.2)
  and (A2) with respect to $\SS_\GS$.
\end{proposition}

\begin{proof}
This follows combining Lemma \ref{lem:P}, Lemma \ref{lem:quot} and Lemma \ref{lem:A}.
\end{proof}

We close the section on almost-arithmetic actions with a proposition
about molecules of the form $(R,F,\emptyset)$, as a counterpart to
Proposition \ref{cor:eta} above.

\begin{definition}\label{def:iota}
  Let $\GS$ be an almost-arithmetic $G$-semimatroid. Given a molecule
  $\mol:=(R,F,\emptyset)$ of $\SS_{\GS}$, choose an orbit $\oo\in
  \oC{R}/G$ and fix a representative $X_R\in \oo$. For every
  $F'\subseteq F$ let 
  $\mathcal X(F')\subseteq \oC{R\cup F'}/G$ denote the subset consisting of
  orbits of the form $GY$ with $X_R\subseteq gY$ for some $g\in
  G$, i.e., 
  $$\mathcal X(F') := (\oC{R\cup F'}/G)_{\geq \oo} \subseteq \CC_{\GS}.$$
    Fix a numbering of the elements of $F$ and, recalling Definition \ref{def:lezete}, let $$\widetilde{Z}_F^{\mol}(F'):= \overline{Z}^{\mol}(F',\emptyset) \cap
  \mathcal X(F).$$ The sets $\{\widetilde{Z}_F^{\mol}(F')\}_{F'\subseteq F}$
  partition $\mathcal X(F)$. Thus, for every $\oo \in \mathcal
  X(F)$ we can consider the unique $F'\subseteq F$ for which $\oo \in \widetilde{Z}_F^{\mol}(F')$ and define the number
$$
\iota(\oo) := \vert F \vert - \vert F' \vert.
$$ 
\end{definition}

\begin{lemma}\label{lem:Zr} Let $\GS$ be an almost-arithmetic $G$-semimatroid and let ${\mol}:=(R,F,\emptyset)$ be a molecule of $\SS_{\GS}$. Then for all $F'\subseteq F$ we have
$$
\vert \widetilde{Z}^{\mol} _F (F') \vert = \frac{\rho(R,R\cup F')}{m_{\GS}(R)}.
$$
  In particular, this cardinality does not depend on the choice of the representative $X_R$ and of the numbering of the elements of $F$.
\end{lemma}

\begin{proof}
By construction, $\vert Z^{\mol}(F',\emptyset)\cap \mathcal X(F)\vert = \sum_{(F'',\emptyset)\leq  (F',\emptyset)} \vert \widetilde{Z}^{\mol}_F (F'') \vert$.
Hence (following the notation of \cite{Sta}, to which we refer for basics about  M\"obius transforms),
%\todo{Check notation Stanley Moebius transform}
 $\vert \widetilde{Z}^{\mol}_F (F') \vert = (\mu  \Psi)(F',\emptyset)$, the evaluation at $(F',\emptyset)$ of the M\"obius transform of the function 
 $$\Psi : P[R,F',\emptyset] \to \mathbb Z, \quad (F'',\emptyset) \mapsto \vert Z^{\mol}(F'',\emptyset)\cap \mathcal X(F)\vert  = m(R\cup F'')/m(R)$$ (where the equality holds by   Lemma \ref{lem:quot}).
 %, $\Psi(F'',\emptyset) = m(R\cup F'')/m(R)$ and, b
 By the same computation as in the proof of Lemma \ref{lem:PP}, the M\"obius transform $(\mu \Psi)$ then satisfies 
 $$(\mu \Psi)(F',\emptyset) = \rho(R,R\cup F')/m(R)$$
whence the claim.
\end{proof}

\begin{proposition}\label{prop:iota}
  Let $\GS$ be almost-arithmetic and let $\mol:=(R,F,\emptyset)$
  be a molecule of $\SS_{\GS}$. Then, with the notations of Definition \ref{def:iota},
$$
\sum_{F'\subseteq F} \frac{\rho(R,R\cup F')}{m_{\GS}(R)} x^{\vert
  F\setminus F'
  \vert}
= \sum_{\oo\in \mathcal{X} (F)} x^{\iota (\oo)}.
$$
\end{proposition}

\begin{proof} The proof reduces to the following direct computation, where the first equality is Lemma \ref{lem:Zr}.
\begin{align*}
\sum_{F'\subseteq F} \frac{\rho(R,R\cup F')}{m_{\GS}(R)} x^{\vert
  F\setminus F'
  \vert}
&=\sum_{F'\subseteq F} \vert \widetilde{Z}^{\mol}_F (F') \vert x^{\vert
  F\setminus F'
  \vert}\\
&=\sum_{F'\subseteq F} \sum_{\oo\in \widetilde{Z}^{\mol}_F (F')} x^{\vert
  F\setminus F'
  \vert}
= \sum_{\oo\in \mathcal{X} (F)} x^{\iota (\oo)}
\end{align*}
\end{proof}

\section{Arithmetic actions}\label{sec:arithm}
In this section we assume that the actions under consideration are
arithmetic. A glance back at Definition \ref{def:args} will remind the reader that this assumption is much more restrictive (and more algebraic in nature) than
almost-arithmetic. 
\begin{lemma}\label{lem:unabrep}
 Let $\GS$ be an arithmetic $G$-semimatroid and consider $A\subseteq \ES$. Then, for any two $X,Y\in \oC{A}$, 
 \begin{center}
 (i) $\Gamma(X)=\Gamma(Y),\quad$ (ii) $\Gamma^X= \Gamma^Y$.
 %,\quad$ (iii) $W(X)=W(Y),\quad$ (iv) $h_X=h_Y.$
 \end{center}
\end{lemma}

\newcommand{\eqdef}{\overset{\mathrm{def}}{=\joinrel=}}

\begin{proof}
  Fix two sets $X,Y\in \oC{A}$ as in the claim. By Lemma \ref{lem:stabs}.(a), $\stab(X)=\stab(Y)$, hence immediately $\Gamma(X)=\Gamma(Y)$. Moreover, since every arithmetic action is \ST,  $X$ and $Y$ contain exactly one element $x_a$ resp.\ $y_a$ of every
  orbit in $A$: in fact, $X=\{a\in A \mid x_a\}$, $Y=\{a\in A \mid y_a \}$.
  % and $\Gamma^Y\simeq \Gamma^X$. 
  In order to prove (ii), we recall Definition \ref{not:puntino} and compute 
  \begin{center}
    $\Gamma^X\eqdef \prod_{a\in A} \Gamma(x_a)=\prod_{a\in A} \Gamma(y_a)\eqdef\Gamma^Y$, 
  \end{center}
where the equality in the middle is part (i) applied to $X=\{x_a\}$, $Y=\{y_b\}$.
\end{proof}

In particular, for arithmetic actions we can simplify Definition \ref{not:puntino} as follows.

 \begin{definition}
  Given $A\in
  \underline \CC$, choose $X\in \oC{A}$ and write 
  $$\Gamma^A:=\Gamma^X,\quad\Gamma(A):=\Gamma(X),$$
 % \todo{when $\oC{A}=\emptyset$ then both groups are the trivial group}
  %\quad W(A):=W(X), \quad h_A:=h_X.$$
    By Lemma \ref{lem:unabrep}, these are well-defined and independent from the choice of $X$.
  \end{definition}

\begin{lemma}\label{lem:Wunabrep}
 Let $\GS$ be an arithmetic $G$-semimatroid, consider $A\in \underline{\CC}$ and pick any two $X,Y\in\oC{A}$. Then,
 %\todo{when empty... otherwise}
 \begin{itemize}
 \item[(i)]  $W(X)$ and $W(Y)$ are conjugated subgroups of $\Gamma^A$.
 \item[(ii)] $\mm(A) = [W(X):h_X(G)]$
 \end{itemize}
 
\end{lemma}

\begin{proof}$\,$

\begin{itemize}
\item[(i)]
 For every $a\in A$ choose $g_a\in G$ with $x_a=g_a(y_a)$; with this, define the $A$-tuple $\gamma_{YX}:=([g_a])_{a\in A}\in \Gamma^A$. We have immediately 
  \begin{center}
  ($\ast$) $X=\gamma_{YX}.Y$, hence ($\ast\ast$) $\gamma_{YX}\in W(Y)$.
  \end{center}
 % (again, for the definitions see Notation \ref{not:puntino}). 
   % and $\gamma_{YX}^{-1}\in W(Y)$. 
   \begin{itemize}
   \item[{\em Claim.}] $W(X)=\gamma_{YX}W(Y)\gamma_{YX}^{-1}$, hence $W(X)$ and $W(Y)$ are conjugate in $\Gamma^A$.
   \item[{\em Proof.}]
        By symmetry, it is enough to show $\gamma_{YX}W(Y)\gamma_{YX}^{-1}\subseteq W(X)$. Let, thus, $\gamma\in W(Y)$. For arithmetic actions multiplication is well defined in the group $W(Y)$, thus ($\ast\ast$) implies $\gamma_{YX}\gamma\in W(Y)$. With this, 
    $$
    (\gamma_{YX}\gamma{\gamma_{YX}}^{-1}).X 
    \overset{\mathrm{(}\ast\mathrm{)}}{=} 
    (\gamma_{YX}\gamma).Y \in \CC
    $$ and therefore $(\gamma_{YX}\gamma{\gamma_{YX}}^{-1})\in W(X)$. \qed
   \end{itemize}
\item[(ii)] The choice of $X$ fixes a function
\begin{equation}\label{eqqdefb}
b_X:\oC{A} \to W(X),\quad \{g_x x\mid x\in X\} \mapsto ([g_x]_x)_{x\in
  X}
\end{equation}
which is easily seen to be bijective. Moreover, for every $g\in G$ and $Y\in \oC{A}$,
\begin{equation}\label{eqdefb}
b_X(gY) = h_X(g)b_X(Y).
\end{equation}
Thus $b_X$ induces a bijection of sets 
$\oC{A}/G \to W(X)/h_X(G)$ mapping an orbit $GY$ to the coset $h_X(G)b_X(Y)$. We now compute
$$
\mm(A) = \vert \oC{A}/G \vert = \vert W(X)/h_X(G) \vert = [W(X):h_X(G)].
$$

   \end{itemize}

\end{proof}

\begin{definition}\label{def:wneu}
Let $\GS$ be an arithmetic $G$-semimatroid, and consider $A\in \underline{\CC}$. Choose $X\in\oC{A}$ and $x_0$. The projection $\Gamma^X \to \Gamma^{X\setminus x_0}$ induces a group homomorphism 
$$
w_{X,x_0}: W(X) \to W(X\setminus x_0), \quad ([g_x]_x)_{x\in X} \mapsto ([g_x]_x)_{x\in X\setminus x_0}
$$
\end{definition}

\begin{remark} Let $\GS$ be arithmetic. Consider $A\in \underline{\CC}$ and $a_0\in A$, choose $X\in \oC{A}$, and let $x_0\in a_0\cap X$. 
%as well as Definition \ref{not:puntino}.
%\todo{Cite definitions etc.}
The following diagram is commutative
\begin{equation}\label{diag}
\begin{CD}
   \oC{A} @>{b_{X}}>> W(X) @<{h_X}<< G \\          
   @V{w_{A,a_0}}VV     @V{\setmap{X}{x_0}}VV @|\\
  \oC{A\setminus a_0} @>{b_{X\setminus x_0}}>> W(X\setminus x_0) @<{h_{X\setminus x_0}}<< G \\
\end{CD}
\end{equation}
where the maps $b_{\ast}$ are defined in Equation \eqref{eqqdefb}.\end{remark}

\subsection{Arithmetic matroids} Theorem \nolinebreak\ref{thm:arithm} follows easily from the next lemma,
which proves that arithmetic actions induce the last of the defining
properties of arithmetic matroids which was not fulfilled by
almost-arithmetic actions (Example \ref{ex:noarithm} shows that this difference
is nontrivial).

\begin{lemma}\label{lem:lari}
  Let $\GS$ be a $G$-semimatroid associated to an arithmetic
  action. Then $\mm$ satisfies property (A.1.1) of Definition \ref{def:AM}.
\end{lemma}
  
\begin{proof}
Consider $A\in \underline \CC $ and $a_0\in A$ such that $\rkE(A\setminus a_0)
= \rkE(A)$. Choose $X\in \oC{A}$ and $x_0\in a_0\cap X$.
Using Lemma \ref{lem:Wunabrep}.(ii) we have $\mm(A\setminus a_0) = [W(X\setminus x_0):h_{X\setminus x_0}(G)]$.

 By Lemma \ref{lem:inj-surj}, the condition on the ranks implies that $w_{A,a_0}$ is injective. Commutativity of the left-hand side square in Diagram \eqref{diag} implies that $\setmap{X}{x0}$ is injective. Therefore (using again Lemma \ref{lem:Wunabrep})  we can write
%  The injective homomorphism $\setmap{A}{a_0}: W(A)\to W(A\setminus a_0)$ maps $\theta_A(G)$ to $\theta_{A\setminus a_0}(G)$.
  $$\mm(A)=[W(X):h_{X}(G)]=[\im(\setmap{X}{x_0}): h_{X\setminus x_0}(G)].$$
    Now the claim follows
  from multiplicativity of the index in the chain of subgroups $h_{X\setminus x_0} (g) \subseteq \im(\setmap{X}{x_0}) \subseteq W(X) $, which allows us to write
  
  $$\mm(A\setminus a_0)=[W(X\setminus
  x_0): \im(\setmap{X}{x_0})]\mm(A)$$
  proving in particular that $\mm(A)$ divides $\mm(A\setminus a_0)$, as claimed. 
\end{proof}

%
%\begin{definition}\label{def:bw} Let $\GS$ be \ST.
%  Given $A\in \underline \CC$, every $X\in\oC{A}$ determines a
%  bijection
%$$
%b_X:\oC{A} \to W(A),\quad \{g_x x\mid x\in X\} \mapsto ([g_x]_x)_{x\in
%  X}
%$$
%and the action of a $g\in G$ on $\oC{A}$ corresponds to diagonal
%(left) multiplication by $h_A(g)$ in $W(A)$. 
%\end{definition}
%Via $b_X$, for every $a_0\in A$, the
% map 
%\begin{equation}
%  \label{eq:7}
%  \setmap{A}{a_0}: \oC{A} \to \oC{A\setminus a_0}, \quad X\mapsto
%  X\setminus a_0
%\end{equation}
%considered above induces a map $W(A)\to W(A\setminus a_0)$  which, by slight abuse of notation, we also call $\setmap{A}{a_0}$. This is the restriction of the projection $\Gamma^{A}\to\Gamma^{A\setminus a_0}$. By Lemma \ref{lem:inj-surj}, this map is injective whenever 
%$\rkE(A)=\rkE(A\setminus a_0)$.
%
%
%
%As a preparation for the next section, let us here discuss some further aspects of arithmetic actions.

\subsection{Matroids over rings}
\label{sec:algebraic}

We now outline a link to the theory of matroids over rings. 
We focus mainly on matroids over the ring $\mathbb Z$ both for conciseness' sake and because this is the case most strongly related to arithmetic matroids (see \cite[Section 6.1]{FM}). Our goal is to give a direct combinatorial interpretation of some matroids over $\mathbb Z$ arising from group actions on semimatroids (and, in particular, from toric arrangements).

\def\cyc{Cyc}

With this in mind, from now we will let 
$\GS$ %$\GS: G\circlearrowright (S,\CC,\rk)$ 
denote an arithmetic $G$-semimatroid and consider the following condition.
\begin{itemize}
\item[(\cyc)] For every $e\in \ES $, $\Gamma^{\{e\}}$ is a cyclic group.
\end{itemize}
\begin{remark}\label{rem:WurA}
An immediate consequence of (\cyc) is that, for every $A\subset \ES$, the group $\Gamma^A$ is abelian. In particular, Lemma \ref{lem:Wunabrep}.(i) implies $W(X)=W(Y)$ and $h_X=h_Y$ for all $X,Y\in \oC{A}$. 
\end{remark}

\begin{definition}\label{def:allcan}
Let $\GS$ denote an arithmetic $G$-semimatroid. Then, for every $A \subseteq \ES$ and every $a_0\in A$ we have the following canonical group homomorphisms.
\begin{itemize}
\item[(i)] $g_{A,a_0}: \Gamma(A) \to \Gamma(A\setminus a_0)$, induced by the inclusion $\stab(A) \subseteq \stab(A\setminus a_0)$.
\item[(ii)] $\pi_{A,a_0}: \Gamma^{A} \to \Gamma^{A\setminus a_0}$, the canonical projection along the $a_0$-coordinate.
\end{itemize}
When (\cyc) holds and if $A\in\underline{\CC}$, with Remark \ref{rem:WurA} we can set $W(A):=W(X)$ and $h_A:=h_X$ (see Definition \ref{not:puntino}), where $X$ is any element $X\in \oC{A}$. We then have more canonical homomorphisms.
\begin{itemize}
\item[(iii)] $w_{A,a_0}: W(A) \to W(A\setminus a_0)$, induced by $\pi_{A,a_0}$ and equal to the map of Definition \ref{def:wneu} (see Remark \ref{an}.(a) below). 
\item[(iv)] $j'_A: \Gamma(A) \to W(A)$, induced by $h'_A$, and $j_A: \Gamma(A) \to \Gamma^A$, induced by $h_A$ (see Remark \ref{an}.(b) below). 
\end{itemize}

\end{definition}
\begin{remark}\label{an}$\,$
\begin{itemize}
\item[(a)] The maps $w_{A,a_0}$ defined in (iii) above should be regarded as the natural "enriched" version of their namesakes from Definition \ref{def:WA}. In fact, as maps of sets, the two correspond via the natural bijections $b_A: \oC{A} \to W(A)$ (cf.\ Equation \eqref{eqqdefb}). More precisely the following diagram (of sets) commutes.

\begin{center}
\begin{tikzcd}
\oC{A} 
\arrow{rrr}{\setmap{A}{a_0}}[swap]{\textrm{Definition \ref{def:WA}}} 
\dar{b_A}
&&&
\oC{A\setminus a_0} 
\dar{b_{A\setminus a_0}} \\
W(A) 
\arrow{rrr}{w_{A,a_0}}[swap]{\textrm{Definition \ref{def:allcan}.(iii)}}
&&& 
W(A\setminus a_0)
\end{tikzcd}
\end{center}
\item[(b)] The homomorphisms $j_A$ and $j'_A$ of Definition \ref{def:allcan}.(iv) are well defined and injective. In fact, since $\ker h_A = \ker h'_A = \stab(A)$, both $h_A$ and $h'_A$ factor uniquely by injective maps through the quotient $q: G\to \Gamma(A)$. We summarize with the following diagram.

\begin{center}
\begin{tikzcd}
G \arrow{rr}{{h}_A}\arrow{dd}[swap]{q}\arrow[ddrr, pos=.3, outer sep=-2pt, "h'_A"] 
&& W(A) \arrow[dd, hook, "\iota"] \\
&&\\
\Gamma(A)\arrow[rr, swap, tail, "j_A"] 
\arrow[uurr, tail, crossing over, pos=0.3, outer sep=-2pt, swap, "j'_A"]
 && \Gamma^A 
\end{tikzcd}
\end{center}
\end{itemize}
\end{remark}

\begin{definition}\label{def:muA}
%Write $h''_A:\Gamma(A) \to \Gamma^A$ for the composition of $h'_A$ with the inclusion $W(A)\hookrightarrow \Gamma^A$. 
Given an arithmetic $G$-semimatroid $\GS$ satisfying (\cyc) define, for every $A\subseteq \ES$ such that $\Ac\in \underline{\CC}$, an abelian group
$$
M_\GS (A) := \Gamma^{A^c} / \im (h'_{A^c}).
$$
Moreover, for every $a_0\in \ES$ let $\mu_{A,a_0}: M_\GS(A) \to M_{\GS}(A\cup a_0)$ be the unique group homomorphism that makes the following diagram of short exact sequences commute.
\begin{equation}\label{diag:defM}
\begin{tikzcd}
0 \rar & 
    \Gamma(A^c) \arrow{rr}{j_{A^c}} \dar{g_{A^c,a_0}} && 
    \Gamma^{A^c} \rar \dar{\pi_{A,a_0}} &
    M_\GS(A) \rar \dar[dashed]{\mu_{A,a_0}} \rar & 0\\
0 \rar & 
    \Gamma(A^c\setminus a_0) \arrow{rr}{j_{A^c\setminus a_0}} & &
    \Gamma^{A^c\setminus a_0} \rar  &
    M_\GS(A\cup a_0) \rar  \rar & 0\\    
\end{tikzcd}
\end{equation}

\end{definition}

\begin{lemma}\label{lem:gpwm}
Let $\GS$ be arithmetic, suppose that (\cyc) holds, and recall Definition \ref{def:allcan}. Then, for every $A\subseteq \ES$ and every $a_0\in \ES$, \begin{itemize}
 \item[(i)] $g_{A,a_0}$ is surjective with cyclic kernel;
 \item[(ii)] $\pi_{A,a_0}$ is surjective with cyclic kernel;
 \item[(iii)] $w_{A,a_0}$ is surjective with cyclic kernel;
 \item[(iv)] $\mu_{A,a_0}$ is surjective with cyclic kernel.   
 \end{itemize}
\end{lemma}
\begin{proof}
Part (ii) is clear from (\cyc), and implies Part (iii) since $w_{A,a_0}$ is the restriction of $\pi_{A,a_0}$ to $W(A)$. Surjectivity of $g_{A,a_0}$ is also evident from the definition. With these preliminary remarks we can complete the diagram in Definition \ref{def:muA} with the kernels and cokernels of the vertical maps, obtaining the diagram in Figure \ref{fig:GGG}. 
\begin{figure}[ht]
\begin{center}
\begin{tikzcd}
& 0 \dar & 0 \dar & 0 \dar & \\
  0\rar   & \ker(g_{A^c,a_0}) \dar\rar & \ker(\pi_{A^c,a_0}) \dar\rar  & \ker (\mu_{A,a_0})\dar\rar & 0 \\
0 \rar & 
    \Gamma(A^c) \arrow{r} \dar{g_{A^c,a_0}} &
    \Gamma^{A^c} \rar \dar{\pi_{A,a_0}} &
    M_\GS(A) \rar \dar{\mu_{A,a_0}} \rar & 0\\
0 \rar & 
    \Gamma(A^c\setminus a_0) \arrow{r}{}\dar & 
    \Gamma^{A^c\setminus a_0} \rar\dar  &
    M_\GS(A\cup a_0) \rar  \rar\dar & 0\\    
&  0\rar[dashed]  & 0\rar[dashed] & \operatorname{coker} (\mu_{A,a_0})\rar[dashed] & 0
\end{tikzcd}
\end{center}
\caption{Diagram for the proof of Lemma \ref{lem:gpwm}.}
\label{fig:GGG}
\end{figure}
%Since cokernel is a right-exact functor,\todo{reference for cokernel exact} 
We first check that the bottom row (dashed) is exact and thus we obtain $\operatorname{coker} (\mu_{A,a_0}) = 0$. Then, the nine lemma implies that the top row is exact: since we know that $\ker(\pi_{A^c,a_0})$ is cyclic, we can thus deduce ciclicity of $\ker(g_{A^c,a_0})$ and $\ker(\pi_{A,a_0})$. This concludes the proof of (i) and (iv).
\end{proof}

\begin{lemma}\label{lem:Mpushout}
Let $\GS$ be arithmetic and suppose that (\cyc) holds. Then, for every $A\subseteq \ES$ such that $A^c\in \underline{\CC}$ and every $a_0,b_0\in \ES$, the following is a pushout square.
\begin{center}
\begin{tikzcd}
    M_\GS(A) \arrow{rr}{\mu_{A,a_0}}\dar{\mu_{A,b_0}} 
        && M_\GS (A\cup \{a_0\})\dar{\mu_{A\cup \{a_0\}, b_0}} \\
    M_\GS(A\cup \{b_0\}) \arrow{rr}{\mu_{A\cup \{b_0\},a_0}} 
        && M_\GS (A\cup \{a_0,b_0\})
\end{tikzcd}
\end{center}
\end{lemma}
\begin{proof} The morphism of short exact sequences defining the maps $\mu_{\ast,\ast}$ described in Diagram \eqref{diag:defM} can be fit together to a square of short exact sequences as follows, where for simplicity we write $A':=A\cup \{a_0\}$, $A'':=A\cup \{b_0\}$, $A''':=A\cup \{a_0, b_0\}$, so that the right-hand side square is indeed the square appearing in the claim.

\begin{center}
\begin{tikzpicture}
%\node (a) at (0,0) {a};  
%\node[above right] (b) {b};
%\node[above right = of a] (c) {c};
%\node[above right = 2cm of a] (d) {d};
%\node[above right = 2cm and 3cm of a] (e) {e};
%\draw[help lines,step=5mm,gray!20] (0,0) grid (10,1);
\node (A) at (0,0) {};
\node[right = 0cm of A, anchor=center] (ZZA) {$0$};
\node[right = 2.5cm of A, anchor=center] (GaA) {$\Gamma(\overline{A})$};
\node[right = 5cm of A, anchor=center] (GA) {$\Gamma^{\overline{A}}$};
\node[right = 7.5cm of A, anchor=center] (MA) {$M_\GS(A)$};
\node[right = 10cm of A, anchor=center] (ZA) {$0$};
\draw[->, dashed] (ZZA) -- (GaA);
\draw[->, dashed] (GaA) -- (GA);
\draw[->] (GA) -- (MA);
\draw[->] (MA) -- (ZA);
%%%%%%%%%%%%%%%%%%%%%%%%%%%%%%%
\node (B) at (-1,-1) {};
\node[right = 0cm of B, anchor=center] (ZZB) {$0$};
\node[right = 2.5cm of B, anchor=center] (GaB) {$\Gamma(\overline{A'})$};
\node[right = 5cm of B, anchor=center] (GB) {$\Gamma^{\overline{A'}}$};
\node[right = 7.5cm of B, anchor=center] (MB) {$M_\GS({A'})$};
\node[right = 10cm of B, anchor=center] (ZB) {$0$};
\draw[->, dashed] (ZZB) -- (GaB);
% Other arrows below
%%%%%%%%%%%%%%%%%%%%%%%%%%%%%%%
\node (C) at (1,-2.5) {};
\node[right = 0cm of C, anchor=center] (ZZC) {$0$};
\node[right = 2.5cm of C, anchor=center] (GaC) {$\Gamma(\overline{A''})$};
\node[right = 5cm of C, anchor=center] (GC) {$\Gamma^{\overline{A''}}$};
\node[right = 7.5cm of C, anchor=center] (MC) {$M_\GS({A''})$};
\node[right = 10cm of C, anchor=center] (ZC) {$0$};
\draw[->, dashed] (ZZC) -- (GaC);
\draw[->] (GaC) -- (GC);
\draw[->] (GC) -- (MC);
\draw[->] (MC) -- (ZC);
%%%%%%%%%%%%%%%%%%%%%%%%%%%%%%%
\node (D) at (0,-3.5) {};
\node[right = 0cm of D, anchor=center] (ZZD) {$0$};
\node[right = 2.5cm of D, anchor=center] (GaD) {$\Gamma(\overline{A'''})$};
\node[right = 5cm of D, anchor=center] (GD) {$\Gamma^{\overline{A'''}}$};
\node[right = 7.5cm of D, anchor=center] (MD) {$M_\GS({A'''})$};
\node[right = 10cm of D, anchor=center] (ZD) {$0$};
\draw[->, dashed] (ZZD) -- (GaD);
\draw[->] (GaD) -- (GD);
\draw[->] (GD) -- (MD);
\draw[->] (MD) -- (ZD);
%%%%%%%%%%%%%%%%%%%%%%%%%%%%%%%%
\draw[->, dashed] (GaA) -- (GaB); 
\draw[draw=white,solid,line width=2mm,fill=white] (GaB) -- (GaD);
\draw[->, dashed] (GaB) -- (GaD);
\draw[->, dashed] (GaA) -- (GaC);
\draw[->] (GaC) -- (GaD);
%%%%%%%%%%%%%%%%%%%%%%%%%%%%%%%%%
\draw[->] (GA) -- (GB);
\draw[draw=white,solid,line width=2mm,fill=white] (GB) -- (GD); 
\draw[->] (GB) -- (GD);
\draw[->] (GA) -- (GC);
\draw[->] (GC) -- (GD);
%%%%%%%%%%%%%%%%%%%%%%%%%%%%%%%%%
\draw[->] (MA) -- (MB);
\draw[draw=white,solid,line width=2mm,fill=white] (MB) -- (MD); 
\draw[->] (MB) -- (MD);
\draw[->] (MA) -- (MC);
\draw[->] (MC) -- (MD);
%%%% Horizontal B-arrows
\draw[draw=white,solid,line width=2mm,fill=white] (GaB) -- (GB);
\draw[draw=white,solid,line width=2mm,fill=white] (GB) -- (MB);
\draw[draw=white,solid,line width=2mm,fill=white] (MB) -- (ZB);
\draw[->, dashed] (GaB) -- (GB);
\draw[->] (GB) -- (MB);
\draw[->] (MB) -- (ZB);
\end{tikzpicture}
\end{center}
By part (i) and (ii) of Lemma \ref{lem:gpwm}, by exactness of the rows  and with Definition \ref{def:muA}, the part of the diagram drawn with solid arrows satisfies the assumptions of Lemma \ref{lem:backdoor}, which allows us to conclude that the right-hand side square is a pushout square, as was to be shown.
\end{proof}

\begin{proposition}\label{prop:D1}
$M_{\GS}$ is a representable matroid over $\mathbb Z$.
\end{proposition}

\begin{proof}
This follows combining Lemma \ref{lem:gpwm}.(iv) and Lemma \ref{lem:Mpushout}.
\end{proof}

\begin{lemma}\label{lem:rextr}
Let $\GS$ be an arithmetic $G$-semimatroid such that all groups $\Gamma^a$ are infinite cyclic. Then for all $A\in \underline{\CC}$ the rank of $W(A)$ as a $\mathbb Z$-module is
$$
\operatorname{rank}_{\mathbb Z}(W(A)) = \rkE(A)
$$
\end{lemma}
\begin{proof}
Let $F\subseteq A$ be a maximal independent set in $A$, i.e., a subset with $\vert F \vert = \rkE(A)$. In particular, such an $F$ satisfies $\vert F \vert = \rkE(F)$ and thus, by Lemma \ref{lem:inj-surj} and Definition \ref{not:puntino} 
\begin{equation}\label{eq:aww}
W(F) = \Gamma^F \simeq \mathbb Z^{\vert F \vert} .  
\end{equation}
Moreover, since $\rkE(F)=\rkE(A)$, by Lemma \ref{lem:inj-surj} and Remark \ref{an}.(a), the group homomorphism $w_{A,A\setminus F}: W(A) \to W(F)$ is injective and, by the additivity theorem for ranks, we have
\begin{equation}\label{eq:awww}
\operatorname{rank}_{\mathbb Z}(W(F)) = 
\operatorname{rank}_{\mathbb Z}(W(A)) + 
\operatorname{rank}_{\mathbb Z}\left(\frac{W(F)}{w_{A,A\setminus F}(W(A))}\right)
\end{equation}
\begin{itemize}
\item[{\em Claim.}] $\operatorname{rank}_{\mathbb Z}\left(\frac{W(F)}{w_{A,A\setminus F}(W(A))}\right)=0$.
\item[{\em Proof.}] The subgroup $h_F(\Gamma(F)) \subseteq W(F)$ is contained in $w_{A,A\setminus F}(W(A))$. By the ``third isomorphism theorem'' for groups we have an isomorphism
$$
\frac{W(F)}{w_{A,A\setminus F}(W(A))} \simeq \frac{W(F)/j'_F(\Gamma(F))}{w_{A,A\setminus F}(W(A))/j'_F(\Gamma(F))}.
$$
The cardinality of $W(F)/j'_F(\Gamma(F))$ equals $\mm(F)$ and is, in particular, finite. Thus both groups above are finite and have rank zero as $\mathbb Z$-modules. \qed
\end{itemize}
With the claim we can conclude with the following computation (where we use Equation \eqref{eq:aww}, Equation \eqref{eq:awww} and the definition of $F$).
$$
\operatorname{rank}_{\mathbb Z}(W(A)) =
\operatorname{rank}_{\mathbb Z}(W(F)) =
\vert F \vert = \rkE(A)
$$

\end{proof}

\begin{corollary}\label{cor:D2}
Let $\GS$ be a centered arithmetic $G$-semimatroid such that all groups $\Gamma^a$ are infinite cyclic. Then for every $A\in \underline{\CC}$ the rank of $M_\GS(A^c)$ as a $\mathbb Z$-module is
$$
\operatorname{rank}_{\mathbb Z}(M_\GS(A^c)) = \vert A\vert - \rkE(A)
$$
\end{corollary}
\begin{proof} First, notice that Remark \ref{an}.(b) implies exactness of the sequence
\begin{center}
\begin{tikzcd}
0 \rar & \Gamma (A) \rar{j'_A} & W(A) \rar & W(A)/j'_A(\Gamma(A)) \rar & 0
\end{tikzcd}
\end{center}
Since the group $W(A)/h_A(\Gamma(A))$ has finite cardinality (equal to $\mm(A)$), the additivity theorem for ranks of abelian groups implies
$$
\operatorname{rank}_{\mathbb Z}(W(A)) = 
\operatorname{rank}_{\mathbb Z}(j'_A(\Gamma(A))).
$$
In particular, using the definitions and Lemma \ref{lem:rextr}, we conclude
$$
\operatorname{rank}_{\mathbb Z}(M_{\GS}(A^c))=
\operatorname{rank}_{\mathbb Z}(\Gamma^A/j_A(\Gamma(A)))
$$
$$
=
\operatorname{rank}_{\mathbb Z}(\Gamma^A)-
\operatorname{rank}_{\mathbb Z}(j_A(\Gamma(A))) =
\vert A \vert - \rkE(A).
$$

\end{proof}

%We can summarize in the following proposition.
%
%\begin{proposition}
%Let $\GS$ be a centered arithmetic $G$-semimatroid such that all groups $\Gamma^a$ are infinite cyclic. Then $M_\GS$ is a representable matroid over $\mathbb Z$ whose underlying matroid is the dual to $(\ES,\rkE)$.
%\end{proposition}

\begin{corollary}\label{cor:D7}
Let $\GS$ be a centered arithmetic $G$-semimatroid such that all groups $\Gamma^a$ are infinite cyclic. Then the underlying matroid of $M_\GS$ is the dual to $(\ES,\rkE)$.\end{corollary}
\begin{proof}
By Remark \ref{rem:uam} and Corollary \ref{cor:D2} the rank function $\rk$ of the underlying matroid satisfies
$$
\rk(\ES) - \rk(A)= \operatorname{rank}_{\mathbb Z} (M_{\GS} (A)) = \vert A^c \vert - \rkE(A^c)
$$
For all $A\subseteq \ES$. After an elementary manipulation we recover $\rkE(A^c)= \rk(A) - \vert A^c \vert - \rkE(\ES)$, proving that $(\ES,\rkE)$ and $(\ES,\rk)$ are dual (see, e.g., \cite[Proposition 2.1.9]{Oxl}). 
\end{proof}

We end by describing a situation where 
%not only the underlying matroid of $M_\GS$ is $(\ES,\rkE)$ but also 
the torsion elements of the modules $M_\GS$ can be interpreted combinatorially.

\begin{proposition}\label{prop:D3}
Let $\GS$ be a centered arithmetic $G$-semimatroid such that all groups $\Gamma^a$ are infinite cyclic and consider $A\subseteq \ES$. If $W(A)$ is a pure subgroup of $\Gamma^A$, then
$$
M_\GS(A)\simeq \mathbb Z^{\vert A^c \vert - \rkE(A)} \oplus W(A)/h_A(G)
$$
\end{proposition}
\begin{proof} Consider the following diagram.
\begin{center}
\begin{tikzcd}
% & & & 0\dar & \\
 & 0\dar & 0\dar & \ker (\varphi) \dar & \\
0\rar & \Gamma(A)\dar{j'_A}\rar{j_A} & \Gamma^A\dar{=}\rar & M_\GS(A)\rar\dar{\varphi}&0\\
0\rar & W(A) \rar[hook]\dar &\Gamma^A\dar \rar{\epsilon} &L(A)\dar \rar& 0\\
 & C(A) \rar & 0 \rar & \operatorname{coker}(\varphi)\rar & 0 \\
 %& 0 & & 0 &  
\end{tikzcd}
\end{center}
By the snake lemma we have an isomorphism $\ker(\varphi)\simeq W(A)/j'_A(\Gamma(A))$. Moreover, exactness of the second row at $L(A)$ implies that the last row is exact at $\operatorname{coker}(\varphi)$, and the latter is thus trivial. Summarizing, we have the following exact sequence.
\begin{center}
\begin{tikzcd}
0\rar & W(A)/j'_A(\Gamma(A)) \rar & M_\GS(A) \rar &  L(A) \rar & 0
\end{tikzcd}
\end{center}
The purity assumption on $W(A)$ means that $L(A)$ is a free abelian group and implies that this sequence splits. Remark \ref{an}.(b) then shows $j'_A(\Gamma(A))=h_A(G)$, proving the claimed isomorphism.

\end{proof}

\begin{corollary}\label{cor:D8}
With the assumptions of Proposition \ref{prop:D3}, the underlying arithmetic matroid of $M_\GS$ is the dual to $(\ES,\rkE,\mm)$.
\end{corollary}

\begin{proof}
After Corollary \ref{cor:D7} we only have to show that $\mm(A)$ equals the cardinality of the torsion part of $M_{\GS}(\ES\setminus A)$, which is a direct consequence of Lemma \ref{lem:Wunabrep}.(ii).
\end{proof}

\begin{remark}
The map $b_X$ of Equation \eqref{eqqdefb} induces a bijection between $\oC{A}/G$ and $C(A)$. The (natural) group structure of $C(A)$ can be seen as additional data that can be extracted from $\GS$. Recent results in the topology of toric arrangements \cite[Example 7.3.2]{CaDe} show that this additional data has an algebraic-topological significance. 
%The next section focuses on another point of interest of the finite groups $C(A)$: namely, in the realizable case they appear naturally as torsion subgroups of the associated matroid over $\mathbb Z$.
\end{remark}

\section{Tutte polynomials of group actions}\label{sec:tutte}
In this section we study the Tutte polynomial associated to a group action on a semimatroid and, as an application, we extend to the generality of group actions on semimatroids (in particular, beyond the realizable case)
two important combinatorial interpretations of Tutte polynomials of toric arrangements.

%Before delving into the matter, recall

Recall our standard setup, e.g., from Section \ref{sec:defs}. We let $\GS$ denote the action of a group $G$ on a finitary semimatroid $\SS=(S,\CC,\rk)$. Write $\LL=\LL(\SS)$ for the geometric semilattice of flats of $\SS$, and let $\mathcal P _{\GS}$ denote the quotient poset of $\LL$ (see Definition \ref{def:LS}). Moreover, recall the set $\CC_{\GS}$ of orbits of the action on $\CC$ and the ``underlying'' locally ranked triple $\SS_{\GS}=(\ES,\underline{\CC},\underline{\rk})$
 
We will make use of  standard terminology about posets (see Section \ref{sec:FSGS} for a review).

\subsection{The characteristic polynomial of $\PS$}\label{ss:CP}
%For general cofinite $G$-semima\-troids, since the associated action on the semimatroids' geometric semilattice is
\begin{remark}
Since $G$ acts on $\LL$ 
by rank-preserving maps, the poset $\PS$ is ranked. With slight abuse of notation we will call $\rkE$ the rank function on $\mathcal P_{\GS}$, given by 
$$\rkE(p) := \rkE(x_p)\quad\quad\textrm{if } p=Gx_p.$$

% (by the rank
%function $\rkE$). 
\end{remark}

 We can thus define
the {\em characteristic polynomial of $\PS$} (e.g., following \cite[\S 3.10]{Sta}) as
$$
\chi_{\GS}(t):= \sum_{p\in \PS} \mu_{\PS}(\hat 0, p)t^{r - \rkE(p)}, 
$$
where $r$ is the rank of $\SS_\GS$ and $\mu_{\GS}$ is the M\"obius function of $\PS$ (notice that $\PS$
has a unique minimal element corresponding to the empty subset of
$\ES$).

\begin{lemma}\label{lem:GeLa}
  Let $\GS$ be \WT. Then, for every $x \in \LL$, the intervals
  $[\hat 0, Gx]$ in $\PS$ and $[\hat 0,x]$ in $\LL$ are
  poset-isomorphic. In particular, intervals in $\PS$ are geometric lattices.
\end{lemma}
\begin{proof}
Choose $x_p\in \LL$, set $p:=Gx_p\in \PS$ and consider any $q\in [\hat{0},p]$. Since $q\leq_{\PS} p$, by definition there is $x_q\in q$ with
  $x_q\leq_{\LL} x_p$. 
  
  Every other $x'_q \in q$ with $x'_{q}\leq_{\LL} x_p$ has the form $x'_q=gx_q$ for some $g\in G$. Thus, for every atom $x_a$ of $\LL$
  with $x_a\leq_\LL x_q\leq_\LL x_p$,  $gx_a \leq_\LL x_p$. In particular, for
  every  $s\in x_a$, $\{s,gs\}\in
  \mathcal C$ and  by \WTy $\rk \{s,gs\} =1$. Thus $gx_a
  \subseteq \cl_{\CC} x_a = x_a$ and, by symmetry, $x_a=gx_a$. This is true
  for all atoms $x_a \leq_{\LL} x_p$ and hence, because the interval $[\hat 0,x_p]$ is
  atomic, we have $x_q=x_q'$.
  
%  For every $q\leq_{\PS} p=Gx_p$, by definition there is $x_q\in q$ with
%  $x_q\leq_{\LL} x_p$. Any other such $x'_q \in q$ must satisfy
%  $x'_q=gx_q$ for some $g\in G$, thus for every atom $x_a$ of $\LL$
%  with $x_a\leq_\LL x_q\leq_\LL x_p$,  $gx_a \leq_\LL x_p$. In particular, for
%  every  $s\in x_a$, $\{s,gs\}\in
%  \mathcal C$ and  by \WTy $\rk \{s,gs\} =1$. Thus $gx_a
%  \subseteq \cl x_a = x_a$ and, by symmetry, $x_a=gx_a$. This is true
%  for all atoms $x_a$ and hence, because the interval $[\hat 0,x_p]$ is
%  atomic, we have $x_q=x_q'$.

  Therefore the mapping
  $$
  [\hat 0 , p]_{\PS} \to [\hat 0, x_p]_\LL, \,\, q\mapsto x_q
  $$
  is well-defined and order preserving. So is clearly its inverse
  $$
   [\hat 0, x_p]_\LL\to[\hat 0 , p]_{\PS},\,\, x\mapsto Gx
  $$
  and thus the two intervals are poset-isomorphic.
\end{proof}

\begin{proof}[Proof of Theorem \ref{thm:CP}] 
Let us first consider some $p\in \PS$ with $p>\hat 0$.
By Hall's theorem \cite[Proposition 3.8.5]{Sta} 
 the number $\mu_{\PS}(\hat
  0, p)$ is the reduced Euler characteristics of the ``open interval'' $[\hat
  0, p]\setminus \{\hat 0 , p\}$. 
  
 % Let $A(p)$ denote the set of atoms of the poset $[\hat 0,p]$.
By Lemma \ref{lem:GeLa}, the interval 
$[\hat 0,p]$ 
is a geometric lattice with set of atoms $A(p)$, and thus it induces a matroid structure on the set $\cup A(p) \subseteq \ES$ (with rank function $\underline{\rk}$). Let $\cl_p$ denote the associated closure operator.

  %with set of atoms, say, , 
  %and thus its
  Following \cite{Yuz}, the reduced Euler characteristics of $[\hat 0,p]$  can be
  computed by means of the {\em atomic complex}: this is the
  simplicial complex on the vertex set $A(p)$ and with set of simplices
$\Delta_p=\{B\subseteq A(p) \mid \vee B < p\}$. We obtain
$$
\mu_{\PS}(\hat 0, p)= \sum_{A\in \Delta_p} (-1)^{\vert A \vert -1}= \sum_{A\in D_p} (-1)^{\vert A \vert},
$$
where $D_p:=\{A\subseteq A(p) \mid \vee A=p\}$
and the second equality is derived from the
boolean identity $\sum_{A\subseteq A(p)} (-1)^{\vert A \vert} = 0$. 
Moreover, setting $$\widetilde{D}_p:=\{\widetilde{A} \subseteq \ES \mid \cl_{p}(\widetilde{A}) = p\}$$ 
%an easy computation (using the fact that $\SS_{\GS}$ has no loops) shows
and using the fact that $\SS$ loopless implies $\SS_{\GS}$ loopless, we can compute
\def\WD{\widetilde{D}}
\def\WA{\widetilde{A}}
\begin{align*}
\sum_{\widetilde{A}\in \widetilde{D}_p} (-1)^{\vert \widetilde{A} \vert} &=
\sum_{A\in D_p} \sum_{
\substack{\WA=\coprod_{a\in A} X_a \\ {\cl_p}(X_a)=a}} 
(-1)^{\vert \WA \vert} \\
&=\sum_{A\in D_p}\prod_{a\in A} \left[\sum_{\emptyset\neq X_a\subseteq a} (-1)^{\vert X_a\vert}\right] =
\sum_{A\in D_p} (-1)^{\vert A \vert} = \mu_{\PS}(\hat 0,p).
\end{align*}
Notice that the equality $\sum_{\widetilde{A}\in \widetilde{D}_p} (-1)^{\vert \widetilde{A} \vert}=\mu_{\PS}(\hat 0,p)$, which we just proved for $p>\hat 0$, holds trivially for $p=\hat 0$. Moreover, 
$\WA\in \WD_p$ implies in particular $\rkE(\WA)=\rkE(p)$. We can rewrite
\begin{align*}
\chi_{\GS}(t)&=\sum_{p\in \PS}\mu_{\PS}(\hat 0, p)t^{r-\rkE(p)}=\sum_{p\in \PS}\sum_{\WA\in \WD_p} (-1)^{\vert \WA \vert} t^{r- \rkE(p)}\\
&= \sum_{\WA\in \CS}(-1)^{\vert \WA \vert}\sum_{p\in P_{\WA}} t^{r-\rkE(\WA)} 
\end{align*}
where for every $\WA\in \CS$ we let 
$$P_{\WA}:=\{p\in \PS \mid \WA\in \WD_p\}=\oC{\WA}/G,$$ 
which is a set with exactly $\mm(\WA)$
elements (see Definition \ref{def:mm}). 
Thus,
\begin{align*}
\chi_{\GS}(t)&= \sum_{A\in \CS}
(-1)^{\vert A \vert} \mm(A) t^{r-\rkE(A)}\\
&= (-1)^r \sum_{A\in \CS}
\mm(A)(-1)^{\vert A \vert - \rk(A)} (-t)^{r-\rkE(A)}\\
&= (-1)^r T_{\GS}(1-t,0)
\end{align*}
where, as above, $r$ denotes the rank of $\SS_{\GS}$.
\end{proof}

\subsection{The corank-nullity polynomial of $\CC_{\GS}$}\label{ss:CN}
 The corank-nullity polynomial of the poset $\CC_{\GS}$ is
$$
s(\CC_\GS; u,v) = \sum_{GX \in \CC_{\GS}} u^{(r-\rk(X))} v^{(\vert X \vert - \rk(X))}.
$$

\begin{proposition}
If $\GS$ is \ST, 
$$
T_{\GS}(x,y) = s(\CC_{\GS}; x-1, y-1). 
$$
\end{proposition}
\begin{proof}
When $\GS$ is \ST, for every $X\in \CC$ we have $\vert X \vert = \vert \underline{X} \vert$. Moreover, by Corollary \ref{cor:rank}, $\rk(X)=\rkE(\underline{X})$. Then,
$$
s(\CC_\GS; u,v) = \sum_{GX \in \CC_{\GS}} u^{(r-\rk(X))} v^{(\vert X \vert - \rk(X))}
=\sum_{A\in \underline{\CC}} \,\,\sum_{\substack{GX\in \CC_{\GS}\\ \underline{X}=A}}
u^{(r-\rkE(A))} v^{(\vert A \vert - \rkE(A))}
$$
and the claim follows by setting $u= x-1$ and $v=y-1$.
\end{proof}

\subsection{Activities} We now turn to a generalization and new
combinatorial interpretation of the basis-activity decomposition of
arithmetic Tutte polynomials as defined in \cite{BM}. 
%Consider an
%almost arithmetic $G$-semimatroid $\GS$ and fix a total
%ordering of $E_\GS$.

\begin{remark}
Since we will not need details here, but only the statement of the next lemma, we refer to Ardila \cite{Ard} for the definition of internal and
external activity of bases of a finite semimatroid.
\end{remark}
 %If $\BB$ is the
%set of bases of a finite semimatroid $\SS$, we denote by
%$I(B)$ and $E(B)$ the sets of internally, resp.\ externally active
%elements of any $B\in \BB$, and write $R_B:=B\setminus I(B)$.

\begin{lemma}[Proposition 9.11 of \cite{Ard}]\label{lemmardila}
Let $\SS=(S,\CC,\rk)$ is a finite semimatroid with set of bases $\BB$ and let a total ordering of $S$ be fixed. For every basis $B\in \BB$ let $E(B)$, resp.\ $I(B)$, denote the set of externally, resp.\ intenally active elements with respect to $B$  and write $R_B:=B\setminus I(B)$. Then, %for every basis $B$, 
$(R_B, I(B),E(B))$ is a molecule, and 
  $$
  \CC = \biguplus_{B\in \BB} [R_B, B\cup E(B)]
  $$
\end{lemma}

We use this decomposition, which generalizes that for
matroids proved in \cite{Crapo}, in order to rewrite the sum in Definition \ref{def:Gtutte} as a sum over all bases.

\begin{Mthm}\label{thm:craponew}
  Let $\GS$ be an almost-arithmetic $G$-semimatroid such that $\SS_{\GS}$ is a semimatroid. Let $\BB_{\GS}$ denote the set of bases of $\SS_{\GS}$ and fix a total ordering of $\ES$. For $B\in \BB_{\GS}$ let $E(B)$, resp.\ $I(B)$ denote the set of externally, resp.\ internally active elements with respect to $B$, and write $R_B:= B\setminus I(B)$. Then
$$
T_\GS(x,y) = \sum_{B\in \BB_{\GS}} \left(\sum_{p\in \mathcal{Z}(B)} x^{\iota(p)}\right)
\left(\sum_{c\in \oC{R_B}/G} y^{\eta_{E(B)}(c)}\right)
$$
where 
 \begin{enumerate}
   \item[$\eta_{E(B)}(c)$] is the number of $e\in E(B)$ with $e\leq \kappa_{\GS}(c)$ in  $\CC_\GS$ (Definition \ref{def:eta}), 
   \item[$\mathcal{Z}(B)$] denotes the set $\mathcal{X}(I(B))$ associated to the molecule $(R_B,I(B),\emptyset)$ in Definition \ref{def:iota} and, accordingly,
   \item[$\iota(p)$] is the number defined in Definition \ref{def:iota}.
\end{enumerate}
In particular, the theorem holds when $\GS$ is centered, in which case it extends \cite[Theorem 6.3]{dAM} to the nonrealizable (and non-arithmetic) case. 
\end{Mthm}

\begin{proof}
  First, using Lemma \ref{lemmardila} we rewrite
 $$
 T_\GS(x,y)=\sum_{B\in \BB} \sum_{A\in \mu(B)} 
\mm (A) (x-1)^{\rk(\SS_{\GS})-\rkE(A)}(y-1)^{\vert A \vert -
    \rkE(A)}
 $$
and then, using  \cite[Lemma 4.3]{BM} (whose proof only uses axiom (A2)) we obtain
$$
 T_\GS(x,y)=$$
$$\sum_{B\in \BB} 
\left(
\sum_{F\subseteq I(B)} \frac{\rho(R_B,R_B \cup (I(B)\setminus
  F))}{m(R_B)} x^{\vert F \vert}
\right)\!
\left(
\sum_{T\subseteq E(B)}\rho(R_B\cup T, R_B\cup E(B)) y^{\vert T \vert}
\right).
$$
Here, in every summand the right-hand side factor is ready to be treated with
Proposition \ref{cor:eta} applied to the molecule $(R_B,\emptyset, E(B))$, while  the left-hand side
factor equals the claimed polynomial by Proposition \ref{prop:iota} applied to the molecule $(R_B,I(B),\emptyset)$.
\end{proof}

\subsection{Deletion-contraction recursion}\label{ss:TG}
We have seen (Section \ref{sec:defs}) that the matroid operations of contraction and deletion extend in a natural way to the context of $G$-semimatroids. In this section we study these operations, showing that the Tutte polynomial of a \ST action decomposes as a weighted sum of the polynomial of any single-element contraction and that of the corresponding deletion.

Recall the definitions and notations from Section \ref{ssec:FSM} and Section \ref{sec:defs}. In the following, given a locally ranked triple $\SS$ we will write $\CC(\SS)$ for its associated simplicial complex (the triple's ``second component''). %simplicial complex.

\begin{lemma}\label{lem:contraction} 
Let $\GS: G\circlearrowright (S,\CC,\rk)$ be a \WT $G$-semimatroid, and fix $e\in \ES$. Then,
\begin{itemize}
\item[(1)] there is a surjection $\phi: \CC(\SS_{\GS / e}) \to \CC(\SS_{\GS}/ {e })$ with \hspace{3cm} \linebreak$\rk_{\GS}(\phi(A)\cup e) - \rk_{\GS}(e)=\rk_{\GS/e}(A)$ which, if the action is \ST, also satisfies $\vert \phi (A) \vert = \vert A\vert$;
\item[(2)] $\PP_{\GS/ e} = (\PP_{\GS})_{\geq e}.$
\end{itemize}
Moreover, 
\begin{itemize}
\item[(3)]
$\displaystyle{m_{\GS }(A \cup e) = \sum_{A'\in \phi^{-1}(A)} m_{\GS/e}(A').}$
\end{itemize}
\end{lemma}

\begin{proof} Let us choose a fixed representative $x_e\in e$. In order to prove (1), we 
start by recalling that, by definition, 
$$\CC_{(\GS / e)} = (\CC_{/ x_e})/\stab(x_e).$$ 
From now, throughout this proof, we write $H:=\stab(x_e)$.
 Recall also the natural order on $\CC_{\GS}$ (Remark \ref{rem:CSposet}) and define %the following function.
$$
\widetilde\phi : \CC_{\GS / e} \to (\CC_{\GS})_{\geq e}, \quad H\{x_1,\ldots, x_k\} \mapsto G\{x_1,\ldots,x_k,x_e\}.
$$
The function $\widetilde \phi$ is a bijection, because the assignment 
$$
G\{x_1,\ldots,x_k,gx_e\} \mapsto H\{g^{-1}x_1,\ldots, g^{-1}x_k\}
$$
determines a well-defined function, inverse to $\widetilde \phi$. 

In order to prove (2) we notice that $\widetilde\phi$ commutes with
the relevant closure operators, i.e.,
$$
\widetilde\phi \circ \kappa_{\GS/e} = \kappa_{\GS} \circ \widetilde\phi.
$$
Bijectivity of $\widetilde \phi$ implies then that $\PP_{\GS/e} =
\kappa_{\GS} ((\CC_{\GS})_{\geq e})$, and the latter is easily
seen to equal $(\PP_{\GS} )_{\geq e}$. Thus, (2) holds.

Consider now the map
$$
\phi: \underline{\CC_{/x_e}} \to \underline{\CC}_{/e} ; \{Hx_1,\ldots, Hx_k\} \mapsto \{Gx_1,\ldots, Gx_k\}
$$
and the following diagram
$$
\begin{CD}
\CC_{\GS /e} @>{\widetilde \phi}>> (\CC_{\GS})_{\geq e} \\
@V{\lfloor \cdot \rfloor}VV  @VV{\lfloor \cdot \rfloor \setminus \{e\}}V \\
\underline{\CC_{/x_e}} @>{\phi}>> \underline{\CC}_{/e}
\end{CD}
$$
where commutativity  is evident once we evaluate all maps on a specific argument as follows.
$$
\begin{tikzcd}
H\{x_1,\ldots,x_k\} \arrow[mapsto]{r} \arrow[mapsto]{d} & 
G\{x_1,\ldots,x_k,x_e\} \arrow[mapsto]{d} \\
\{Hx_1,\ldots,Hx_k\} \arrow[mapsto]{r} &
 \{Gx_1,\ldots,Gx_k\}
\end{tikzcd}
$$
Now, for every $A\in \underline{\CC}_{/e}$ the map $\widetilde\phi$ gives a bijection between the $\lfloor \cdot \rfloor \setminus \{e\}$-preimage of $A$ and the $\lfloor \cdot \rfloor $-preimage of $\phi^{-1}(A)$, which proves (3). 
Claim (1) follows by inspecting the definition of the rank and, for the claim about cardinality, by noticing that if $Hx_1 \neq Hx_2$ and $gx_1=x_2$ for some $g\in G$, then $\{x_1,gx_1\}\in \CC$ and by \STy $x_1=gx_1=x_2$, a contradiction.
\end{proof}

\begin{proposition}\label{prop:CoDe} Let $\GS$ denote a $G$-semimatroid and fix $e\in \ES$.
	If $\GS$ is \WT{} -- resp.\ \ST, normal, arithmetic --,
  then so are also $\GS / e$ and $\GS \setminus e$. Moreover, if $\GS$ is \WT and cofinite, then $\GS/e$ and $\GS\setminus e$ are also cofinite.
\end{proposition}

\begin{proof}
	The treatment of $\GS\setminus e$ is trivial: indeed, the same group acts on a smaller set of elements with the same constraints. We will thus examine the case $\GS/e$. Choose $x_e\in e$ and let $H:=\stab(x_e)$.
	\begin{itemize}
	\item[--] {\em $\GS$ \WT.} To check \WTy for $\GS/e$  consider some $y\in S_{/x_e}$ and suppose $\{y, hy\}\in \mathcal{C}_{/x_e}$ for some $h\in H$. This means by definition that $\{y,hy,x_e\}\in \CC$, thus $\{y,hy\}\in \CC$ and, by \WTy of $\GS$, we have $\rk_{\CC}(\{y,hy\})=\rk_{\CC}(\{y\})$. Now by (R3) we know
	$$\rk_{\CC}(\{y\}) + \rk_{\CC}(\{y,hy,x_e\}) \leq \rk_{\CC}(\{y,x_e\}) + \rk_{\CC}(\{y,hy\}).$$
	By subtracting $\rk_{\CC}(\{y\})$ from both sides we obtain the inequality \quad \linebreak$\rk_{\CC}(\{y,hy,x_e\}) \leq \rk_{\CC}(\{y,x_e\})$ and, by (R2), $\rk_{\CC}(\{y,hy,x_e\}) = \rk_{\CC}(\{y,x_e\})$. We are now left with computing
\begin{align*}
	\rk_{\CC/x_e}(\{y,hy\}) &\stackrel{\textrm{def.}}{=}  \rk_{\CC}(\{y,hy,x_e\}) - \rk_{\CC}(\{x_e\})\\ 
	&=\rk_{\CC}(\{y,x_e\}) - \rk_{\CC}(\{x_e\}) \stackrel{\textrm{def.}}{=}  \rk_{\CC/x_e}(\{y\})
	\end{align*}
	\item[--] {\em $\GS$ \ST.} As above, consider some $y\in S_{/x_e}$ and suppose $\{y, hy\}\in \mathcal{C}_{/x_e}$ for some $h\in H$. This means that $\{y,hy,x_e\}\in \CC$, thus $\{y,hy\}\in \CC$ and, by \STy of $\GS$, $hy=y$ as required.
	\item[--] {\em $\GS$ normal.} Let $X \in \CC_{/x_e}$ then $\stab_H(X) = \stab_G(X)\cap H$ is normal in $G$ because it is the intersection of two normal subgroups. {\em A fortiori} it is normal in $H$.
	\item[--] {\em $\GS$ arithmetic.} Let $X =\{x_1,\ldots,x_k\}\in \CC_{/x_e}$. For all $i$ there is a natural group homomorphism
	$$\omega_i: \Gamma_{/e}(x_i) = H /\stab_H(x_i) \hookrightarrow G/\stab_G(x_i)=\Gamma(x_i)$$
	and these induce {a} natural group homomorphism
	$$
	\omega : \Gamma_{/e}^X \to \Gamma^{X\cup x_e}, \quad (\gamma_1,\ldots,\gamma_k) \mapsto 
	(\id,\omega_1(\gamma_1),\ldots,\omega_k(\gamma_k)).
	$$ 
	Now consider $\gamma,\gamma'\in W_{/e}(X)$. Then clearly $\omega(\gamma),\omega(\gamma')\in W(X\cup x_e)$ and, by arithmeticity of $\GS$,  
	$$\quad\omega(\gamma)\omega(\gamma')=(\id,\omega_1(\gamma_1)\omega_1(\gamma'_1),\ldots )=(\id,\omega_1(\gamma_1\gamma'_1),\ldots )\in W(X\cup x_e).$$ 
	Now, this means that $\omega(\gamma\gamma').(X\cup x_e)=\gamma\gamma'.X\cup\{x_e\}\in \CC$, hence \linebreak$\gamma\gamma'.X\in \CC_{/x_e}$ thus by definition $\gamma\gamma'\in W_{/e}(X)$.
	\item[--] {\em $\GS$ (\WT and) cofinite.} Cofiniteness of $\GS\setminus e$ is trivial, and that of $\GS/e$ is a consequence of Lemma \ref{lem:contraction}.(3).
	\end{itemize}
\end{proof}

We can now state and prove the desired recursion for Tutte polynomials of \ST $G$-semimatroids, generalizing the corresponding result of \cite{BM} for the arithmetic and centered case.

\begin{proof}[Proof of Theorem \ref{thm:MainCD}]
 In this proof for greater clarity we will write $\rk_{\GS}$, resp.\ $\rk_{\GS/e}$ for the rank functions of $\SS_{\GS}$, resp.\ $\SS_{\GS/e}$ (in particular, $\rk_{\GS}$ corresponds to what we called $\rkE$ previously).
 
 We follow \cite[Proposition 8.2]{Ard}, where the analogous results for semimatroids are proved, and start by rewriting the definition.
\begin{align*}
T_\GS(x,y):&= \sum_{A\in \CS} \mm (A) (x-1)^{r(\SS_{\GS})-\rk_{\GS}(A)}(y-1)^{\vert A \vert - \rk_{\GS}(A)}\\
&= \sum_{
\underbrace{{\scriptstyle{ A\in \CS,\,\,\, e\not\in A}}}_{\textrm{i.e., }A \in \CS_{\setminus e} = \,\CC(\SS_{\GS
  \setminus e})}
} \mm (A) (x-1)^{r(\SS_{\GS})-\rk_{\GS}(A)}(y-1)^{\vert A \vert - \rk_{\GS}(A)}\\
 &\quad+ \sum_{A\cup e \in \CS} \mm (A\cup e)
 (x-1)^{r(\SS_{\GS})-\rk_{\GS}(A\cup e)}(y-1)^{\vert A \cup e
   \vert - \rk_{\GS}(A\cup e)}
\\
\end{align*}
The second summand can be rewritten as follows by Lemma \ref{lem:contraction}.
 $$
\underbrace{\sum_{A\in \CS_{/e}} \sum_{A'\in \phi^{-1}(A)}}_{A'\in \CC(\SS_{\GS/e})}
m_{\GS/e} (A')
 (x-1)^{r(\SS_{\GS/e})-\rk_{\GS/e}(A')}(y-1)^{\vert A'
   \vert +1 - \rk_{\GS/e}(A') - \rk_{\GS}(e)}
$$
If $e$ is neither a loop nor an isthmus, by Remark \ref{lem:deletion}
and  Lemma \ref{lem:contraction} 
we have $\rk(\SS_{\GS}) = \rk(\SS_{\GS \setminus e})$ and
$\rk_{\GS}(e)=1$, thus
the two summands are exactly $T_{\GS \setminus e} (x,y)$ and $T_{\GS /
  e}(x,y)$, respectively. If $e$ isn isthmus, $\rk(\SS_{\GS}) =
\rk(\SS_{\GS \setminus e})-1$ (and $\rk_{\GS}(e)=1$) and thus 
we have $T_{\GS}(x,y) = (x-1)T_{\GS \setminus e} (x,y) + T_{\GS /e}
(x,y) $. 
Finally, when $e$ is a loop we have $\rk_{\GS}(e)=0$ (but still
$\rk(\SS_{\GS}) = \rk(\SS_{\GS \setminus e})$) and we easily get the
claimed identity.
\end{proof}

\appendix
\section{An algebraic lemma}
We give the proof of the following auxiliary lemma for completeness'sake and in order not to clutter the exposition in the main text.

\begin{lemma}\label{lem:backdoor}
Consider the following commutative diagram of abelian groups with exact rows and where the arrows $\twoheadrightarrow$ denote epimorphisms.
\begin{center}
\begin{tikzcd}[column sep=.05in,row sep=.1in]
&&&&& B_0\arrow{rrr}\arrow{rdd}\arrow{ld} &   &   & C_0\arrow{ld}\arrow{rdd} &\\
&&&&B_1\arrow[crossing over, two heads]{rrr} & &   & C_1&   & \\
&&&A_2\arrow{rrr}\arrow[two heads]{dl}&&   & B_2\arrow[two heads]{rrr}\arrow[two heads]{dl} &   &   & C_2\arrow{dl}  \\
&&A_3\arrow{rrr}&&& B_3\arrow[two heads, from= uul, crossing over]\arrow[two heads]{rrr} &   &   & C_3\arrow[from= uul, crossing over]&  \\
\end{tikzcd}
%%%%%%%%%%%
%\begin{tikzcd}[column sep=.05in,row sep=.1in]
%&&&&& B_0\arrow{rrr}\arrow[two heads]{rdd}\arrow[two heads]{ld} &   &   & C_0\arrow{rrr}\arrow{ld}\arrow{rdd} & & & 0 & \\
%%
%&&&&B_1\arrow{rrr} & &   & C_1\arrow[crossing over]{rrr} &   & & 0 & & \\
%%
%&&&A_2\arrow{rrr}\arrow[two heads]{dl}&&   & B_2\arrow{rrr}\arrow[two heads]{dl} &   &   & C_2\arrow{rrr}\arrow{dl} & & & 0\\
%%
%&&A_3\arrow{rrr}&&& B_3\arrow[two heads, from= uul, crossing over]\arrow{rrr} &   &   & C_3\arrow[from= uul, crossing over]\arrow{rrr} & & & 0 & \\
%\end{tikzcd}
%%%%%%%%%%%
\end{center}
If the square of the $B_i$ is a pushout square, then so is the square of the $C_i$.
\end{lemma}
\begin{proof}
We name the arrows in the diagram as below and we verify the pushout property by considering a co-cone of the diagram spanned by $C_0,C_1, C_2$, which consists of a group $H$ and two arrows $h_1,h_2$ such that $h_1\circ \hat{c}_1 = h_2\circ\hat{c}_2$. One verifies that the group $H$ with the morphisms 
$\hat{h}_1:= h_1 \circ \epsilon_1$, $\hat{h}_2:= h_2 \circ \epsilon_2$ defines a co-cone on the diagram spanned by $B_0,B_1,B_2$. Since by assumption the $B_i$ span a pushout square, there is a unique arrow $\varphi$ with
$$
\varphi \circ b_1 = \hat{h}_1 =h_1 \circ \epsilon_1, \quad
\varphi \circ b_2 = \hat{h}_2 =h_2 \circ \epsilon_2 
$$
\begin{center}
\begin{tikzcd}[column sep=.3in,row sep=.2in]
&&& B_0\arrow{rrrr}\arrow{rdd}\arrow{ld} &  & &   & C_0\arrow[ld, swap, "\hat{c}_1"]\arrow{rdd}{\hat{c}_2} &\\
&&B_1\arrow[crossing over, two heads]{rrrr}[near end]{\epsilon_1} & &&   & C_1&   & \\
&A_2\arrow{rrr}[near start]{j_2}\arrow[two heads]{dl}{a}&&   & B_2\arrow[two heads]{rrrr}[near start]{\epsilon_2}\arrow[two heads]{dl}{b_2} & &  &   & C_2\arrow{dl}{c_2}  \\
%
%&&&&&&&&&& \\
%
A_3\arrow{rrr}{j_3}&&& B_3\arrow[two heads, from= uul, crossing over, swap, pos=0.7, "b_1"]\arrow[two heads]{rrrr}{\epsilon_3} &  & &   & C_3\arrow[from= uul, crossing over, swap, pos=0.7, "c_1"]&  \\
%&&&&&&&&& \\
&&&&&&&&&& \\
&&&&&&& H
\arrow[from=uulllllll, dotted, bend right =10, swap, "0"]
\arrow[from=uullll, dotted, swap, pos=.4, "\exists ! \varphi"]
\arrow[from=uuuul, crossing over, dashed, bend right, swap, pos=0.7, "h_1"]
\arrow[from=uuur, crossing over, dashed, bend left, "h_2"] & \\
\end{tikzcd}
%%%%%%%%%%%%%%%%%%%%
%\begin{tikzcd}[column sep=.05in,row sep=.1in]
%&&&&&& B_0\arrow{rrr}\arrow[two heads]{rdd}\arrow[two heads]{ld} &   &   & C_0\arrow{rrr}\arrow{ld}\arrow{rdd} & & & 0 & \\
%%
%&&&&&B_1\arrow{rrr} & &   & C_1\arrow[crossing over]{rrr} &   & & 0 & & \\
%%
%&0\arrow{rrr}&&&A_2\arrow{rrr}&&   & B_2\arrow{rrr}\arrow[two heads]{dl} &   &   & C_2\arrow{rrr}\arrow{dl}{c_2} & & & 0\\
%%
%0\arrow{rrr}&&&A_3\arrow{rrr}&&& B_3\arrow[two heads, from= uul, crossing over]\arrow{rrr} &   &   & C_3\arrow[from= uul, crossing over, swap, "c_1"]\arrow{rrr} & & & 0 & \\
%&&&&&&&&&&&& \\
%&&&&&&&&&&&& \\
%&&&&&&&&& H
%\arrow[from=uuuuul, crossing over, dashed, bend right]
%\arrow[from=uuuur, crossing over, dashed, bend left] &&& \\
%\end{tikzcd}
\end{center}
Notice that 
$$
\varphi \circ j_3 \circ a = \varphi \circ b_2 \circ j_2 = \hat{h}_2 \circ j_2 =
h_2 \circ \underbrace{\epsilon_2 \circ j_2}_{=0} = 0 = 0\circ a
$$
and, since $a$ is an epimorphism, by right cancellation we obtain
$$\varphi \circ j_3 = 0.$$
Exactness of the bottom row, by the universal property of cokernels, shows that there exist a unique $g$ with $g\circ \epsilon_3= \varphi$.

\noindent{\em Claim.} For every $g': C_3\to H$ and every $i=1,2$,
\begin{center}
 $g'\circ c_i=h_i$ is equivalent to $\epsilon_3 \circ g' = \varphi$.
\end{center}
\noindent{\em Proof.} By right cancellativity of epimorphisms, $g'\circ c_i=h_i$ is equivalent to  
$$g'\circ c_i\circ \epsilon_i =h_i\circ \epsilon_i.$$
By commutativity of the diagram, the left-hand side of this equation equals $g'\circ \epsilon_3\circ b_i$. By the definition of $\varphi$, the right-hand side equals $\varphi \circ b_i$. Again, by right-cancellativity of the epimorphism $b_i$ we obtain the claimed equivalence. \qed

Using the claim we see immediately that our $g$ satisfies $g\circ c_1=h_1$ and $g\circ c_2=h_2$. Moreover, for every $g'$ with the same commutativity properties the claim implies that $g'\circ\epsilon_3 = \varphi$, and by the uniqueness in the definition of $g$ we must have $g'=g$. This concludes the proof that the square of the $C_i$ is a pushout.
\end{proof}

\addtocontents{toc}{\protect\setcounter{tocdepth}{-1}}

\bibliographystyle{plain}
\bibliography{Bib}

\def\cprime{$'$}
\begin{thebibliography}{10}

\bibitem{MAG}
Marcelo Aguiar.
\newblock personal communication.

\bibitem{Ard}
Federico Ardila.
\newblock Semimatroids and their {T}utte polynomials.
\newblock {\em Rev. Colombiana Mat.}, 41(1):39--66, 2007.

\bibitem{BK}
Eric Babson and Dmitry~N. Kozlov.
\newblock Group actions on posets.
\newblock {\em J. Algebra}, 285(2):439--450, 2005.

\bibitem{Bibby}
C.~{Bibby}.
\newblock {Cohomology of abelian arrangements}.
\newblock {\em ArXiv e-prints}, October 2013.

\bibitem{BM}
Petter Br{\"a}nd{\'e}n and Luca Moci.
\newblock The multivariate arithmetic {T}utte polynomial.
\newblock {\em Trans. Amer. Math. Soc.}, 366(10):5523--5540, 2014.

\bibitem{CaDe}
F.~{Callegaro} and E.~{Delucchi}.
\newblock {The integer cohomology algebra of toric arrangements}.
\newblock {\em ArXiv e-prints}, June 2015.

\bibitem{PJC}
Peter~J. Cameron.
\newblock Cycle index, weight enumerator, and {T}utte polynomial.
\newblock {\em Electron. J. Combin.}, 9(1):Note 2, 10 pp. (electronic), 2002.

\bibitem{Crapo}
Henry~H. Crapo.
\newblock The {T}utte polynomial.
\newblock {\em Aequationes Math.}, 3:211--229, 1969.

\bibitem{dAM}
Michele D'Adderio and Luca Moci.
\newblock Arithmetic matroids, the {T}utte polynomial and toric arrangements.
\newblock {\em Adv. Math.}, 232:335--367, 2013.

\bibitem{DMG}
Michele D'Adderio and Luca Moci.
\newblock Graph colorings, flows and arithmetic {T}utte polynomial.
\newblock {\em J. Combin. Theory Ser. A}, 120(1):11--27, 2013.

\bibitem{dAD2}
Giacomo d'Antonio and Emanuele Delucchi.
\newblock {Minimality of toric arrangements}.
\newblock {\em To appear in Journal of the EMS}, December 2011.

\bibitem{dAD1}
Giacomo d'Antonio and Emanuele Delucchi.
\newblock A salvetti complex for toric arrangements and its fundamental group.
\newblock {\em Int. Math. Res. Not. IMRN}, 6:Art. ID rnr161, 32, 2011.

\bibitem{dCP1}
Corrado De~Concini and Claudio Procesi.
\newblock {On the geometry of toric arrangements}.
\newblock {\em Transformation Groups}, 10(3):387--422, 2005.

\bibitem{dCP2}
Corrado De~Concini and Claudio Procesi.
\newblock {\em {Topics in hyperplane arrangements, polytopes and box-splines}}.
\newblock Springer Verlag, 2010.

\bibitem{dCPV}
Corrado De~Concini, Claudio Procesi, and Mich\`ele Vergne.
\newblock Vector partition functions and index of transversally elliptic
  operators.
\newblock {\em Transform. Groups}, 15(4):775--811, 2010.

\bibitem{DK}
E.~{Delucchi} and K.~{Knauer}.
\newblock {Finitary affine oriented matroids: covectors and topological
  realization}.
\newblock {\em Preprint in preparation}.

\bibitem{MaKl}
Art~M. Duval, Caroline~J. Klivans, and Jeremy~L. Martin.
\newblock Cuts and flows of cell complexes.
\newblock {\em J. Algebraic Combin.}, 41(4):969--999, 2015.

\bibitem{ERS}
R.~Ehrenborg, M.~Readdy, and M.~Slone.
\newblock {Affine and toric hyperplane arrangements}.
\newblock {\em Discrete and Computational Geometry}, 41(4):481--512, 2009.

\bibitem{FM}
A.~{Fink} and L.~{Moci}.
\newblock {Matroids over a ring}.
\newblock {\em ArXiv e-prints}, March 2015.

\bibitem{FiPisa}
Alex Fink.
\newblock Polytopes and moduli of matroids over rings, February 2015.
\newblock Talk at the session ``Algebraic topology, geometric and combinatorial
  group theory" of the program ``Perspectives in Lie Theory'', Centro De
  Giorgi, Pisa.

\bibitem{Gru}
Branko Gr{\"u}nbaum.
\newblock {\em Configurations of points and lines}, volume 103 of {\em Graduate
  Studies in Mathematics}.
\newblock American Mathematical Society, Providence, RI, 2009.

\bibitem{KTT}
Hidehiko Kamiya, Akimichi Takemura, and Hiroaki Terao.
\newblock Periodicity of hyperplane arrangements with integral coefficients
  modulo positive integers.
\newblock {\em J. Algebraic Combin.}, 27(3):317--330, 2008.

\bibitem{KTT2}
Hidehiko Kamiya, Akimichi Takemura, and Hiroaki Terao.
\newblock Periodicity of non-central integral arrangements modulo positive
  integers.
\newblock {\em Ann. Comb.}, 15(3):449--464, 2011.

\bibitem{Kawa}
Yukihito Kawahara.
\newblock On matroids and {O}rlik-{S}olomon algebras.
\newblock {\em Ann. Comb.}, 8(1):63--80, 2004.

\bibitem{Law}
Jim Lawrence.
\newblock Enumeration in torus arrangements.
\newblock {\em European J. Combin.}, 32(6):870--881, 2011.

\bibitem{MSZ}
J.~P. {May}, M.~{Stephan}, and I.~{Zakharevich}.
\newblock {The homotopy theory of equivariant posets}.
\newblock {\em ArXiv e-prints}, January 2016.

\bibitem{MM}
Peter McMullen.
\newblock Volumes of projections of unit cubes.
\newblock {\em Bull. London Math. Soc.}, 16(3):278--280, 1984.

\bibitem{phi}
Luca Moci.
\newblock Combinatorics and topology of toric arrangements defined by root
  systems.
\newblock {\em Atti Accad. Naz. Lincei Cl. Sci. Fis. Mat. Natur. Rend. Lincei
  (9) Mat. Appl.}, 19(4):293--308, 2008.

\bibitem{Moc1}
Luca Moci.
\newblock A {T}utte polynomial for toric arrangements.
\newblock {\em Trans. Amer. Math. Soc.}, 364(2):1067--1088, 2012.

\bibitem{OT}
Peter Orlik and Hiroaki Terao.
\newblock {\em Arrangements of hyperplanes}, volume 300 of {\em Grundlehren der
  Mathematischen Wissenschaften [Fundamental Principles of Mathematical
  Sciences]}.
\newblock Springer-Verlag, Berlin, 1992.

\bibitem{Oxl}
James Oxley.
\newblock {\em Matroid theory}, volume~21 of {\em Oxford Graduate Texts in
  Mathematics}.
\newblock Oxford University Press, Oxford, second edition, 2011.

\bibitem{Polya}
G.~P{\'o}lya.
\newblock Kombinatorische {A}nzahlbestimmungen f\"ur {G}ruppen, {G}raphen und
  chemische {V}erbindungen.
\newblock {\em Acta Math.}, 68(1):145--254, 1937.

\bibitem{StaG}
Richard~P. Stanley.
\newblock Some aspects of groups acting on finite posets.
\newblock {\em J. Combin. Theory Ser. A}, 32(2):132--161, 1982.

\bibitem{Sta}
Richard~P. Stanley.
\newblock {\em Enumerative Combinatorics, vol. 1}.
\newblock Cambridge University Presss, Cambridge, second edition, 1986.

\bibitem{StAC}
Richard~P. Stanley.
\newblock {\em Algebraic combinatorics}.
\newblock Undergraduate Texts in Mathematics. Springer, New York, 2013.
\newblock Walks, trees, tableaux, and more.

\bibitem{ThWe}
J.~Th{\'e}venaz and P.~J. Webb.
\newblock Homotopy equivalence of posets with a group action.
\newblock {\em J. Combin. Theory Ser. A}, 56(2):173--181, 1991.

\bibitem{WW}
Michelle Wachs and James Walker.
\newblock On geometric semilattices.
\newblock {\em Order \textbf{2}}, pages 367--385, 1986.

\bibitem{Welsh}
D.~J.~A. Welsh.
\newblock {\em Matroid theory}.
\newblock Academic Press [Harcourt Brace Jovanovich Publishers], London, 1976.
\newblock L. M. S. Monographs, No. 8.

\bibitem{Yuz}
Sergey Yuzvinsky.
\newblock Rational model of subspace complement on atomic complex.
\newblock {\em Publ. Inst. Math. (Beograd) (N.S.)}, 66(80):157--164, 1999.
\newblock Geometric combinatorics (Kotor, 1998).

\bibitem{ZIM}
S.~M.~A. Zaidi, M.~Irfan, and Gulam Muhiuddin.
\newblock On the category of {$G$}-sets. {I}.
\newblock {\em Int. Math. Forum}, 4(5-8):383--393, 2009.

\end{thebibliography}
\end{document}